\documentclass[article]{amsart}

\usepackage{graphicx,mathrsfs,amssymb,amsmath,color,sansmath}
\usepackage{hyperref}

\usepackage{cmap} 

\newtheorem{theorem}{Theorem}[section]
\newtheorem{lemma}[theorem]{Lemma}
\newtheorem{corollary}[theorem]{Corollary}
\newtheorem{proposition}[theorem]{Proposition}
\newtheorem{remark}[theorem]{Remark}
\newtheorem{definition}[theorem]{Definition}

\numberwithin{equation}{section}

\newcommand{\cz}{{\mathbb C}}
\newcommand{\gz}{{\mathbb Z}}
\newcommand{\nz}{{\mathbb N}}
\newcommand{\rz}{{\mathbb R}}
\newcommand{\sz}{{\mathbb S}}

\newcommand{\bfB}{\mathbf{B}}
\newcommand{\bfg}{\mathbf{g}}
\newcommand{\bfh}{\mathbf{h}}

\newcommand{\bfkappa}{\boldsymbol{\kappa}}
\newcommand{\bflambda}{\boldsymbol{\lambda}}

\newcommand{\bfp}{\mathbf{p}}
\newcommand{\bfq}{\mathbf{q}}
\newcommand{\bfr}{\mathbf{r}}
\newcommand{\bfs}{\mathbf{s}}
\newcommand{\bfS}{\mathbf{S}}
\newcommand{\bfx}{\mathbf{x}}

\newcommand{\wtbfB}{\widetilde{\mathbf{B}}}
\newcommand{\wtbfS}{\mathbf{\widetilde{S}}}

\newcommand{\calA}{\mathcal{A}}

\newcommand{\calU}{\mathcal{U}}
\newcommand{\calW}{\mathcal{W}}

\newcommand{\frakg}{\mathfrak{g}}
\newcommand{\frakG}{\mathfrak{G}}
\newcommand{\frakB}{\mathfrak{B}}
\newcommand{\frakS}{\mathfrak{S}}
\newcommand{\frakX}{\mathfrak{X}}

\newcommand{\scrC}{\mathscr{C}}
\newcommand{\scrD}{\mathscr{D}}
\newcommand{\scrF}{\mathscr{F}}
\newcommand{\scrL}{\mathscr{L}}
\newcommand{\scrS}{\mathscr{S}}

\newcommand{\dbar}{d\hspace*{-0.08em}\bar{}\hspace*{0.1em}}
\newcommand{\eps}{\varepsilon}
\newcommand{\forget}[1]{}
\newcommand{\Hom}{\mathrm{hom}}
\newcommand{\lra}{\longrightarrow}
\newcommand{\op}{{\mathop{\mathrm{op}}}}

\newcommand{\rpbar}{\overline{\rz}_+}
\newcommand{\spk}[1]{\langle#1\rangle}
\newcommand{\schnitt}{\mathop{\mbox{\Large$\cap$}}}
\newcommand{\tr}{\mathrm{tr}}
\newcommand{\trinorm}[1]{|\hspace*{-1pt}|\hspace*{-1pt}|#1|\hspace*{-1pt}|\hspace*{-1pt}|}
\newcommand{\wh}{\widehat}
\newcommand{\whsz}{\widehat{\mathbb S}}
\newcommand{\wt}{\widetilde}

\begin{document}

\title[Parametric bvp with global projection conditions]{
Calculus for parametric boundary problems\\ with global projection  conditions
}
\author{J\"org Seiler}
\address{Dipartimento di Matematica, Universit\`{a} di Torino, Italy}
\email{joerg.seiler@unito.it}


\begin{abstract}
A pseudodifferential calculus for parameter-dependent operators on smooth manifolds with boundary 
in the spirit of Boutet de Monvel's algebra is constructed. The calculus contains, in particular, the resolvents of realizations 
of differential operators subject to global projection boundary conditions (spectral boundary conditions are a particular 
example); resolvent trace asymptotics are easily derived. The calculus is related to but different from the calculi developed by Grubb 
and Grubb-Seeley. We use ideas from the theory of pseudodifferential operators on manifolds with edges due to 
Schulze, in particular the concept of operator-valued symbols twisted by a group-action. 
Parameter-ellipticity in the calculus is characterized by the invertibility of three principal symbols: the homogeneous principal 
symbol, the principal boundary symbol, and the so-called principal limit symbol. The principal boundary symbol has, in general, 
a singularity in the co-variable/parameter space, the principal limit symbol is a new ingredient of our calculus. 

\vspace*{5mm}


\end{abstract}

\vspace*{-20mm}

\maketitle
{\small
\setcounter{tocdepth}{3}
\tableofcontents
}
\section{Introduction} 

Parameter-dependent pseudodifferential operators have proven to be a very powerful tool in partial differential 
equations and in geometric analysis. One of the main goals in this approach is to obtain a precise 
description of the resolvent of certain (differential) operators and to use this structure in the study of the specific 
problem. Typical applications concern maximal regularity of evolution equations, the analysis of spectral and 
heat trace asymptotics, and complex powers of operators (just to name a few). This idea has been introduced in 
the 1960's, see for example the pioneering works of Seeley \cite{Seel67,Seel69,Seel69-2}
on boundary value problems, and has been refined over the following decades, in particular under the aspect  of 
``algebras'' or ``calculi'' of pseudodifferential operators. 

In \cite{Bout71}, Boutet de Monvel introdued a calculus of psudodifferential operators which contains all classical differential 
boundary value problems as well as the parametrices (i.e., inverses modulo smoothing remainders)  of Shapiro-Lopatinskij 
elliptic problems. Nowadays this calculus is referred to as Boutet de Monvel's algebra. The elements in this calculus 
are block-matrix operators of the form 
 $$\begin{pmatrix}  A_++G&K\\T&Q  \end{pmatrix}: 
     \begin{matrix}  \scrC^\infty(M,E_0)\\ \oplus\\   \scrC^\infty(\partial M,F_0)\end{matrix}
     \lra
     \begin{matrix}  \scrC^\infty(M,E_1)\\ \oplus\\   \scrC^\infty(\partial M,F_1)\end{matrix},
 $$
where $M$ is a smooth compact manifold with boundary, $E_j$ and $F_j$ are vector-bundles over $M$ and $\partial M$, 
respectively. The bundles $F_0$ or $F_1$ are also allowed to be zero-dimensional. $A_+$ is the restriction of a 
pseudodifferential operator $A$ from an ambient manifold of $M$ to $M$ itself, $G$ is a so-called singular Green operator, 
$T$ is a trace operator, $K$ is a Poisson or potential operator, $Q$ is a pseudodifferential operator on the boundary. Elements 
with $A_+=0$ we shall refer to as \emph{generalized singular Green operators}. 
Initially defined  on smooth functions, the operators extend by continuity to scales of Sobolev spaces (which are 
$L_2$-Bessel potential spaces in this paper; however, as shown in Grubb-Kokholm \cite{GrKo}, one can also consider the 
$L_p$-scale). 
For nice introductions to this calculus we refer the reader to Schrohe \cite{Schr01} and Schulze \cite{Schu-BVP}. 
The study of resolvents of differential operators subject to classical differential boundary conditions leads to a 
parameter-dependent version of Boutet de Monvel's algebra where the parameter enters in a ``strong'' way, i.e., as additional 
co-variable; the trick to consider a parameter as additional co-variable goes back to Agmon 
\cite{Agmo}. Details are given in Section \ref{sec:2.3} or can be found in Schrohe, Schulze \cite{ScSc1}, for example. 

Differential operators equipped with boundary conditions like spectral boundary conditions cannot be handled in 
Boutet's original calculus and not every elliptic differential operator can be supplemented with boundary conditions 
to become a Shapiro-Lopatinskii elliptic problem, since it must satisfy the so-called Atiyah-Bott 
condition, cf. \cite{AtBo}. Based on a systematic study of such effects by Nazaikinskij, Schulze, Shatalov, and Sternin in
\cite{NSSS98}, Schulze in \cite{Schu37,LiuSchulze} introduced an extended calculus that resolves this shortcoming.
 The elements of the extended calculus have the form 
 \begin{align}\label{eq:intro01}
     \begin{pmatrix}  1&0\\0&P_1 \end{pmatrix} 
     \begin{pmatrix}  A_++G&K\\T&Q  \end{pmatrix}
     \begin{pmatrix}  1&0\\0&P_0  \end{pmatrix}
 \end{align}
with pseudodifferential projections (i.e., zero order idempotents) $P_j$ on the boundary. The operator action refers to 
the scale of spaces where the full Sobolev spaces are substituted by closed subspaces determined by the projections. 
In this paper we extend this approach to parameter-dependent boundary problems; in this way, resolvents of elliptic 
differential operators subject to global projection boundary conditions can be treated. As it turns out, 
the calculus we develop also allows to obtain resolvent trace asymptotics and thus provides, in particular, a method 
to derive asymptotics of first order operators subject to APS boundary conditions due Grubb and Seeley \cite{GrSe,Grub02} 
and of second order operators due to Grubb \cite{Grub03}, see Sections \ref{sec:7.2} and \ref{sec:7.3}. 

Since the structure of the operators in \eqref{eq:intro01} reminds of classical Toeplitz operators on the unit circle, we shall 
also use the terminology of operators of \emph{Toeplitz type}. In \cite{Seil12} the author has considered such operators from 
a general point of view: given a calculus, does there exist an associated calculus of Toeplitz operators? As it turns out, 
if the calculus is closed under taking adjoints and if ellipticity in the calculus is characterized by the invertibility of certain 
principal symbols, then there is a canonical way to characterize ellipticity and existence of a parametrix for operators of 
Toeplitz type whenever the projections belong to the original calculus. We summarize these results in Section \ref{sec:6}. 
In this set-up, the calculus may also be a parameter-dependent one. 
For this reason, the main issue in this paper is actually to construct a pseudodifferential calculus which 
\begin{itemize}
 \item[i$)$] contains Boutet de Monvel's algebra of strongly parameter-dependent operators and pseudodifferential 
  operators on the boundary which do not depend on the parameter and where
 \item[ii$)$] parameter-ellipticity is characterized by the invertibility of certain principal symbols. 
\end{itemize}
Point i) has already been addressed by Grubb in \cite{Grub01}. It is remarkable that the resulting calculus is in appearance quite 
different from ours. The reason might be that \cite{Grub01} is focused on resolvent trace expansions while for us point ii) 
is central. 
Let us emphasize once more that, once such a calculus has been established, we can substitute 
the projections in \eqref{eq:intro01} by arbitrary projections from the calculus. 

The present paper uses ideas of the author's work \cite{Seil22-1} where the analogues of i) and ii) have been addressed for 
operators on manifolds without boundary, i.e., a calculus containing both strongly parameter-dependent and 
parameter-independent pseudodifferential operators on $\rz^n$ or on closed manifolds has been constructed. 
The paper \cite{Seil22-1} takes much 
inspiration from the concept of pseudodifferential operators with \emph{finite regularity} due to Grubb 
\cite{Grub}.  The regularity number indicates how far an operator is from being strongly parameter-dependent; 
parameter-independent zero-order operators (in particular projections) have regularity zero, strongly parameter-dependent 
operators have infinite regularity. The monograph \cite{Grub} introduces another Boutet de Monvel type calculus for 
boundary problems 
which allows the construction of resolvents for a wide class of operators (which includes certain nonlocal 
perturbations of differential operators) but, at the same time has some limitations, for example it only yields finite resolvent 
trace expansions. 
An important requirement for the notion of ellipticity in \cite{Grub} is that of positive regularity. The main finding of \cite{Seil22-1} 
is a new ``geometric'' interpretation of the regularity as a specific kind of singular structure (i.e., non-smoothness) of symbols in 
co-variable and parameter. This permits not only to eliminate the requirement of positive regularity in the boundaryless case 
but also leads to a hierarchy of principal symbols in the calculus characterizing the ellipticity. 

In the present work we shall find symbolic structures in the spirit of \cite{Seil22-1} for boundary value problems. 
The \emph{homogeneous principal symbol} and the \emph{principal boundary symbol} are already known from the classic 
Boutet de Monvel's calculus, though in our calculus the principal boundary symbol generally has a singularity which is absent  
in the classical setting. We shall introduce two further new principal symbols, called \emph{principal limit symbol} and 
\emph{principal angular symbol}. Let us mention that the ellipticity condition (III) in \cite[Definition 1.5.5]{Grub} can be 
interpreted in terms of the principal angular symbol. 

Let us touch upon another key aspect of this paper. For simplicity of presentation let us focus here on Poisson operators, 
similar observations hold for trace and singular Green operators. A Poisson operator on the half-space has the form 
 $$[\op(k)(\mu)u](x) =\int e^{-ix'\xi'} k(x',\xi',\mu;x_n)\wh u(\xi')\,\dbar\xi',$$
where $\wh u$ is the Fourier transform and the \emph{symbol-kernel} $k$ has the form 
 $$k(x',\xi',\mu;x_n)=[\xi',\mu]^{1/2}k'(x',\xi',\mu;[\xi',\mu]x_n)$$ 
where $[\cdot]$ is a smoothed norm function and $k'(x',\xi',\mu;t)$ is smooth and rapidly decreasing in $t\in\rpbar$; 
note the ``twisting'' between $x_n$ and the co-variable and parameter $(\xi',\mu)$. 
Denoting with $\kappa_\lambda$, $\lambda>0$, the \emph{dilation group-action} acting on functions by 
$(\kappa_\lambda u)(t)=\lambda^{1/2}u(\lambda t)$, $k$ can be written as 
 $$k(x',\xi',\mu;x_n)=(\kappa_{[\xi',\mu]} k')(x',\xi',\mu;x_n).$$ 
This leads to the fact that Poisson operators can be viewed as pseudodifferential operators  
with operator-valued symbols twisted by a group-action in the sense of Schulze (see the textbooks 
\cite{Schulze-North-Holland,Schulze-Akademie}, for example); we will recall this concept in Section \ref{sec:2.1.1}. 
The mapping $k\mapsto k'=\kappa^{-1}_{[\xi',\mu]}k$ provides a canonical passage from twisted to 
un-twisted symbols. It is important for our analysis that the operation of un-twisting persists on the 
\emph{operator-level}, i.e., we establish a one-to-one correspondence between twisted 
and un-twisted Poisson operators, see Section \ref{sec:4.4}.  

Twisted operator-valued symbols have been systematically employed by Schulze for the construction 
of calculi of pseudodifferential operators on manifolds with singularities, in particular on manifolds with \emph{edges}, 
see for instance \cite{ReSc2,Schulze-North-Holland}. 
The local model of a manifold with edge (also called manifold with fibered boundary) is a wedge 
$\rz^q\times X^\wedge$, where $X^\wedge$ is an infinite cone with base $X$, a closed smooth manifold. 
The half-space is a particular case, where $X$ is simply a point. The so-called edge algebra of Schulze is a calculus that, 
on a formal level, has many similarities with Boutet de Monvel's algebra; there also exists a version with global projection 
edge-conditions (without parameter), see \cite{ScSe2004,ScSe2006,LiuSchulze}. The parameter-dependent edge algebra 
contains operators with symbol kernels, which have a twisted structure similar to the one described above. 
It seems plausible that the operation of un-twisting extends to this more general setting and thus one can hope to apply 
our approach to manifolds with edges, too. This will be subject of future work. 
 
The structure of the paper is as follows. In Section \ref{sec:2} we recall the concept of operator-valued symbols twisted by a 
group-action and review the strongly pa\-ra\-me\-ter-dependent Boutet de Monvel calculus. In Section \ref{sec:3} we intoduce 
various symbol classes; of particular importance are the so-called \emph{symbols with expansion at infinity}. In Section \ref{sec:4} 
we shall describe generalized singular Green symbols by associated classes of \emph{symbol-kernels} and shall prove the 
(un-)twisting on operator-level discussed above.  Enlarging Boutet's calculus with strong parameter-dependence by the class 
of generalized Green symbols with expansion at infinity yields our new calculus as is dicussed in Section \ref{sec:5}. 
In Section \ref{sec:6} we summarize results on Toeplitz type pseudodifferential operators which we apply  
in Sextion \ref{sec:7}. In particular, we consider boundary value problems for differential operators subject to global projection 
boundary conditions and derive the structure of the resolvent of the associated realization and resolvent trace asymptotics. 
We compare this with previous work of Grubb. 
In Section \ref{sec:8} we lay out how the calculus can be lifted to compact manifolds with boundary. 


\section{Boutet's algebra with strong parameter-dependence} \label{sec:2}

In this section we describe Boutet de Monvel's algebra with parameter on the half-space $\rz^n_+$. 
The parameter enters in a ``strong way'', which is suited for describing resolvents of 
classical differential boundary value problems. 

\subsection{The algebra of singular Green operators} \label{sec:2.1}
 
The description of the class of generalized singular Green operators is based on the use of pseudodifferential operators 
with twisted  operator-valued symbols in the sense of Schulze. 

\subsubsection{Twisted operator-valued symbols - the abstract setting} \label{sec:2.1.1}

A \emph{group-action} $\{\kappa_\lambda\}$ on a Hilbert space $E$ is a strongly continuous map 
 $$\lambda\mapsto\kappa_\lambda:\rz_+\lra\scrL(E)$$
such that $\kappa_\lambda\kappa_\sigma=\kappa_{\lambda\sigma}$ for all $\lambda,\sigma>0$ 
and $\kappa_1=1$. In particular, $\kappa_\lambda^{-1}=\kappa_{1/\lambda}$. By a well-known result from the theory of 
operator semi-groups there exist constants $C\ge 0$ and $M\in\rz$ such that 
 $$\|\kappa_\lambda\|_{\scrL(E)}\le C\max(\lambda^M,\lambda^{-M}).$$
In the following definition and throughout the whole paper, $y\mapsto[y]$ denotes a smooth positive function defined on 
$\rz^m$ $($we shall use the same notation independent of the specific value of $m)$ which coincides with the usual modulus 
$|y|$ outside the unit-ball.    

\begin{definition}
Let $s\in\rz$ and $E$ be a Hilbert space with group-action $\{\kappa_\lambda\}$. Define $\calW^s(\rz^{n-1},E)$ 
as the closure of $\scrS(\rz^{n-1},E)$ with respect to the norm 
 $$\|u\|_{\calW^s(\rz^{n-1},E)}=\Big(\int_{\rz^{n-1}}[\xi']^{2s}\|\kappa_{[\xi']}^{-1}\wh{u}(\xi')\|_E^2\,d\xi'\Big)^{1/2}.$$
\end{definition}

Here, $\wh u$ denotes the standard Fourier transform, 
 $$\wh{u}(\xi')=\int_{\rz^{n-1}} e^{-ix'\xi'}u(x')\,dx'.$$
The spaces $\calW^s(\rz^{n-1},E)$ are called $($abstract$)$ \emph{edge Sobolev spaces} 
associated with $E$. By construction it is obvious that the Fourier multiplier with symbol $\kappa_{[\xi']}$ 
induces an isometric isomorphism $H^s(\rz^{n-1},E)\lra \calW^s(\rz^{n-1},E)$ for every $s\in\rz$ where $H^s(\rz^{n-1},E)$ 
is the standard Bessel potential space with norm $\|[\cdot]^s\wh u\|_{L^2(\rz^{n-1},E)}$. 

\begin{definition}\label{def:op-symbol1}
For $j=0,1$ let $E_j$ be Hilbert spaces with group-action $\{\kappa_{j,\lambda}\}$. 
The space $S^d_{1,0}(\rz^{n-1};E_0,E_1)$ consists of all smooth functions 
$p:\rz^{n-1}\times\rz^{n-1}\to\scrL(E_0,E_1)$ with 
 $$\|\kappa_{1,[\xi']}^{-1}\{D^{\alpha'}_{\xi'}D^{\beta'}_{x'} p(x',\xi')\}
     \kappa_{0,[\xi']}\|_{\scrL(E_0,E_1)}   \lesssim [\xi']^{d-|\alpha'|}$$
for every choice of multi-indices $\alpha',\beta'\in\nz^{n-1}_0$. 
\end{definition}

With a symbol $p\in S^d_{1,0}(\rz^{n-1};E_0,E_1)$ we associate the pseudodifferential operator 
 $$\op(p)=p(x',D'):\scrS(\rz^{n-1},E_0)\lra\scrS(\rz^{n-1},E_1)$$ 
in the usual way, i.e., 
 $$[\op(p)u](x')=\int e^{ix'\xi'}p(x',\xi')\wh{u}(\xi')\,\dbar\xi',$$
where integration is over $\rz^{n-1}$ and $\dbar\xi'=(2\pi)^{-(n-1)}d\xi'$. 
The following theorem is a consequence of \cite[Theorem 3.14]{Seil97} and the use of order reducing 
operators. 

\begin{theorem}
With the notation from Definition $\ref{def:op-symbol1}$, 
 $$\op(p):\calW^s(\rz^{n-1},E_0)\lra \calW^{s-d}(\rz^{n-1},E_1),\qquad s\in\rz.$$
\end{theorem}

There is a full calculus for such pseudodifferential operators, including composition, formal adjoint, ellipticity and 
parametrix construction;  we will present it in the version with parameter. In the following we shall use the notation  
\begin{align}\label{eq:group-action}
 \kappa(\xi',\mu):=\kappa_{[\xi',\mu]},\qquad (\xi',\mu)\in\rz^{n-1}\times\rpbar.
\end{align}

\begin{definition}\label{def:op-symbol2}
For $j=0,1$, let $E_j$ be Hilbert spaces with group-action $\{\kappa_{j,\lambda}\}$. 
The space $S^d_{1,0}(\rz^{n-1}\times\rpbar;E_0,E_1)$ consists of all smooth functions 
$p:\rz^{n-1}\times\rz^{n-1}\times\rpbar\to\scrL(E_0,E_1)$ with 
\begin{align}\label{eq:symbol-estimate01}
 \|\kappa_{1}^{-1}(\xi',\mu)\{D^{\alpha'}_{\xi'}D^{\beta'}_{x'}D^j_\mu p(x',\xi',\mu)\}
     \kappa_{0}(\xi',\mu)\|_{\scrL(E_0,E_1)}
   \lesssim [\xi',\mu]^{d-|\alpha'|-j}
\end{align}
for every choice of multi-indices $\alpha',\beta'\in\nz^{n-1}_0$ and $j\in\nz_0$. 
\end{definition}

Given $p\in S^d_{1,0}(\rz^{n-1}\times\rpbar;E_0,E_1)$ define $p_\mu(x',\xi'):=p(x',\xi',\mu)$.    
Then $p_\mu$ is a symbol in $S^d_{1,0}(\rz^{n-1};E_0,E_1)$ for every $\mu$. Thus we can consider the 
associated family of pseudodifferential operators; we shall use the notation  
 $$\op(p)(\mu)=p(x',D',\mu):=\op(p_\mu).$$
 
The class of \emph{regularizing} or \emph{smoothing} symbols is 
 $$S^{-\infty}(\rz^{n-1}\times\rpbar;E_0,E_1):=
    \schnitt_{d\in\rz}S^d_{1,0}(\rz^{n-1}\times\rpbar;E_0,E_1);$$
it is easy to see that 
\begin{align*}
 S^{-\infty}(\rz^{n-1}\times\rpbar;E_0,E_1)
     =\scrS(\rz^{n-1}\times\rpbar,\scrC^\infty_b(\rz^{n-1}_{x'};\scrL(E_0,E_1))),
\end{align*}
i.e., smoothing symbols are rapidly decreasing in $(\xi',\mu)$.  
 
If $V$ is a conic subset of $\rz^m$ for some $m$, a function $p:\rz^{n-1}\times V\to\scrL(E_0,E_1)$ 
shall be called \emph{twisted homogeneous} of degree $d$ $($with respect to the variable $v)$ provided  
 $$p(x',\lambda v)=\lambda^d\,\kappa_{1,\lambda}\,p(x',v)\,\kappa^{-1}_{0,\lambda}
    \qquad\forall\;(x',v)\in\rz^{n-1}\times V\quad\forall\;\lambda>0.$$ 

\begin{definition}\label{def:twisted-homogeneous}
For $j=0,1$ let $E_j$ be Hilbert spaces with group-action $\{\kappa_{j,\lambda}\}$. 
Then $S^d_\Hom(\rz^{n-1}\times\rpbar;E_0,E_1)$ denotes the space of all functions $p$ which are 
twisted homogeneous in the above sense with $V=(\rz^{n-1}\times\rpbar)\setminus\{0\}$ and which 
satisfy the estimates 
\begin{align}\label{eq:symbol-estimate02}
 \|\kappa_{1,|\xi',\mu|}^{-1}\{D^{\alpha'}_{\xi'}D^{\beta'}_{x'}D^j_\mu p(x',\xi',\mu)\}
     \kappa_{0,|\xi',\mu|}\|_{\scrL(E_0,E_1)}
   \lesssim |\xi',\mu|^{d-|\alpha'|-j}
\end{align}
for every choice of multi-indices $\alpha',\beta'\in\nz^{n-1}_0$ and $j\in\nz_0$. 
\end{definition}

If $\chi(\xi',\mu)$ is a zero-excision function $($i.e., $\chi\in\scrC^\infty(\rz^{n-1}\times\rpbar)$ such that $\chi\equiv0$ 
in some neighborhood of the origin and $1-\chi$ has compact support$)$ and 
$p\in S^d_\Hom(\rz^{n-1}\times\rpbar;E_0,E_1)$ then 
$\chi p\in S^d_{1,0}(\rz^{n-1}\times\rpbar;E_0,E_1)$. Hence the following definition makes sense: 

\begin{definition}\label{def:poly-homogeneous}
For $j=0,1$ let $E_j$ be Hilbert spaces with group-action $\{\kappa_{j,\lambda}\}$. 
A symbol $p\in S^d_{1,0}(\rz^{n-1}\times\rpbar;E_0,E_1)$ is called $($twisted$)$ poly-homogeneous provided 
there exists a sequence of symbols $p^{(d-j)}\in S^{d-j}_\Hom(\rz^{n-1}\times\rpbar;E_0,E_1)$ 
such that, for every $N\in\nz_0$, 
 $$p-\sum_{j=0}^{N-1}\chi(\xi',\mu)p^{(d-j)}\;\in\; S^{d-N}_{1,0}(\rz^{n-1}\times\rpbar;E_0,E_1).$$
The space of such symbols shall be denoted by $S^d(\rz^{n-1}\times\rpbar;E_0,E_1)$, the 
leading component $p^{(d)}$ is the so-called \emph{homogeneous principal symbol} of $p$.  
\end{definition}

The homogeneous principal symbol can be obtained from the full symbol, i.e.,
\begin{align}\label{eq:formula-principal-symbol}
    p^{(d)}(x',\xi',\mu)=\lim_{\lambda\to+\infty}\lambda^{-d}\kappa_{1,\lambda}^{-1}
    p(x',\lambda\xi',\lambda\mu)\kappa_{0,\lambda},\qquad (\xi',\mu)\not=0.
\end{align}

\begin{theorem}\label{thm:Leibniz01}
Let $p_j\in S^{d_j}_{1,0}(\rz^{n-1}\times\rpbar;E_j,E_{j+1})$, $j=0,1$. Then there exists a unique symbol 
$p_1\#p_0\in S^{d_0+d_1}_{1,0}(\rz^{n-1}\times\rpbar;E_0,E_2)$, the so-called \emph{Leibniz product} of 
$p_0$ and $p_1$, such that, for all $\mu$,  
 $$\op(p_1)(\mu)\op(p_0)(\mu)=\op(p_1\#p_0)(\mu).$$
In case both $p_0$ and $p_1$ are poly-homogeneous, so is $p_1\#p_0$; moreover 
 $$(p_1\#p_0)^{(d_0+d_1)}=p_1^{(d_1)}p_0^{(d_0)}.$$
\end{theorem}

The Leibniz product of two symbols (not only here but also for all other classes of symbols in this paper) 
is given by the formula 
 $$(p\#q)(x',\xi')=\mathrm{Os}-\iint e^{-iy'\eta'}p(x',\xi'+\eta')q(x'+y',\xi')\,dy'\dbar\eta'$$ 
where the integral is an \emph{oscillatory integral} (for the precise definition of the oscillatory integral see 
Theorem \ref{thm:oscillatory-int}, below). In this formula $p$ and  
$q$ may also depend on the parameter $\mu$.

\begin{definition}
A symbol $p\in S^d(\rz^{n-1}\times\rpbar;E_0,E_1)$ is called \emph{elliptic} if its homogeneos 
principal symbol is pointwise invertible whenever $(\xi',\mu)\not=0$ and 
\begin{align}\label{eq:uniform}
 \|p^{(d)}(x',\xi',\mu)^{-1}\|_{\scrL(E_1,E_0)}\lesssim 1\qquad\forall\;x'\in\rz^{n-1}\quad\forall\;|\xi',\mu|=1.
\end{align}
\end{definition}

This estimate \eqref{eq:uniform} is necessary (only) for ensuring the boundedness of the inverse in the $x'$-variable 
and implies 
$(p^{(d)})^{-1}\in S^{-d}_\Hom(\rz^{n-1}\times\rpbar;E_1,E_0)$.  

\begin{theorem}
Let $p\in S^{d}(\rz^{n-1}\times\rpbar;E_0,E_{1})$ be elliptic. Then there exists a symbol 
$q\in S^{-d}(\rz^{n-1}\times\rpbar;E_1,E_0)$ such that $q\#p-1$ and $p\#q-1$ are symbols of order $-\infty$ 
that vanish for sufficiently large $\mu$. In particular, 
$\op(p)(\mu):\calW^s(\rz^{n-1},E_0)\to \calW^{s-d}(\rz^{n-1},E_1)$ is invertible for large $\mu$ with  
$\op(p)(\mu)^{-1}=\op(q)(\mu)$. 
\end{theorem}

We also shall need the notion of the \emph{formal adjoint} of an operator. 

\begin{definition}\label{def:Hilbert-triple}
Let $E$ and $\wt E$ be Hilbert spaces with group-action $\{\kappa_\lambda\}$ and $\{\wt\kappa_\lambda\}$, respectively. 
Let $H$ be a Hilbert space. We call $(E,H,\wt{E})$ a Hilbert-triple if the inner product of $H$ induces a 
continuous non-degenerate sesquilinear pairing $E\times \wt{E}\to\cz$ that permits to identify the dual 
space of $E$ with $\wt{E}$ and vice versa. Moreover, we require 
 $$(\kappa_{\lambda}e,\wt{e})_H=(e,\wt{\kappa}_{\lambda}\wt{e})_H\qquad 
     \forall\;e\in E\quad\forall\:\wt{e}\in\wt E\quad\forall\:\lambda>0.$$
\end{definition}

Let $(E_j,H_j,\wt{E}_{j})$, $j=0,1$, be Hilbert-triples. If $T\in\scrL(E_0,E_1)$ then there exists a unique operator 
$T^*\in\scrL(\wt{E}_1,\wt{E}_0)$ such that $(Te_0,\wt e_1)_{\wt H}=(e_0,T^*\wt e_1)_{H}$ for every 
$e_0\in E_0$ and $\wt e_1\in\wt E_1$. By slight abuse of language we call $T^*$ the adjoint of $T$. 
Moreover, if $A:\scrS(\rz^n,E_0)\to\scrS(\rz^n,E_1)$, then 
the operator $A^*:\scrS(\rz^n,\wt{E}_1)\to\scrS(\rz^n,\wt{E}_0)$ defined by  
 $$(Au,v)_{L^2(\rz^n,H_1)}=(u,A^*v)_{L^2(\rz^n,H_0)},\qquad 
     u\in\scrS(\rz^n,E_0),\;v\in\scrS(\rz^n,\wt{E}_1).$$ 
is called the \emph{formal adjoint} of $A$. 

\begin{theorem}\label{thm:adjoint01}
Let $(E_j,H_j,\wt{E}_{j})$, $j=0,1$, be Hilbert triples and 
$p\in S^{d}_{1,0}(\rz^{n-1}\times\rpbar;E_0,E_{1})$. Then there exists a unique symbol 
$p^{(*)}\in S^{-d}_{1,0}(\rz^{n-1}\times\rpbar;\wt E_1,\wt E_0)$ such that, for all $\mu$,  
 $$\op(p)(\mu)^*=\op(p^{(*)})(\mu).$$
If $p$ is poly-homogeneous then so is $p^{(*)}$ and 
$(p^{(*)})^{(d)}=(p^{(d)})^*$, where $*$ on the right is the pointwise adjoint principal symbol. 
\end{theorem}

The formal adjoint symbol (not only here but also for all other classes of symbols in this paper) 
is given by the formula 
 $$p^{(*)}(x',\xi')=\mathrm{Os}-\iint e^{-iy'\eta'}p(x'+y',\xi'+\eta')^*\,dy'\dbar\eta'.$$

Of course, all the above considerations include the case of trivial group-actions, i.e.,  
$\kappa_j\equiv1$. To emphasize this particular case we shall use a different notation: 

\begin{definition}\label{def:ohne-kappa01}
If $E_0$ and $E_1$ carry the trivial group action we set
\begin{align*}
 S^d_{1,0}(\rz^{n-1}\times\rpbar;\scrL(E_0,E_1))&:=S^d_{1,0}(\rz^{n-1}\times\rpbar;E_0,E_1), \\ 
 S^d_\Hom(\rz^{n-1}\times\rpbar;\scrL(E_0,E_1))&:=S^d_\Hom(\rz^{n-1}\times\rpbar;E_0,E_1),\\
 S^d(\rz^{n-1}\times\rpbar;\scrL(E_0,E_1))&:=S^d(\rz^{n-1}\times\rpbar;E_0,E_1).
\end{align*}
\end{definition}

Let us finally comment on the topology of symbol spaces. Using the estimates \eqref{eq:symbol-estimate01}
and \eqref{eq:symbol-estimate02} it is obvious how to define a Fréchet topology on the spaces 
$S^d_{1,0}(\rz^{n-1}\times\rpbar;E_0,E_1)$ and $S^d_\Hom(\rz^{n-1}\times\rpbar;E_0,E_1)$, respectively. 
The class of poly-homogeneous symbols then carries the projective topology under the maps 
\begin{align*}
 p&\mapsto p^{(d-j)}:S^d(\rz^{n-1}\times\rpbar;E_0,E_1)\lra S^d_\Hom(\rz^{n-1}\times\rpbar;E_0,E_1), \\ 
 p&\mapsto p-\sum_{j<N}\chi p^{(d-j)}:S^d(\rz^{n-1}\times\rpbar;E_0,E_1)\lra 
 S^{d-N}_{1,0}(\rz^{n-1}\times\rpbar;E_0,E_1), 
\end{align*}
with $j,N\in\nz_0$. This makes $S^d(\rz^{n-1}\times\rpbar;E_0,E_1)$ a Frèchet space, too. Both maps 
$(p_1,p_0)\mapsto p_1\#p_0$ and $p\mapsto p^{(*)}$ are then continuous in the respective symbol spaces, 
cf. Theorems \ref{thm:Leibniz01} and \ref{thm:adjoint01}. 

\subsubsection{Sobolev spaces and the standard group-action} \label{sec:2.1.2}

Let us now apply the above abstract concept to operators on the half-space. 
 
\begin{definition}\label{def:group-action}
Let $I=\rz$ or $I=\rz_+$. For $\varphi\in\scrD(I)$ define 
$\kappa_\lambda\varphi\in\scrD(I)$ by 
$(\kappa_\lambda \varphi)(t)=\lambda^{1/2}\varphi(\lambda t)$. For a distribution 
$T\in\scrD'(I)$ define $\kappa_\lambda T\in\scrD'(I)$ by 
 $$(\kappa_\lambda T)(\varphi)=T(\kappa^{-1}_{\lambda}\varphi),\qquad \varphi\in\scrD(I).$$
\end{definition}

It is easily seen that $\kappa_\lambda$ of the previous definition induces a  
group-action on the spaces $H^{s,\delta}(\rz)$, $H^{s,\delta}(\rz_+)$, and $H^{s,\delta}_0(\rpbar)$ 
for any choice of $s,\delta\in\rz$. We shall refer to it as the \emph{dilation group-action}. Here, 
\begin{align*}
 H_0^s(\rpbar)&=\{u\in H^s(\rz)\mid \mathrm{supp}\,u\subseteq\rpbar\},\\ 
 H^s(\rz_+)&=\{u\in \scrD'(\rz_+)\mid u=v|_{\rz_+}\text{ for some }v\in H^s(\rz)\};
\end{align*} 
moreover, 
 $$H^{s,\delta}(\rz)=\{[\cdot]^{-\delta} u\mid u\in H^s(\rz)\}$$ 
and analogously for the other spaces. 
Note that every $\kappa_\lambda$ is unitary as an operator on $L^2(\rz)$ and $L^2(\rz_+)$, respectively.  
Moreover, $(H^{s,\delta}(\rz),L^2(\rz),H^{-s,-\delta}(\rz))$ and 
$(H^{s,\delta}_0(\rpbar),L^2(\rz_+),H^{-s,-\delta}(\rz_+))$ are Hilbert triples in the sense of 
Definition \ref{def:Hilbert-triple}. The restriction of the Schwartz space to the half-axis is denoted by 
\begin{align*}
 \scrS(\rz_+)&=\{u\in \scrD'(\rz_+)\mid u=v|_{\rz_+}\text{ for some }v\in \scrS(\rz)\}. 
\end{align*}

The following result is the main motivation for introducing edge Sobolev spaces involving group-actions. 

\begin{theorem}
Let $u\in\scrS(\rz^n)\cong\scrS(\rz^{n-1},\scrS(\rz))$. Then, for every $s\in\rz$,  
 $$\|u\|_{H^s(\rz^n)}=\|u\|_{\calW^s(\rz^{n-1},H^s(\rz))}.$$
\end{theorem}

This yields an isometric isomorphism of $H^s(\rz^n)$ and $\calW^s(\rz^{n-1},H^s(\rz))$. 
By this result one then obtains the identifications
\begin{align*}
 H^s_0(\overline{\rz^n_+})=\calW^s(\rz^{n-1},H^s_0(\rpbar)),\qquad 
 H^s(\rz^n_+)=\calW^s(\rz^{n-1},H^s(\rz_+))
\end{align*}

\begin{remark}\label{rem:group-action-operator}
Recall notation $\eqref{eq:group-action}$. 
For each $\mu\in\rpbar$ and $s\in\rz$ the map  
$$\op(\kappa)(\mu): H^s(\rz^{n-1},H^s(\rz_+))\lra \calW^s(\rz^{n-1},H^s(\rz_+))=H^{s}(\rz^{n}_+)$$
is an isomorphism with inverse $\op(\kappa^{-1})(\mu)$;  
we shall refer to it as the \emph{group-action operator}. 
If $u\in\scrS(\rz^n_+)\cong \scrS(\rz^{n-1},\scrS(\rz_+))$ then we can write explicitly 
\begin{equation*}
	[\op(\kappa)(\mu)u](x)=\scrF^{-1}_{\xi'\to x'}\big([\xi',\mu]^{1/2}\wh{u}(\xi',[\xi',\mu]x_n)\big)(x'),
\end{equation*}
where $\wh{u}$ indicates the partial Fourier transform of $u$ in the variable $x'$. 
\end{remark}

In the sequel we shall frequently apply the following simple observation which is easily proved by induction. 

\begin{lemma}
Let $u\in \scrS(\rz_+)$ and $\{\kappa_\lambda\}$ be the dilation group-action. Then  
 $$\frac{d^\ell}{d\lambda^\ell}\kappa_\lambda u= 
     \frac{1}{\lambda^\ell}\kappa_\lambda \Theta_\ell u$$
for every $\ell\in\nz$, where 
 $$\Theta_\ell=\prod_{k=0}^{\ell-1}\Big(t\partial_{t}+\frac{1}{2}-k\Big):\scrS(\rz_+)\lra\scrS(\rz_+).$$
Note that $\kappa_\lambda$ and $\Theta_\ell$ commute.
\end{lemma}

If $a:U\subset\rz^m\to\rz_+$ is a smooth function, the previous lemma in combination with the chain-rule yields 
that $D^\alpha_y \kappa_{a(y)}u$, $|\alpha|\ge1$, is a finite linear-combination of terms of the form 
\begin{align}\label{eq:chainrule}
\begin{split}
 &a(y)^{-m}D^{\alpha_1}_ya(y)\cdot\ldots\cdot D^{\alpha_m}_ya(y)\kappa_{a(y)}\Theta_m u,\\
 &\text{with }1\le m\le|\alpha|,\;\alpha_1+\ldots+\alpha_m=\alpha.
\end{split}
\end{align}

\subsubsection{Generalized singular Green operators}\label{sec:2.1.3}

In the following we shall extend the dilation group-action from Definition \ref{def:group-action} to 
$\scrD'(I,\cz^L)$ (by component-wise application) and to $\scrD'I,\cz^L)\oplus\cz^{M}$ as 
$\kappa_\lambda\oplus 1$. By slight abuse of notation we write $\bfkappa_\lambda:=\kappa_\lambda\oplus 1$ 
for any choice of $L$ and $M$. Correspondingly, we write 
\begin{align}\label{eq:kappa-matrix}
 \bfkappa(\xi',\mu):=\kappa(\xi',\mu)\oplus1=\begin{pmatrix}\kappa(\xi',\mu)&0\\0&1\end{pmatrix}:
 \begin{matrix}\scrD'(I,\cz^L)\\ \oplus\\ \cz^M\end{matrix}
 \lra\begin{matrix}\scrD'(I,\cz^L)\\ \oplus\\ \cz^M\end{matrix}.
\end{align}

In order to have a managable formalism at hand, let us introduce so-called \emph{weight-data}: 

\begin{definition}\label{def:weight}
A \emph{weight-datum} is a tuple $\frakg=((L_0,M_0),(L_1,M_1))$ with $L_j,M_j\in\nz_0$, $j=0,1$. Its inverse 
weight--datum is $\frakg^{-1}:=((L_1,M_1),(L_0,M_0))$. 

A datum $\frakg_1$ is said to be \emph{composable} with $\frakg_0$ provided 
$\frakg_j=((L_j,M_j),(L_{j+1},M_{j+1}))$, $j=0,1$. 
In this case we define the product weight-datum 
$\frakg_1\frakg_0:=((L_0,M_0),(L_2,M_2))$. 
\end{definition}

\begin{definition}\label{def:singular-green01}
$B^{d;0}_G(\rz^{n-1}\times\rpbar;\frakg)$, $\frakg=((L_0,M_0),(L_1,M_1))$,  denotes the space 
 $$\mathop{\mbox{\Large$\cap$}}_{s,s^\prime,\delta,\delta^\prime\in\rz} 
    S^d(\rz^{n-1}\times\rpbar;H^{s,\delta}_0(\rpbar,\cz^{L_0})\oplus\cz^{M_0},
    H^{s^\prime,\delta^\prime}(\rz_+,\cz^{L_1})\oplus\cz^{M_1}).$$ 
Elements of $B^{d;0}_G(\rz^{n-1}\times\rpbar;\frakg)$ are so-called \emph{generalized singular Green symbols}. 
\end{definition}

In the previous definition $d$ represents the order of the symbol while the superscript $0$ indicates that these 
symbols have \emph{type} or \emph{class} 0. Below we shall define symbols of positive type. Note also that 
we allow the values $M_0=0$ or $M_1=0$, respectively, where $X\oplus\cz^0:=X\times\{0\}:=X$. In any case, using 
the topology of spaces of operator-valued symbols, we naturally obtain a Fréchet topology on 
$B^{d;0}_G(\rz^{n-1}\times\rpbar;\frakg)$. 

Any generalized singular Green symbol $\bfg\in B^{d;0}_G(\rz^{n-1}\times\rpbar;\frakg)$ 
has a representation in block-matrix form, i.e.,  
\begin{align}\label{eq:block-matrix}
  \bfg(x^\prime,\xi^\prime,\mu)
  =\begin{pmatrix}
  g(x^\prime,\xi^\prime,\mu) & k(x^\prime,\xi^\prime,\mu)\\
  t(x^\prime,\xi^\prime,\mu) & q(x^\prime,\xi^\prime,\mu)
\end{pmatrix};
\end{align}
then $g$ is called a \emph{singular Green symbol}, $k$ is a \emph{Poisson symbol}, 
$t$ is a \emph{trace symbol}  
and  $q$ is a usual pseudodifferential symbol. Note that if $M_0=0$ then 
$\bfg=\begin{pmatrix}g\\t\end{pmatrix}$, while $\bfg=(g\quad k)$ in case $M_1=0$.  
The abstract theory of operator-valued symbols yields, in 
particular, that 
 $$\op(\bfg)(\mu): 
    \begin{matrix}H^s_0(\overline{\rz^n_+},\cz^{L_0})\\ \oplus \\ H^{s}(\rz^{n-1},\cz^{M_0})\end{matrix}
    \lra
    \begin{matrix}H^{s-d}(\rz^n_+,\cz^{L_1})\\ \oplus \\  H^{s-d}(\rz^{n-1},\cz^{M_1})\end{matrix},
    \qquad s\in\rz,
 $$
continuously for every $\mu\in\rpbar$. Moreover, the class of generalized singular Green operators is closed 
under composition and taking the formal adjoint. Extension by zero from $\rz_+$ to $\rz$ defines a meaningful embedding  
\begin{align}\label{eq:e-plus}
 e_+:H^s(\rz_+)\lra H^{-1/2}_0(\rpbar),\qquad s>-\frac12.
\end{align}
In this sense, generalized singular Green symbols $\bfg$ of order $d$ and type $0$ can be considered as 
operator-valued symbols belonging to 
  $$\mathop{\mbox{\Large$\cap$}}_{\substack{s>-1/2,\\ s^\prime,\delta,\delta^\prime\in\rz}} 
    S^d(\rz^{n-1}\times\rpbar;H^{s,\delta}(\rz_+,\cz^{L_0})\oplus\cz^{M_0},
    H^{s^\prime,\delta^\prime}(\rz_+,\cz^{L_1})\oplus\cz^{M_1}),$$ 
hence induce operators 
 $$\op(\bfg)(\mu): 
    \begin{matrix}H^s(\rz^n_+,\cz^{L_0})\\ \oplus \\ H^{s}(\rz^{n-1},\cz^{M_0})\end{matrix}
    \lra
    \begin{matrix}H^{s-d}(\rz^n_+,\cz^{L_1})\\ \oplus \\  H^{s-d}(\rz^{n-1},\cz^{M_1})\end{matrix},
    \qquad s>-\frac12. 
 $$

The operator of differentiation on $\scrD^\prime(\rz_+,\cz^L)$, denoted by $\partial_+$, 
induces $($constant$)$ operator-valued symbols 
\begin{align}\label{eq:partial-plus}
\boldsymbol{\partial}_+
:=\begin{pmatrix}\partial_+&0\\0&1\end{pmatrix}\in  S^1(\rz^{n-1}\times\rpbar;
H^{s,\delta}(\rz_+,\cz^L)\oplus\cz^M,H^{s-1,\delta}(\rz_+,\cz^L)\oplus\cz^{M})
\end{align}
for every $s,\delta\in\rz$ $($by slight abuse of notation, we shall use the same notation for every choice 
of $L$ and $M)$. Note that $\boldsymbol{\partial_+}$ is twisted homogeneous of degree 1 with respect to 
$\{\bfkappa_\lambda\}$. 

\begin{definition}\label{def:singular-green02}
Let $r\in\nz_0$. Then $B^{d;r}_G(\rz^{n-1}\times\rpbar;\frakg)$, $\frakg=((L_0,M_0),(L_1,M_1))$,  
denotes the space of all symbols of the form 
 $$\bfg=\bfg_0+\sum_{j=1}^{r}\bfg_j\boldsymbol{\partial}_+^j,\qquad 
   \bfg_j\in B^{d-j;0}_G((\rz^{n-1}\times\rpbar;\frakg).$$
$d$ is called the order of $\bfg$ while $r$ is its \emph{type} $($also called \emph{class} in the 
literature$)$. 
\end{definition}

By construction, if $\bfg$ is as in the previous definition, then 
 $$\bfg\in\mathop{\mbox{\Large$\cap$}}_{\substack{s>r-1/2,\\ s^\prime,\delta,\delta^\prime\in\rz}} 
    S^d(\rz^{n-1}\times\rpbar;H^{s,\delta}(\rz_+,\cz^{L_0})\oplus\cz^{M_0},
    H^{s^\prime,\delta^\prime}(\rz_+,\cz^{L_1})\oplus\cz^{M_1}),$$ 
hence induces maps 
\begin{align}\label{eq:mapping-property01}
 \op(\bfg)(\mu): 
    \begin{matrix}H^s(\rz^n_+,\cz^{L_0})\\ \oplus \\ H^{s}(\rz^{n-1},\cz^{M_0})\end{matrix}
    \lra
    \begin{matrix}H^{s-d}(\rz^n_+,\cz^{L_1})\\ \oplus \\  H^{s-d}(\rz^{n-1},\cz^{M_1})\end{matrix},
    \qquad s>r-\frac12. 
\end{align}
Moreover, $\bfg$ has the homogeneous principal symbol  
\begin{align*}
 \bfg^{(d)}(x',\xi',\mu)=\bfg_0^{(d)}+\sum_{j=1}^{r}\bfg_j^{(d-j)}(x',\xi',\mu)\boldsymbol{\partial}_+^j
\end{align*}
which is defined whenever $(\xi',\mu)\not=0$.
Using the map 
$(\bfg_0,\ldots,\bfg_r)\mapsto \bfg_0+\sum\limits_{j=1}^{r}\bfg_j\boldsymbol{\partial}_+^j$
and factoring out its kernel, we obtain a natural Fréchet topology on 
$B^{d;r}_G(\rz^{n-1}\times\rpbar;\frakg)$.

\begin{remark}\label{rem:sigular-green-half-axis}
Using symbols without parameter, cf. Definition $\ref{def:op-symbol1}$, we can define in the same way the classes 
of parameter-independent generalized singular Green symbols $B^{d;r}_G(\rz^{n-1};\frakg)$. 
In the special case $n=1$ there is also no dependence on $(x',\xi')$, in particular the order $d$ is obsolete. 
To avoid confusion we use the notation $\Gamma^{r}_G$ for the class of generalized singular Green operators of type $r$
on the half-axis $\rz_+$; to be precise, if $\frakg=((L_0,M_0),(L_1,M_1))$, 
\begin{align*}
 \Gamma^0_G(\rz_+;\frakg)
 =\mathop{\mbox{\Large$\cap$}}_{s,s^\prime,\delta,\delta^\prime\in\rz} 
    \scrL(H^{s,\delta}_0(\rpbar,\cz^{L_0})\oplus\cz^{M_0},
    H^{s^\prime,\delta^\prime}(\rz_+,\cz^{L_1})\oplus\cz^{M_1}))
\end{align*}
and operators of type $r>0$ are then defined analogously to Definition  $\ref{def:singular-green02}$. 
\end{remark}

\subsection{Pseudodifferential operators and the transmission property} \label{sec:2.2}

We shall write $S^d_{1,0}(\rz^n\times\rpbar;\scrL(\cz^{L_0},\cz^{L_1}))$ for the standard Hörmander class 
of pseudodifferential symbols $a=a(x,\xi,\mu)$ and $S^d(\rz^n\times\rpbar;\scrL(\cz^{L_0},\cz^{L_1}))$ 
for the subclass of poly-homogeneous symbols. 
As before, we shall denote the homogeneous components of a 
poly-homogeneous symbol $a$ of order $d$ by $a^{(d-\ell)}$, $\ell\in\nz_0$. We werite $x=(x',x_n)$ and 
$\xi=(\xi',\xi_n)$.  

\begin{definition}\label{def:transmission-property}
$a\in S^d(\rz^n\times\rpbar;\scrL(\cz^{L_0},\cz^{L_1}))$, $d\in\gz$, is said to satisfy the 
\emph{transmission condition} $($with respect to $x_n=0)$ provided 
\begin{align*}
 (\partial^k_{x_n}\partial^{\alpha'}_{\xi'}\partial^j_\mu a^{(d-\ell)})
 &(x',0,\xi',1,\mu)\big|_{(\xi',\mu)=0}\\
 &= (-1)^{d-\ell-|\alpha'|-j}
 (\partial^k_{x_n}\partial^{\alpha'}_{\xi'}\partial^j_\mu a^{(d-\ell)})(x',0,\xi',-1,\mu)\big|_{(\xi',\mu)=0}
\end{align*}
for every choice of $\alpha'\in\nz_0^{n-1}$ and $j,k,\ell\in\nz_0$. Write 
$S^d_\tr(\rz^n\times\rpbar;\scrL(\cz^{L_0},\cz^{L_1}))$ for the space of all such symbols. 
\end{definition}

$S^d_\tr(\rz^n\times\rpbar;\scrL(\cz^{L_0},\cz^{L_1}))$ is a closed subspace of 
$S^d(\rz^n\times\rpbar;\scrL(\cz^{L_0},\cz^{L_1}))$. Moreover, symbols of order $-\infty$ and 
symbols that vanish to infinite order in $x_n=0$ always satisfy the transmission condition. 
The transmission condition is preserved under composition, formal adjoint (in case $d\le0)$, and parametrix 
construction. 
With $e^+$ from \eqref{eq:e-plus}, for a symbol $a(x,\xi,\mu)$ we then define 
\begin{align*}
 \op^+(a)(x',\xi',\mu)=r^+\,\op_{x_n}(a)(x',\xi',\mu)\,e^+,
\end{align*}
where $\op_{x_n}(a)(x',\xi',\mu)=a(x',x_n,\xi',D_{x_n},\mu)$ means the pseudodifferential operator on $\rz$ 
with respect to the $x_n$-variable with co-variable $\xi_n$ and $r^+:\scrS'(\rz)\to\scrD'(\rz_+)$ is the operator 
of restriction. 

\begin{theorem}\label{thm:op-plus}
Let $S^d_\tr(\rz^n\times\rpbar;\scrL(\cz^{L_0},\cz^{L_1}))$. Then, for all $s>-\frac12$ and $\delta\in\rz$,  
 $$\op^+(a)\in S^d_{1,0}(\rz^{n-1}\times\rpbar;
    H^{s,\delta}(\rz_+,\cz^{L_0}),H^{s-d,\delta}(\rz_+,\cz^{L_1})).$$
In case $a(x,\xi,\mu)=a(x',\xi,\mu)$ does not depend on the $x_n$-variable, 
 $$\op^+(a)\in S^d(\rz^{n-1}\times\rpbar;
    H^{s,\delta}(\rz_+,\cz^{L_0}),H^{s-d,\delta}(\rz_+,\cz^{L_1}))$$
is poly-homogeneous with homogeneous principal symbol $\op^+(a^{(d)})(x',\xi',\mu)$. 
\end{theorem}

\subsection{Boundary symbols with strong parameter-dependence} \label{sec:2.3}

\begin{definition}\label{def:bdry-symbol01}
Let $d\in\gz$ and $r\in\nz_0$. Then $B^{d;r}(\rz^{n-1}\times\rpbar;\frakg)$ with weight-datum 
$\frakg=((L_0,M_0),(L_1,M_1))$,  
denotes the space of all symbols of the form 
 $$\bfp(x',\xi',\mu)
     =\begin{pmatrix}\op^+(a)(x',\xi',\mu)&0\\0&0\end{pmatrix}+\bfg(x^\prime,\xi^\prime,\mu),$$
where $a\in S^d_\tr(\rz^n\times\rpbar;\scrL(\cz^{L_0},\cz^{L_1}))$ satisfies the transmission condition 
and $\bfg$ is a generalized singular Green symbol belonging to 
$B^{d;r}_G(\rz^{n-1}\times\rpbar;\frakg)$.  Any such $\bfp$ is called a \emph{boundary symbol}. 
$\bfp$ is called a \emph{regularizing} boundary symbol if both $a$ and $\bfg$ have order $d=-\infty$. 
The class of regularizing boundary symbols of type $r$ is denoted by 
$B^{-\infty;r}(\rz^{n-1}\times\rpbar;\frakg)$. 
\end{definition}

Note that the representation of $\bfp$ in the previous definition is not unique. It can be shown that if $\bfp$ has 
another representation with pseudodifferential part $\op^+(b)(x',\xi',\mu)$ then there exists a smoothing symbol 
$r\in S^{-\infty}(\rz^n\times\rpbar;\scrL(\cz^{L_0},\cz^{L_1}))$ such that 
$a(x,\xi;\mu)-b(x,\xi;\mu)=r(x,\xi;\mu)$ whenever $x_n\ge0$. 

If $\bfp$ is as in Definition \ref{def:bdry-symbol01} then $\op(\bfp)(\mu)$ has property 
\eqref{eq:mapping-property01}. With $\bfp$ we associate its 
\emph{homogeneous principal symbol} 
 $$\sigma_\psi^{(d)}(\bfp)(x,\xi,\mu)=a^{(d)}(x,\xi,\mu),\qquad  x_n\ge0,\quad (\xi,\mu)\not=0,$$
as well as its \emph{principal boundary symbol} 
 $$\sigma_\partial^{(d)}(\bfp)(x',\xi',\mu)
    =\begin{pmatrix}\op^+(a_0^{(d)})(x',\xi',\mu)&0\\0&0\end{pmatrix}
    +\bfg^{(d)}(x',\xi',\mu),\qquad (\xi',\mu)\not=0;$$
here $a_0$ denotes the symbol with coefficients ``frozen'' at the boundary, i.e., 
 $$a_0(x,\xi,\mu)=a(x',0,\xi,\mu).$$
Note that, for every $s>r-1/2$,  
\begin{align*}
 \sigma_\partial^{(d)}(\bfp)\in 
 S^d_\Hom(\rz^{n-1}\times\rpbar;H^s(\rz_+,\cz^{L_0})\oplus\cz^{M_0},
   H^{s-d}(\rz_+,\cz^{L_1})\oplus\cz^{M_1}).
\end{align*}

\begin{theorem}\label{thm:adjoint-strong}
Let $d\le 0$. Then $\bfp\mapsto \bfp^{(*)}$ induces continuous maps 
\begin{align*} 
 B^{d;0}(\rz^{n-1}\times\rpbar;\frakg)
  &\lra B^{d;0}(\rz^{n-1}\times\rpbar;\frakg^{-1}),\\
 B^{d;0}_G(\rz^{n-1}\times\rpbar;\frakg)
  &\lra B^{d;0}_G(\rz^{n-1}\times\rpbar;\frakg^{-1}).
\end{align*}
Both homogeneous principal symbol and principal boundary symbol behave naturally under taking the 
formal adjoint, i.e. 
 $$\sigma_\psi^{(d)}(\bfp^{(*)})=\sigma_\psi^{(d)}(\bfp)^*,\qquad  
     \sigma_\partial^{(d)}(\bfp^{(*)})=\sigma_\partial^{(d)}(\bfp)^*,$$
where $*$ on the right-hand sides means the adjoint of $(L_1\times L_0)$-matrices and the adjoint 
with respect to the corresponding $L^2(\rz_+)$-pairings, respectively
\end{theorem}

\begin{theorem}\label{thm:comp-strong}
Let $\frakg_1$ be composable with $\frakg_0$, cf. Definition $\ref{def:weight}$.
Then $(\bfp_1,\bfp_0)\mapsto \bfp_1\#\bfp_0$ induces continuous maps
\begin{align*} 
 B^{d_1;r_1}&(\rz^{n-1}\times\rpbar;\frakg_1)\times  
     B^{d_0;r_0}(\rz^{n-1}\times\rpbar;\frakg_0)\\
    &\lra B^{d;r}(\rz^{n-1}\times\rpbar;\frakg_1\frakg_0),  
\end{align*}
where 
 $$d=d_0+d_1,\qquad r=\max(r_1+d_0,r_0).$$
Both homogeneous principal symbol and principal boundary symbol are multiplicative, i.e., 
\begin{align*} 
  \sigma_\psi^{(d_0+d_1)}(\bfp_1\#\bfp_0)
     &=\sigma_\psi^{(d_1)}(\bfp_1)\sigma_\psi^{(d_0)}(\bfp_0),\\
  \sigma_\partial^{(d_0+d_1)}(\bfp_1\#\bfp_0)
     &=\sigma_\partial^{(d_1)}(\bfp_1)\sigma_\partial^{(d_0)}(\bfp_0).
\end{align*}
The class of generalized singular Green operators is a two-sided ideal with respect to composition. Moreover, 
$(\bfp_1,\bfp_0)\mapsto \bfp_1\#\bfp_0$ restricts to 
\begin{align*} 
 B^{d_1;r_1}_G&(\rz^{n-1}\times\rpbar;\frakg_1)\times  
     B^{d_0;r_0}_G(\rz^{n-1}\times\rpbar;\frakg_0)\\
    &\lra B^{d_0+d_1;r_0}_G(\rz^{n-1}\times\rpbar;\frakg_1\frakg_0).   
\end{align*}
\end{theorem}

If in the previous theorem $\op^+(a_j)$ is the pseudodifferential part of $\bfp_j$ then $\bfp_1\#\bfp_0$ 
can be written with the pseudodifferential part $\op^+(a_1\#a_0)$. 

\subsubsection{Reduction of orders}\label{sec:2.3.1}

Let $\psi\in\scrS(\rz)$ with $\psi(0)=1$ and $\wh{\psi}(\tau)=0$ whenever $\tau\le0$. 
Let $c>2|\psi'(t)|$ for every $t$ and define 
 $$\lambda^d_-(\xi,\mu)=\Big([\xi',\mu]\psi\Big(\frac{\xi_n}{c[\xi',\mu]}\Big)-i\xi_n\Big)^d,\qquad d\in\gz.$$

\begin{theorem}\label{thm:reduction}
The following affirmations are true$:$
\begin{itemize}
 \item[a$)$] $\op^+(\lambda^{d}_-)(\xi',\mu)$ is an elliptic element of 
  $B^{d;0}(\rz^{n-1}\times\rpbar;(1,0),(1,0))$ for every $d\in\gz$. 
 \item[b$)$] $\op^+(\lambda^{d_1}_-)\#\op^+(\lambda^{d_0}_-)=\op^+(\lambda^{d_0+d_1}_-)$ for all 
  $d_0,d_1\in\gz$.  
 \item[c$)$]  $\op(\op^+(\lambda^{d}_-))(\mu):H^s(\rz^n_+)\to H^{s-d}(\rz^n_+)$ is an isomorphism 
  for every $s\in\rz$ and every $\mu\ge0$. 
 \item[d$)$] If $a\in S^m_\tr(\rz^{n}\times\rpbar)$ has the transmission property then, for all $d\in\gz$,  
  $$\op^+(\lambda^{d}_-)\#\op^+(a)=\op^+(\lambda^{d}_-\#a).$$
\end{itemize}
\end{theorem}

For every $L,M\in\nz_0$ let us define
\begin{align} \label{eq:order-reduction}
 \bflambda^d_-(\xi',\mu)=
    \begin{pmatrix}\op^+(\lambda^{d}_-)(\xi',\mu)\otimes 1_L&0\\0 &[\xi',\mu]^d\otimes 1_M\end{pmatrix},
\end{align}
i.e., the diagonal matrix with $L$ entries $\op^+(\lambda^{d}_-)(\xi',\mu)$ and $M$ entries 
$[\xi',\mu]^d$ $($we shall use the same notation for any choice of $L$ and $M)$. 
Then the analogue of Theorem \ref{thm:reduction} holds true for $\bflambda^d_-$, $d\in\gz$. 

\subsubsection{Ellipticity and parametrix construction}\label{sec:2.3.2}

Let $d\in\gz$ and $d_+:=\max(d,0)$. A symbol $\bfp\in B^{d;d_+}(\rz^{n-1}\times\rpbar;(L,M_0),(L,M_1))$ 
$($note $L=L_0=L_1)$ is called elliptic provided both homogeneous principal symbol and principal boundary 
symbol are pointwise invertible on their domain and
 $$|\sigma_\psi^{(d)}(\bfp)^{-1}(x,\xi,\mu)|\lesssim 1\qquad \forall\;\substack{x\in\rz^n\\x_n\ge0}\quad
    \forall\;|\xi,\mu|=1$$
as well as, for some (and then for all) $s>d_+-1/2$,  
 $$\|\sigma_\partial^{(d)}(\bfp)^{-1}(x',\xi',\mu)\|_{\scrL(H^{s-d}(\rz_+,\cz^{L})\oplus\cz^{M_1},
 H^s(\rz_+,\cz^{L})\oplus\cz^{M_0})}\lesssim 1$$
uniformly for $x'\in\rz^{n-1}$ and $|\xi',\mu|=1$. These conditions assure that $\sigma_\psi^{(d)}(\bfp)^{-1}$ 
can be extended to a symbol 
\begin{align*} 
 \sigma_\psi^{(d)}(\bfp)^{-1}\in S^{-d}_\Hom(\rz^{n}\times\rpbar;\scrL(\cz^L)),
\end{align*}
and that     
\begin{align*} 
 \sigma_\partial^{(d)}(\bfp)^{-1}\in 
 S^{-d}_\Hom(\rz^{n-1}\times\rpbar;H^{s-d}(\rz_+,\cz^{L})\oplus\cz^{M_0},
 H^s(\rz_+,\cz^{L})\oplus\cz^{M_0});
\end{align*}
it can be shown that the invertibility of the principal boundary symbol is independent of the choice of $s>d_+-\frac12$.  

\begin{definition}\label{def:parametrix01}
Let $\bfp\in B^{d;d_+}(\rz^{n-1}\times\rpbar;\frakg)$, $\frakg=((L,M_0),(L,M_1))$. 
A \emph{parametrix} for $\bfp$ is any symbol  
$\bfq\in B^{-d;(-d)_+}(\rz^{n-1}\times\rpbar;\frakg^{-1})$ such that 
\begin{align*} 
 1-\bfq\#\bfp&\in B^{-\infty;d_+}(\rz^{n-1}\times\rpbar;\frakg^{-1}\frakg),\\
 1-\bfp\#\bfq&\in B^{-\infty;(-d)_+}(\rz^{n-1}\times\rpbar;\frakg\frakg^{-1}).
\end{align*}
\end{definition}

The main theorem of parameter-elliptic theory is then the following: 

\begin{theorem}\label{thm:parametrix01}
Let $\bfp\in B^{d;d_+}(\rz^{n-1}\times\rpbar;\frakg)$. Then 
$\bfp$ is elliptic if and only if $\bfp$ has a parametrix. 
In this case there exists a $\mu_0$ such that $\op(\bfp)(\mu)$ is invertible for $\mu\ge\mu_0$ and the 
parametrix $\bfq$ can be chosen in such a way that 
$$\op(\bfq)(\mu)=\op(\bfp)(\mu)^{-1}\qquad\forall\;\mu\ge\mu_0.$$ 
\end{theorem}

Let us also remark that the order reductions of \eqref{eq:order-reduction} allow to reduce the 
study of ellipticity and parametrix to the particular case of symbols of order and type zero:  

\begin{remark}\label{rem:order-reduction}
Let $\bfp\in B^{d;d_+}(\rz^{n-1}\times\rpbar;\frakg)$, $\frakg=((L,M_0),(L,M_1))$. 
If $d\ge0$ then $\bfp\#\bflambda^{-d}_-\in B^{0;0}(\rz^{n-1}\times\rpbar;\frakg)$ 
and $\bfp$ is elliptic if and only if $\bfp\#\bflambda^{-d}_-$ is. 
If $d\le0$ then $\bflambda^{-d}_-\#\bfp\in B^{0;0}(\rz^{n-1}\times\rpbar;\frakg)$ 
and $\bfp$ is elliptic if and only if $\bflambda^{-d}_-\#\bfp$ is. 
\end{remark}

\section{Symbols with expansion at infinity -- the abstract setting}\label{sec:3}

\subsection{Twisted operator-valued symbols revisited}\label{sec:3.1}

Let us return to the abstract setting of Section \ref{sec:2.1.1} of Hilbert spaces equipped with a group-action. 
By slight modifcations we will introduce symbol classes 
 $$\wt{S}^{d,\nu}_{1,0}(\rz^{n-1}\times\rpbar;E_0,E_1),\quad 
    \wt{S}^{d,\nu}_\Hom(\rz^{n-1}\times\rpbar;E_0,E_1),\quad
    \wt{S}^{d,\nu}(\rz^{n-1}\times\rpbar;E_0,E_1)$$
with $d,\nu\in\rz$. We limit ourselves to list here the necessary modifications to be made and refer the reader 
to \cite{Seil22-2} where these classes appear in a related but different context.
\begin{itemize}
 \item In Definition \ref{def:op-symbol1} substitute the estimates \eqref{eq:symbol-estimate01} by 
   \begin{align*}
   \begin{split}
     \|\kappa_{1}^{-1}(\xi',\mu)&\{D^{\alpha'}_{\xi'}D^{\beta'}_{x'}D^j_\mu p(x',\xi',\mu)\}
     \kappa_{0}(\xi',\mu)\|_{\scrL(E_0,E_1)}\\
     &\lesssim \spk{\xi'}^{\nu-|\alpha'|}\spk{\xi',\mu}^{d-\nu-j};
   \end{split}
   \end{align*}
   again $d$ is the order of the symbol while $\nu$ is called the \emph{regularity number}. 
 \item The corresponding classes of regularizing symbols are 
   \begin{align*}
    \wt{S}^{d-\infty,\nu-\infty}&(\rz^{n-1}\times\rpbar;E_0,E_1)\\
    &:=\mathop{\mbox{\Large$\cap$}}_{N>0}S^{d-N,\nu-N}_{1,0}(\rz^{n-1}\times\rpbar;E_0,E_1).
   \end{align*}
  Note that these classes only depend on the difference $d-\nu$.  
 \item In order to define $\wt{S}^{d,\nu}_\Hom(\rz^{n-1}\times\rpbar;E_0,E_1)$, in Definition 
  \ref{def:twisted-homogeneous} let $V=(\rz^{n-1}\setminus\{0\})\times\rpbar$ and rather than 
  \eqref{eq:symbol-estimate02} consider the estimates 
   \begin{align*}
   \begin{split}
     \|\kappa_{1,|\xi',\mu|}^{-1}\{D^{\alpha'}_{\xi'}&D^{\beta'}_{x'}D^j_\mu p(x',\xi',\mu)\}
     \kappa_{0,|\xi',\mu|}\|_{\scrL(E_0,E_1)}\\
     &\lesssim |\xi'|^{\nu-|\alpha'|}|\xi',\mu|^{d-\nu-j}.
   \end{split}
   \end{align*}
 \item For the definition of poly-homogeneous symbols, in Definition \ref{def:poly-homogeneous}
  ask for the existence of homogeneous components 
  $p^{(d-j,\nu-j)}\in \wt{S}^{d-j,\nu-j}_\Hom(\rz^{n-1}\times\rpbar;E_0,E_1)$ such that, 
  for every $N\in\nz_0$, 
    $$p-\sum_{j=0}^{N-1}\chi(\xi')p^{(d-j,\nu-j)}\;\in\; 
    \wt{S}^{d-N,\nu-N}_{1,0}(\rz^{n-1}\times\rpbar;E_0,E_1)$$
  $($note that $\chi=\chi(\xi')$ is a zero-excision function of $\xi'$ only$)$. 
  The leading component $p^{(d,\nu)}$ is the \emph{homogeneous principal symbol} of $p$. It can 
  be recovered from the full symbol as in formula \eqref{eq:formula-principal-symbol}, for $\xi'\not=0$. 
 \item The behaviour under Leibniz product and formal adjoint is analogous to that of Theorems \ref{thm:Leibniz01} 
  and \ref{thm:adjoint01}, respectively; both order and regularity number are additive under composition, 
  the homogeneous principal symbol behaves multiplicatively.  
\end{itemize}

As decribed in the end of Section \ref{sec:2.1.1}, all the above spaces can be equipped with natural Fréchet topologies. 

\begin{remark}
$S^d_{1,0}(\rz^{n-1}\times\rpbar;E_0,E_1)$ is continuously embedded in 
$\wt{S}^{d,0}_{1,0}(\rz^{n-1}\times\rpbar;E_0,E_1)$; analogously for the subclasses of homogeneous and 
poly-homogeneous symbols. 
\end{remark}

\forget{
Let us also remark that 
\begin{align*}
 \wt{S}^{-\infty,\nu}_{1,0}(\rz^{n-1}\times\rpbar;E_0,E_1)
 &=\mathop{\mbox{\Large$\cap$}}_{N>0}\wt S^{-N,\nu}_{1,0}(\rz^{n-1}\times\rpbar;E_0,E_1)\\
 &=S^{-\infty}(\rz^{n-1}\times\rpbar;E_0,E_1)
\end{align*}
for every choice of $\nu$, cf. \eqref{eq:regularizing}. 
}

Again the case of trivial group actions $\kappa_j\equiv 1$ is included; in this case we shall write 
\begin{align}\label{eq:ohne-kappa02}
\begin{split}
 \wt{S}^{d,\nu}_{1,0}(\rz^{n-1}\times\rpbar;\scrL(E_0,E_1))
   &:=\wt{S}^{d,\nu}_{1,0}(\rz^{n-1}\times\rpbar;E_0,E_1), \\ 
 \wt{S}^{d,\nu}_\Hom(\rz^{n-1}\times\rpbar;\scrL(E_0,E_1))
   &:=\wt{S}^{d,\nu}_\Hom(\rz^{n-1}\times\rpbar;E_0,E_1), \\ 
  \wt{S}^{d,\nu}(\rz^{n-1}\times\rpbar;\scrL(E_0,E_1))
  &:=S^{d,\nu}(\rz^{n-1}\times\rpbar;E_0,E_1).
\end{split}
\end{align}

\subsubsection{Symbols with values in Fréchet spaces} \label{sec:3.1.1}

We shall also need a variant of the symbol spaces of Definition \ref{def:ohne-kappa01} and 
\eqref{eq:ohne-kappa02} where the space 
$\scrL(E_0,E_1)$ is substituted by a general Fréchet space $F$ (i.e., the symbols take values in $F)$. 
Let $F$ be equipped with a system of semi-norms $\{\trinorm{\cdot}^{(\ell)}_F\mid \ell\in\nz\}$  

When before we asked symbol estimates with trivial group-action and using the 
operator-norm of $\scrL(E_0,E_1)$, we now require such estimates for every semi-norm 
$\trinorm{\cdot}^{(\ell)}_F$, $\ell\in\nz$. This leads to spaces 
 $$ {S}^{d}_{1,0}(\rz^{n-1}\times\rpbar;F),\quad 
     {S}^{d}_\Hom(\rz^{n-1}\times\rpbar;F),\quad 
      {S}^{d}(\rz^{n-1}\times\rpbar;F),$$
as well as 
 $$ \wt{S}^{d,\nu}_{1,0}(\rz^{n-1}\times\rpbar;F),\quad 
     \wt{S}^{d,\nu}_\Hom(\rz^{n-1}\times\rpbar;F),\quad 
      \wt{S}^{d,\nu}(\rz^{n-1}\times\rpbar;F).$$
These are all Fréchet spaces again. 

\begin{remark}
$S^d_{1,0}(\rz^{n-1}\times\rpbar;F)$ is continuously embedded in 
$\wt{S}^{d,0}_{1,0}(\rz^{n-1}\times\rpbar;F)$; analogously for the subclasses of homogeneous and 
poly-homogeneous symbols. 
\end{remark}

We shall also need classes of symbols not depending on the parameter $\mu$, namely  
\begin{align*}
 S^d_{1,0}(\rz^{n-1};F),\qquad  S^d_\Hom(\rz^{n-1};F),\qquad S^d(\rz^{n-1};F). 
\end{align*}
These are constructed similarly as above, by simply cancelling the parameter $\mu$ in the definitions of 
$S^d_{1,0}(\rz^{n-1}\times\rpbar;F)$, $S^d_\Hom(\rz^{n-1}\times\rpbar;F)$, and 
$S^d(\rz^{n-1}\times\rpbar;F)$, respectively. 

\begin{remark}
If $p(x',\xi')\in S^d_{1,0}(\rz^{n-1};F)$ is considered as a $\mu$-dependent symbol 
then $p\in \wt{S}^{d,d}_{1,0}(\rz^{n-1}\times\rpbar;F)$. 
Analogously in the poly-homogeneous case. 
\end{remark}

\subsection{$F$-valued symbols with expansion at infinity} \label{sec:3.2}

In this section we describe two subclasses of $\wt{S}^{d,\nu}_{1,0}(\rz^{n-1}\times\rpbar;F)$ which 
will be crucial for this paper. In case $F=\scrL(\cz^L,\cz^M)$ these classes have been introduced in 
\cite{Seil22-1} where operators on manifolds without boundary were considered. 
The extension to $F$-valued symbols is straight-forward. 

\begin{definition}\label{def:symbol-expansion}
Let $\wtbfS^{d,\nu}_{1,0}(\rz^{n-1}\times\rpbar;F)$ denote the space of all symbols 
$p(x',\xi',\mu)\in \wt{S}^{d,\nu}_{1,0}(\rz^{n-1}\times\rpbar;F)$ for which 
exist symbols 
 $$p^\infty_{[d,\nu+j]}(x',\xi')\in S^{\nu+j}_{1,0}(\rz^{n-1};F),\qquad j\in\nz_0,$$ 
such that, for every $N\in\nz_0$,   
    $$p(x',\xi',\mu)      -\sum_{j=0}^{N-1}p_{[d,\nu+j]}^\infty(x',\xi')[\xi',\mu]^{d-\nu-j}
       \;\in\; \wt{S}^{d,\nu+N}_{1,0}(\rz^{n-1}\times\rpbar;F).$$ 
The symbol $p^\infty_{[d,\nu]}$ shall be called the \emph{principal limit-symbol} of $p$. 
\end{definition}

The principal limit-symbol can be recovered from the full symbol$:$
\begin{align*}
 p^\infty_{[d,\nu]}(x',\xi')=\lim_{\mu\to+\infty}[\xi',\mu]^{\nu-d}p(x',\xi',\mu)
 \quad\text{(convergence in }S^{\nu+1}(\rz^{n-1};F)). 
\end{align*}
A Fréchet topology on $\wtbfS^{d,\nu}_{1,0}(\rz^{n-1}\times\rpbar;F)$ is given by the projective topology 
under the maps 
 $$p\mapsto p^\infty_{[d,\nu+j]},\qquad p\mapsto p-\sum_{j<N}p^\infty_{[d,\nu+j]}[\xi',\mu]^{d-\nu-j}
 \qquad(j,N\in\nz_0).$$ 

\begin{remark}\label{rem:equivalent-expansion}
In Definition $\ref{def:symbol-expansion}$ we require an expansion in powers of $[\xi',\mu]$. As shown in 
\cite[Theorem 7.7]{Seil22-1} there is a great flexibility in the choice of substitutes for $[\xi',\mu]$ without altering 
the class $\wtbfS^{d,\nu}(\rz^{n-1}\times\rpbar;F)$ itself. For example, we can take $\spk{\xi',\mu}$, i.e., 
$p$ belongs to $\wtbfS^{d,\nu}(\rz^{n-1}\times\rpbar;F)$ if and only if there exist symbols 
$\wt p^\infty_{[d,\nu+j]}\in S^{\nu+j}_{1,0}(\rz^{n-1};F)$ such that  
  $$p(x',\xi',\mu) -\sum_{j=0}^{N-1}\wt p_{[d,\nu+j]}^\infty(x',\xi')\spk{\xi',\mu}^{d-\nu-j}
       \;\in\; \wt{S}^{d,\nu+N}_{1,0}(\rz^{n-1}\times\rpbar;F)$$ 
for every $N\in\nz_0$. In this case even $\wt p_{[d,\nu]}^\infty=p_{[d,\nu]}^\infty$. 
\end{remark}

\subsubsection{A characterization of smoothing symbols}\label{sec:3.2.1}

The class of smoothing symbols with expansion at infinity is  
 $$\wtbfS^{d-\infty,\nu-\infty}(\rz^{n-1}\times\rpbar;F)
    =\mathop{\mbox{\Large$\cap$}}_{N>0}\wtbfS^{d-N,\nu-N}_{1,0}(\rz^{n-1}\times\rpbar;F).$$
Since for $p\in \wtbfS^{d-N,\nu-N}_{1,0}(\rz^{n-1}\times\rpbar;F)$ one has 
 $$p^\infty_{[d,\nu+j]}=p^\infty_{[d-N,\nu-N+j]},\qquad j\ge0, $$
for a smoothing symbol $p$ all symbols $p^\infty_{[d,\nu+j]}$ are smoothing, too. 

In the following proposition, let $S^d_{1,0}(\rpbar;F)$ be the space of all $p\in\scrC^\infty(\rpbar,F)$ with 
$\trinorm{D^j_\mu p(\mu)}^{(\ell)}_F\lesssim\spk{\mu}^{d-j}$ 
for every $j,\ell$. Similarly $S^d(\rpbar;F)$ is the subspace of poly-homogeneous symbols in $\mu$. 

\begin{proposition}\label{prop:smoothing}
For a symbol $p(x',\xi',\mu)$ it holds 
 $$p\in \wtbfS^{d-\infty,\nu-\infty}(\rz^{n-1}\times\rpbar;F) \iff  
    p\in S^{d-\nu}(\rpbar;S^{-\infty}(\rz^{n-1};F)).$$ 
In this case 
 $$p^{(d-\nu)}(\mu)=p^\infty_{[d,\nu]}\mu^{d-\nu}.$$
\end{proposition}
\begin{proof}
It is easy to verify that  
 $$\wt S^{d-\infty,\nu-\infty}(\rz^{n-1}\times\rpbar;F)=S^{d-\nu}_{1,0}(\rpbar;S^{-\infty}(\rz^{n-1};F)).$$
Therefore $($cf. Remark \ref{rem:equivalent-expansion})  
$p\in \wtbfS^{d-\infty,\nu-\infty}(\rz^{n-1}\times\rpbar;F)$ if and only if there 
exist symbols $\wt p_j(x',\xi')\in S^{-\infty}(\rz^{n-1};F)$, with $\wt p_0=p^\infty_{[d,\nu]}$, 
such that, for every $N\in\nz_0$, 
\begin{equation}\label{eq:expA}
  p(x',\xi',\mu)-\sum_{j=0}^{N-1}\wt p_j(x',\xi')\spk{\xi',\mu}^{d-\nu-j}
       \;\in\; S^{d-\nu-N}_{1,0}(\rpbar;S^{-\infty}(\rz^{n-1};F)).
\end{equation}
On the other hand, $p\in S^{d-\nu}(\rpbar;S^{-\infty}(\rz^{n-1};F))$ if and only if there 
exist symbols $q_j(x',\xi')\in S^{-\infty}(\rz^{n-1};F)$ such that, for every $N\in\nz_0$, 
\begin{equation}\label{eq:expB}
  p(x',\xi',\mu)-\sum_{j=0}^{N-1}q_j(x',\xi')\chi(\mu)\mu^{d-\nu-j}
       \;\in\; S^{d-\nu-N}_{1,0}(\rpbar;S^{-\infty}(\rz^{n-1};F)), 
\end{equation}
where $\chi(\mu)$ is an arbitrarily fixed zero-excision function. 
Let us now denote by $a^{(\rho-j)}(\mu)=a_{\rho,j}\mu^{\rho-j}$ the homogeneous components of 
$a_\rho(\mu)=\spk{\mu}^\rho\in S^\rho(\rpbar)$; note that $a_{\rho,0}=1$. 
Since $\spk{\xi',\spk{\xi'}\mu}^\rho=\spk{\xi'}^\rho \spk{\mu}^\rho$, 
\begin{equation*}
 \spk{\xi',\mu}^\rho=\sum_{k=0}^{L-1}\wt\chi(\spk{\xi'}^{-1}\mu)a_{\rho, k}\spk{\xi'}^k \mu^{\rho-k}
    + \spk{\xi'}^\rho\ \wt r_L(\spk{\xi'}^{-1}\mu), \qquad
    \wt r_{L}\in S^{\rho-L}_{1,0}(\rpbar)
\end{equation*}
for every $L\in\nz_0$, where the zero-excision function $\wt\chi$ is chosen in such a way that 
$\chi(\mu)\wt\chi(\spk{\xi'}^{-1}\mu)=\wt\chi(\spk{\xi'}^{-1}\mu)$ for all $\xi'$ and $\mu$. 
By direct calculation one checks that
\begin{align*}
 \spk{\xi'}^\rho \wt r_L(\spk{\xi'}^{-1}\mu) &\in S^{\rho-L}_{1,0}(\rpbar;S^{\max(L,\rho)}(\rz^{n-1})),\\
 (1-\wt\chi)(\spk{\xi'}^{-1}\mu)\chi(\mu) &\in S^{-m}_{1,0}(\rpbar;S^{m}(\rz^{n-1}))\qquad\forall\;m.
\end{align*}
This yields 
\begin{equation*}
 \spk{\xi',\mu}^\rho\equiv \sum_{k=0}^{L-1}\chi(\mu)a_{\rho, k}\spk{\xi'}^k \mu^{\rho-k}
 \mod S^{\rho-L}_{1,0}(\rpbar;S^{\max(L,\rho)}(\rz^{n-1})). 
\end{equation*}
Taking $\rho=d-\nu-j$ and $L=N-j$, we find  
\begin{equation}\label{eq:exp01}
 \spk{\xi',\mu}^{d-\nu-j}=\sum_{k=j}^{N-1}a_{d-\nu-j,k-j}\spk{\xi'}^{k-j} \chi(\mu)\mu^{d-\nu-k}
 +r_j^N,\qquad 0\le j\le N-1,
\end{equation}
with $r_j^N\in S^{d-\nu-N}_{1,0}(\rpbar;S^{\max(N,d-\nu)-j}(\rz^{n-1}))$. Equivalently,
 $$
     \begin{pmatrix}
     \spk{\xi',\mu}^{d-\nu} \\ \spk{\xi',\mu}^{d-\nu-1} \\ \vdots \\ \spk{\xi',\mu}^{d-\nu-(N-1)}
     \end{pmatrix}
     =
     \begin{pmatrix}
     \alpha_{00}&\alpha_{01}&\cdots&\alpha_{0,N-1}\\
             0        &\alpha_{11}&\cdots&\alpha_{1,N-1}\\
           \vdots   & \ddots        &\ddots&\vdots\\
             0        &   \cdots      &     0   &\alpha_{N-1,N-1}
     \end{pmatrix}
     \begin{pmatrix}
     \chi(\mu)\mu^{d-\nu} \\ \chi(\mu)\mu^{d-\nu-1} \\ \vdots \\ \chi(\mu)\mu^{d-\nu-(N-1)}
     \end{pmatrix}
     +
     \begin{pmatrix}
     r_0^N \\ r_1^N \\ \vdots \\ r_{N-1}^{N}
     \end{pmatrix}
$$
with $\alpha_{jk}(\xi')=a_{d-\nu-j,k-j}\spk{\xi'}^{k-j}$ for $j\le k$, where $\alpha_{jj}=1$ for all $j$. 
Since the inverse $(\beta_{jk})$ of the upper-right triangular matrix $(\alpha_{jk})$ has the same structure, i.e., 
$\beta_{jk}(\xi')=b_{jk}\spk{\xi'}^{k-j}$ for $j\le k$ for certain numbers $b_{jk}$ where $b_{jj}=1$ for all $j$, 
the latter system, hence \eqref{eq:exp01}, is equivalent to 
\begin{equation}\label{eq:exp02}
 \chi(\mu)\mu^{d-\nu-j}=\sum_{k=j}^{N-1}b_{jk}\spk{\xi'}^{k-j} \spk{\xi',\mu}^{d-\nu-k}
 +\wt r_j^N,\qquad 0\le j\le N-1,
\end{equation}
with $\wt r_j^N\in S^{d-\nu-N}_{1,0}(\rpbar;S^{\max(N,d-\nu)-j}(\rz^{n-1}))$. 

Clearly \eqref{eq:exp01} and \eqref{eq:exp02} imply that we can transform an expansion \eqref{eq:expA} 
into an expansion \eqref{eq:expB} and vice versa, without changing the coefficient of the leading term. 
This yields the claim. 
\end{proof}

\subsubsection{The subclass of poly-homogeneous symbols}\label{sec:3.2.2}

To define the subclass of poly-homogeneous symbols we need some further preparation. First of all let 
\begin{align*}
\begin{split}
 \sz^{n-1}_+&=\{(\xi',\mu)\in\rz^{n-1}\times\rpbar\mid |\xi'|^2+\mu^2=1\},\\
 \whsz^{n-1}_+&=\sz^{n-1}_+\setminus\{(0,1)\}
\end{split}
\end{align*}
be the unit semi-sphere and the punctured unit semi-sphere, respectively. 
The latter shall be identified with $(0,1]\times\sz^{n-2}$ via 
 $$(r,\phi)\mapsto (r\phi,\sqrt{1-r^2}):(0,1]\times\sz^{n-2}\lra \whsz^{n-1}_+.$$
For a Fréchet space $F$ we let $\scrC^\infty_B((0,1];F)$ be the space of all smooth functions $u$ on 
$(0,1]$ with values in $F$ such that $(r\partial_r)^j u$ is bounded on $(0,1]$ for every $j\in\nz_0$. 

\begin{definition}\label{def:singularity}
$r^\nu\scrC^\infty_B(\whsz^{n-1};F)$ consists of all functions 
$\wh{p}\in\scrC^\infty(\whsz^{n-1};F)$ such that 
$r^{-\nu}\wh{p}(r,\phi)$ belongs to $\scrC^\infty_B((0,1],\scrC^\infty(\sz^{n-2};F))$. 
The subspace of all $\wh p$ such that $r^{-\nu}\wh{p}(r,\phi)$ 
extends to a function in $\scrC^\infty([0,1]\times\sz^{n-2};F)$ is denoted by 
$r^\nu\scrC^\infty_T(\whsz^{n-1};F)$.
\end{definition}

These spaces again are Fréchet in a natural way, hence it makes also sense to consider 
$\scrC^\infty_b(\rz^{n-1}_{x'},r^\nu\scrC^\infty_{B/T}(\whsz^{n-1};F))$ including an additional variable $x'$ 
with $\scrC^\infty_b$-dependence (i.e., all partial derivatives with $x'$ remain bounded functions in $x'$). 
The subscript $T$ in the definition indicates the existence of a Taylor expansion in $r$ centered in $r=0$. 

\begin{theorem}
The restriction to $\rz^{n-1}\times\whsz^{n-1}$ of functions $p(x',\xi',\mu)$ defined on 
$\rz^{n-1}\times(\rz^{n-1}\setminus\{0\})\times\rpbar$ 
yields isomorphisms 
 $$\wt{S}^{d,\nu}_\Hom(\rz^{n-1}\times\rpbar;F)\cong 
     \scrC^\infty_b\big(\rz^{n-1},r^\nu\scrC^\infty_{B}(\whsz^{n-1};F)\big);$$
the inverse map is given by homogeneous extension of degree $d$, i.e., 
 $$p(x',\xi',\mu)=|\xi',\mu|^d\, \wh p\Big(x',\frac{(\xi',\mu)}{|\xi',\mu|}\Big),\qquad \xi'\not=0.$$
\end{theorem}

\begin{definition}\label{def:hom}
$\wtbfS^{d,\nu}_\Hom(\rz^{n-1}\times\rpbar;F)$ consists of all smooth $F$-valued functions 
$p(x',\xi',\mu)$ which are homogeneous in $(\xi',\mu)$ of degree $d$ on $\rz^{n-1}\times(\rz^{n-1}\setminus\{0\})\times\rpbar$  
and have the property 
 $$p\big|_{\rz^{n-1}\times\whsz^{n-1}_+}\in 
     \scrC^\infty_b\big(\rz^{n-1}_{x'},r^\nu\scrC^\infty_T(\whsz^{n-1};F)\big).$$
The so-called \emph{principal angular symbol} of $p\in\wtbfS^{d,\nu}_\Hom(\rz^{n-1}\times\rpbar;F)$ is 
 $$p_{\spk{\nu}}(x',\xi'):=|\xi'|^{\nu} \lim_{r\to0+}r^{-\nu}p\Big(x',r\frac{\xi'}{|\xi'|},\sqrt{1-r^2}\Big)
     \;\in\; S^{\nu}_\Hom(\rz^{n-1};F).$$
\end{definition}

The principal angular symbol is determined by the leading term in the Taylor expansion. 
By construction, $\wtbfS^{d,\nu}_\Hom(\rz^{n-1}\times\rpbar;F)\subset 
     \wt{S}^{d,\nu}_\Hom(\rz^{n-1}\times\rpbar;F)$. Moreover   
 $$S^{d}_\Hom(\rz^{n-1}\times\rpbar;F)\subset 
 \wtbfS^{d,0}_\Hom(\rz^{n-1}\times\rpbar;F),$$
since $\scrC^\infty(\sz^{n-1};F)\subset \scrC^\infty_{T}(\whsz^{n-1};F)$. 
 
\begin{proposition}\label{prop:excision}
Let $p(x',\xi',\mu)\in \wtbfS^{d,\nu}_\Hom(\rz^{n-1}\times\rpbar;F)$ and $\chi(\xi')$ be a zero-excision 
function. Then $\chi(\xi')p(x',\xi',\mu)$ belongs to $\wtbfS^{d,\nu}_{1,0}(\rz^{n-1}\times\rpbar;F)$ 
and has principal limit-symbol
 $$(\chi p)_{[d,\nu]}^\infty(x',\xi')=\chi(\xi')p_{\spk{\nu}}(x',\xi').$$
\end{proposition}

Due to the latter proposition the following definition makes sense$:$

\begin{definition}\label{def:angular-symbol02}
A symbol $p$ belongs to the space 
$\wtbfS^{d,\nu}(\rz^{n-1}\times\rpbar;F)$ provided there exist homogeneous components 
$p^{(d-j,\nu-j)}\in \wtbfS^{d-j,\nu-j}_\Hom(\rz^{n-1}\times\rpbar;F)$ such that, for every $N\in\nz_0$, 
    $$p(x',\xi',\mu)-\sum_{j=0}^{N-1}\chi(\xi')p^{(d-j,\nu-j)}(x',\xi',\mu)
       \;\in\; \wtbfS^{d-N,\nu-N}_{1,0}(\rz^{n-1}\times\rpbar;F).$$ 
By definition, the \emph{principal angular symbol} of $p$ is that of its homogeneous principal symbol, i.e., 
$p_{\spk{d,\nu}}(x',\xi'):=p^{(d,\nu)}_{\spk{\nu}}(x',\xi')$. 
\end{definition}

Again, $\wtbfS^{d,\nu}(\rz^{n-1}\times\rpbar;F)$ carries a natural projective topology which makes it a 
Fréchet space. By construction, the principal angular symbol of $p\in \wtbfS^{d,\nu}(\rz^{n-1}\times\rpbar;F)$ satisfies 
  $$p_{\spk{d,\nu}}(x',\xi'):=|\xi'|^{\nu} \lim_{r\to0+}r^{-\nu}p^{(d.\nu)}\Big(x',r\frac{\xi'}{|\xi'|},\sqrt{1-r^2}\Big)
     \;\in\; S^{\nu}_\Hom(\rz^{n-1};F).$$
The principal angular symbol is a ``subordinate'' or ``secondary'' principal symbol, since it is 
uniquely determined by the homogeneous principal symbol. However, it plays an important role in the calculus, 
cf. Section \ref{sec:3.3}. 

\begin{remark}\label{rem:limit-angular}
Let $p\in\wtbfS^{d,\nu}(\rz^{n-1}\times\rpbar;F)$ as in Definition $\ref{def:angular-symbol02}$. Due to 
Proposition $\ref{prop:excision}$, 
 $$p^\infty_{[d,\nu]}(x',\xi')-\sum_{j=0}^{N-1}\chi(\xi')p^{(d-j,\nu-j)}_{\spk{\nu-j}}(x',\xi')
    \in S^{\nu-N}(\rz^{n-1};F).$$
Thus the principal limit symbol $p^\infty_{[d,\nu]}$ is poly-homogeneous and its homogeneous principal symbol 
coincides with the principal angular symbol of $p$, 
\begin{align*}
 p_{\spk{d,\nu}}(x',\xi')=(p^\infty_{[d,\nu]})^{(\nu)}(x',\xi').
\end{align*}
\end{remark}

The following result was already used in \cite{Seil22-1} (in case $F=\cz)$ but without providing an explicit proof.  
For convenience of the reader we give a proof here.  

\begin{proposition}\label{prop: classical}
Let $p(x',\xi',\mu)\in S^{d}(\rz^{n-1}\times\rpbar;F)$. Then $p\in \wtbfS^{d,0}(\rz^{n-1}\times\rpbar;F)$ and 
$p$ has limit symbol 
 $$p^\infty_{[d,0]}(x',\xi')=p^{(d)}(x',0,1),$$
i.e., the homogeneous principal symbol evaluated in $(\xi',\mu)=(0,1)$. 
\end{proposition}
\begin{proof}
By \cite[Proposition 5.3]{Seil22-1}, $p\in\wtbfS^{d,0}_{1,0}(\rz^{n-1}\times\rpbar;F)$ 
and the principal limit symbol is as stated. We have to show that $p$ belongs to the poly-homogeneous class. 

Given a zero-excision function $\chi(\xi')$ let us choose a zero-excision function $\wt\chi(\xi',\mu)$ such that 
$\chi\wt\chi=\chi$. Then 
 $$p-\sum_{j=0}^{N-1}\chi p^{(d-j)}=\chi\Big(p-\sum_{j=0}^{N-1}\wt\chi p^{(d-j)}\Big)+(1-\chi)p.$$
Since $(1-\chi)\in \wtbfS^{-N,-N}_{1,0}(\rz^{n-1}\times\rpbar)$ and $p\in \wtbfS^{d,0}_{1,0}(\rz^{n-1}\times\rpbar;F)$, 
the second term on the right-hand side belongs to $\wtbfS^{d-N,-N}_{1,0}(\rz^{n-1}\times\rpbar;F)$. The first term 
is of the form $\chi p_N$, where 
$p_N\in S^{d-N}(\rz^{n-1}\times\rpbar;F)\subset \wtbfS^{d-N,0}_{1,0}(\rz^{n-1}\times\rpbar;F)$ and 
$\chi\in \wtbfS^{0,0}_{1,0}(\rz^{n-1}\times\rpbar)$. Thus, to complete the proof, it remains to show that 
 $$\wtbfS^{d-N,0}_{1,0}(\rz^{n-1}\times\rpbar;F)\subset \wtbfS^{d-N,-N}_{1,0}(\rz^{n-1}\times\rpbar;F).$$
For the latter rename $d-N$ by $d$ and let $a\in \wtbfS^{d,0}_{1,0}(\rz^{n-1}\times\rpbar;F)$. We will show that 
$a$ belongs to $\wtbfS^{d,-L}_{1,0}(\rz^{n-1}\times\rpbar;F)$ for arbitrary $L\in\nz_0$. This is equivalent to showing 
that $\wt a:=\spk{\xi'}^La\in \wtbfS^{d+L,0}_{1,0}(\rz^{n-1}\times\rpbar;F)$. We define 
$\wt a_{[0]}=\ldots=\wt a_{[L-1]}=0$ and $\wt a_{[L+j]}=\spk{\xi'}^L a_{[j]}$ for $j\ge 0$. Then, for $N\le L$, 
 $$\wt a-\sum_{j=0}^{N-1}\wt a_{[j]}[\xi',\mu]^{d+L-j}=\spk{\xi'}^La\in 
    \wt S^{d+L,L}_{1,0}(\rz^{n-1}\times\rpbar;F)\subset \wt S^{d+L,N}_{1,0}(\rz^{n-1}\times\rpbar;F),$$
since $\wt S^{d,\nu_1}_{1,0}(\rz^{n-1}\times\rpbar;F)\subset \wt S^{d,\nu_0}_{1,0}(\rz^{n-1}\times\rpbar;F)$ 
whenever $\nu_0\le\nu_1$. If $N>L$, 
 $$\wt a-\sum_{j=0}^{N-1}\wt a_{[j]}[\xi',\mu]^{d+L-j}=\spk{\xi'}^L
     \Big(a-\sum_{j=0}^{N-L-1}a_{[j]}[\xi',\mu]^{d-j}\Big)
     \in\wt S^{d+L,N}_{1,0}(\rz^{n-1}\times\rpbar;F),$$
since $\spk{\xi'}^L\wt S^{d,N-L}_{1,0}(\rz^{n-1}\times\rpbar;F)=\wt S^{d+L,N}_{1,0}(\rz^{n-1}\times\rpbar;F)$. 
\end{proof}
 
\begin{remark}
If $p(x',\xi')\in S^d_{1,0}(\rz^{n-1};F)$ is considered as a $\mu$-dependent symbol then 
$p\in\wtbfS^{d,d}_{1,0}(\rz^{n-1}\times\rpbar;F)$. Analogously in the poly-homogeneous case. 
\end{remark}

\subsection{A result on the invertibility of homogeneous components}\label{sec:3.3}

For later purpose we shall discuss in this section a specific result concerning the invertibility of homogeneous symbols. 
It explains the importance of the principal angular symbol. 

\begin{lemma}\label{lem:invertibility01}
Let $p\in \wtbfS^{d,\nu}_\Hom(\rz^{n-1}\times\rpbar;\scrL(E_0,E_1))$. Assume 
\begin{itemize}
 \item[$(1)$] $p(x',\xi',\mu)$ is invertible whenever $\xi'\not=0$ and 
  $$\|p(x',\xi',\mu)^{-1}|\xi'|^\nu\|_{\scrL(E_1,E_0)}\lesssim 1
      \qquad \forall\;x'\in\rz^{n-1}\quad\forall\;(\xi',\mu)\in\whsz^n_+,$$
 \item[$(2)$] The principal angular symbol $p_{\spk{\nu}}(x',\xi')$ is invertible whenever $\xi'\not=0$ and 
  $$\|p_{\spk{\nu}}(x',\xi')^{-1}\|_{\scrL(E_1,E_0)}\lesssim 1\qquad \forall\;x'\in\rz^{n-1}\quad\forall\;|\xi'|=1.$$
\end{itemize}
Then $p^{-1}\in \wtbfS^{-d,-\nu}_\Hom(\rz^{n-1}\times\rpbar;\scrL(E_1,E_0))$. 
\end{lemma}
\begin{proof}
By multiplication with $|\xi'|^{-\nu}|\xi',\mu|^\nu$ we may assume without loss of generality that $\nu=0$. 
Using the notation of Definition \ref{def:singularity} and the subsequent paragraph, $\wh p(x',r,\phi)$ is a smooth 
function on $\rz^{n-1}\times[0,1]\times\sz^{n-2}$ and 
 $$\lim_{r\to 0}\wh p(x',\phi',r)=p_{\spk{\nu}}(x',\phi'),$$
cf. Definition \ref{def:hom} and the subsequent paragraph. 
Thus $(1)$ and $(2)$ together are equivalent to the pointwise invertibility of $\wh{p}$ together with a uniform estimate 
of $\wh{p}^{-1}$ on $\rz^{n-1}\times \sz^{n-2}\times[0,1]$. Therefore, by chain rule, also 
$\wh p^{-1}$ belongs to $\scrC^\infty(\rz^{n-1}\times[0,1]\times\sz^{n-2})$. Homogeneous extension of degree 
$-d$ yields the claim. 
\end{proof}

\begin{remark}\label{rem:uniform-ellipticity}
The reasoning in the proof of Lemma $\ref{lem:invertibility01}$ shows also the following$:$   
If $p(x',\xi',\mu)$ is constant in $x'$ for $|x'|\ge C$ for some $C\ge0$, conditions $(1)$ and $(2)$ are 
equivalent to 
\begin{itemize}
 \item[$(1')$] $p(x',\xi,\mu)$ is invertible for all $x'\in\rz^{n-1}$ and $(\xi',\mu)\in\whsz^{n-1}$, 
 \item[$(2')$] The principal angular symbol $p_{\spk{\nu}}(x',\xi')$ is invertible whenever $|\xi'|=1$. 
\end{itemize}
\end{remark}

We shall need a version of the previous result, concerning the related space  
\begin{align}\label{eq:sum-of-hom}
 \bfS^{d,\nu}_\Hom(\rz^{n-1}\times\rpbar;F)
   :=S^d_\Hom(\rz^{n-1}\times\rpbar;F)+\wtbfS^{d,\nu}_\Hom(\rz^{n-1}\times\rpbar;F)
\end{align}
where $\nu\in\nz_0$. Note that 
 $$\bfS^{d,0}_\Hom(\rz^{n-1}\times\rpbar;F)=
    \wtbfS^{d,0}_\Hom(\rz^{n-1}\times\rpbar;F)$$
and 
 $$\bfS^{d,\nu}_\Hom(\rz^{n-1}\times\rpbar;F)\subset
    \wtbfS^{d,0}_\Hom(\rz^{n-1}\times\rpbar;F),\qquad \nu\ge1.$$
In any case we can associate with $p\in \bfS^{d,\nu}_\Hom(\rz^{n-1}\times\rpbar;F)$ its principal angular symbol 
$p_{\spk{0}}\in S^{0}_\Hom(\rz^{n-1};F)$. Indeed, if $p=p_0+\wt p$ 
with $p_0\in S^d_\Hom(\rz^{n-1}\times\rpbar;F)$ and $\wt p\in\wtbfS^{d,\nu}_\Hom(\rz^{n-1}\times\rpbar;F)$, then 
 $$p_{\spk{0}}(x',\xi')=
     \begin{cases}
      p_0(x',0,1)+\wt p_{\spk{0}}(x',\xi')&:\nu=0\\
      p_0(x',0,1)&:\nu\ge1
    \end{cases}.
 $$
 
\begin{lemma}\label{lem:invertibility02}
Let $p\in \bfS^{d,\nu}_\Hom(\rz^{n-1}\times\rpbar;\scrL(E_0,E_1))$, $\nu\in\nz_0$. Assume 
\begin{itemize}
 \item[i$)$] $p(x',\xi',\mu)$ is invertible whenever $\xi'\not=0$ and 
  $$\|p(x',\xi',\mu)^{-1}\|_{\scrL(E_1,E_0)}\lesssim 1\qquad \forall\;x'\in\rz^{n-1}\quad\forall\;(\xi',\mu)\in\whsz^n_+,$$
 \item[ii$)$] The principal angular symbol $p_{\spk{0}}(x',\xi')$ is invertible whenever $\xi'\not=0$ and 
  $$\|p_{\spk{0}}(x',\xi')^{-1}\|_{\scrL(E_1,E_0)}\lesssim 1\qquad \forall\;x'\in\rz^{n-1}\quad\forall\;|\xi'|=1.$$
\end{itemize}
Then $p^{-1}\in \bfS^{-d,\nu}_\Hom(\rz^{n-1}\times\rpbar;\scrL(E_1,E_0))$ $($note the identical regularity number$)$. 
\end{lemma}
\begin{proof}
In case $\nu=0$ this is a special case of Lemma \ref{lem:invertibility01}. We thus assume $\nu\ge1$. 
By multiplication with $|\xi',\mu|^{-d}$ we may assume without loss of generality that $d=0$. 
On the one hand it follows from \cite[Proposition 3.8]{Seil22-2} that $p^{-1}$ belongs to 
$S^0_\Hom(\rz^{n-1}\times\rpbar;\scrL(E_1,E_0))+\wt S^{0,\nu}_\Hom(\rz^{n-1}\times\rpbar;\scrL(E_1,E_0))$, 
hence can be written as $p^{-1}=p_0+\wt p$. On the other hand, by Lemma \ref{lem:invertibility01}, 
$p^{-1}$ belongs to $\wtbfS^{0,0}_\Hom(\rz^{n-1}\times\rpbar;\scrL(E_1,E_0))$, say $p^{-1}=\wt q$. 
Then  
 $$\wt p=\wt q-p_0\in \wtbfS^{0,0}_\Hom(\rz^{n-1}\times\rpbar;\scrL(E_1,E_0))\cap 
    \wt S^{0,\nu}_\Hom(\rz^{n-1}\times\rpbar;\scrL(E_1,E_0)).$$
The latter space coincides with $\wtbfS^{0,\nu}_\Hom(\rz^{n-1}\times\rpbar;\scrL(E_1,E_0))$.  
\end{proof}

The analogue of Remark $\ref{rem:uniform-ellipticity}$ holds for Lemma $\ref{lem:invertibility02}$.

\forget{
\subsubsection{Nuclearity and tensor product representation}

The following observation is very useful from a technical point of view, since in many situations it will 
allow to reduce the analysis of general symbols to the particular case of symbols having a certain ``product form''. 

For clarity let us now use a subscript $\mathrm{const}$ when we consider the subspaces of the above 
introduced symbol spaces, which consist in those symbols which do not depend on the $x'$-variable.  

\begin{proposition}
Both $S^d_{\mathrm{const}}(\rz^{n-1}\times\rpbar)$ and 
$\wtbfS^{d,\nu}_{\mathrm{const}}(\rz^{n-1}\times\rpbar)$ are nuclear Fréchet spaces. 
\end{proposition}
\begin{proof}
CHECK
\end{proof}

\begin{proposition}\label{prop:tensor}
If $F$ coincides with $\scrS(\rz_+)$ or $\scrS(\rz_+\times\rz_+)$ then 
 $$\frakX(\rz^{n-1}\times\rpbar;F)=\frakX(\rz^{n-1}\times\rpbar)\,\wh\otimes_\pi\,F.$$
\end{proposition}

In particular, due to a classical result on tensor products, see \cite{aaa} for example, for every 
$p\in \frakX(\rz^{n-1}\times\rpbar;F)$ there exist zero-sequences 
$(p_\ell)\subset \frakX(\rz^{n-1}\times\rpbar)$ and 
$(f_\ell)\subset F$ and an abolutely summable numerical sequence $(\lambda_\ell)\subset\cz$ such that 
 $$p=\sum_{\ell=0}^{+\infty}\lambda_\ell\,p_\ell\,f_\ell.$$
}

\section{Singular Green symbols with expansion at infinity} \label{sec:4}

\subsection{Singular Green symbols with finite regularity number} \label{sec:4.1}

In Section \ref{sec:2.1.3} we have introduced the classes 
$B^{d;0}_G(\rz^{n-1}\times\rpbar;\frakg)$ of strongly parameter-dependent 
generalized singular Green symbols of type zero as a realization of the abstract classes 
$S^d(\rz^{n-1}\times\rpbar;E_0,E_1)$, cf. Definition \ref{def:singular-green01}. 
From this resulted symbols of positive type $r\in\nz$ in Definition  \ref{def:singular-green02}. 

We proceed in the analogous way for Green symbols with finite regularity number, based on the symbol classes 
$\wt S^{d,\nu}(\rz^{n-1}\times\rpbar;E_0,E_1)$ from Section \ref{sec:3.1}. Again we shall use the dilation 
group-action as explained in the beginning of Section \ref{sec:2.1.2} and \eqref{eq:kappa-matrix}. 

\begin{definition}\label{def:singular-green03}
$\wt B^{d,\nu;0}_G(\rz^{n-1}\times\rpbar;\frakg)$, $\frakg=((L_0,M_0),(L_1,M_1))$, denotes the space 
 $$\mathop{\mbox{\Large$\cap$}}_{s,s^\prime,\delta,\delta^\prime\in\rz} 
    \wt S^{d,\nu}(\rz^{n-1}\times\rpbar;H^{s,\delta}_0(\rpbar,\cz^{L_0})\oplus\cz^{M_0},
    H^{s^\prime,\delta^\prime}(\rz_+,\cz^{L_1})\oplus\cz^{M_1}).$$ 
Moreover, $\wt B^{d,\nu;r}_G(\rz^{n-1}\times\rpbar;\frakg)$ 
denotes the space of all symbols 
 $$\bfg=\bfg_0+\sum_{j=1}^{r}\bfg_j\boldsymbol{\partial}_+^j,\qquad 
   \bfg_j\in \wt B^{d-j,\nu;0}_G((\rz^{n-1}\times\rpbar;\frakg).$$
\end{definition}

$\bfg\in \wt B^{d,\nu;r}_G(\rz^{n-1}\times\rpbar;\frakg)$ has its homogeneous 
principal symbol  
\begin{align}\label{eq:symbol-positive-type-02}
 \bfg^{(d,\nu)}(x',\xi',\mu)=\bfg_0^{(d,\nu)}
 +\sum_{j=1}^{r}\bfg_j^{(d-j,\nu)}(x',\xi',\mu)\boldsymbol{\partial}_+^j
\end{align}
which is defined whenever $\xi'\not=0$. Moreover, $\op(\bfg)(\mu)$ induces maps 
as in \eqref{eq:mapping-property01}. 

Le us remark that in Definition \ref{def:singular-green03} we employ a notation 
already used in \cite{Seil22-2} in a different context; actually the respective spaces are different 
$($that of \cite{Seil22-2} is larger$)$.

\subsection{From operator-valued symbols to symbol-kernels}\label{sec:4.2}

Generalized singular Green symbols $\bfg$ of type zero from each of the above introduced classes are functions of $(x',\xi',\mu)$ 
taking values in $\Gamma^0_G(\rz_+;\frakg)$, $\frakg=((L_0,M_0),(L_1,M_1))$, the space of type zero generalized singular 
Green operators on $\rz_+$, cf. Remark $\ref{rem:sigular-green-half-axis}$. 
We can characterize these operators as follows: 
\begin{itemize}
 \item[a$)$] Poisson operators: $\cap_{s',\delta'}\scrL(\cz^{M_0},H^{s^\prime,\delta^\prime}(\rz_+,\cz^{L_1}))$ 
  is the space of all operators of the form 
   $$c\mapsto \varphi c,\qquad \varphi \in\scrS(\rz_+,\cz^{L_1\times M_0}),$$ 
 \item[b$)$] Trace operators of type zero: $\cap_{s,\delta}\scrL(H^{s,\delta}_0(\rz_+,\cz^{L_0}),\cz^{M_1})$  
  is the space of all operators of the form 
   $$u\mapsto \int_0^{+\infty}\varphi(y_n)u(y_n)\,dy_n,\qquad \varphi\in\scrS(\rz_+,\cz^{M_1\times L_0}),$$ 
 \item[c$)$] Green operators of type zero: 
  $\cap_{s,\delta,s',\delta'}\scrL(H^{s,\delta}_0(\rz_+,\cz^{L_0}),H^{s^\prime,\delta^\prime}(\rz_+,\cz^{L_1}))$  
  is the space of all operators of the form 
   $$u\mapsto \int_0^{+\infty}\psi(\cdot,y_n)u(y_n)\,dy_n,\qquad 
       \psi\in\scrS(\rz_+\times\rz_+,\cz^{L_1\times L_0}),$$ 
\end{itemize}
where $\cz^{k\times\ell}:=\scrL(\cz^\ell,\cz^k)$ is (isomorphic to) the space of complex 
$(k\times\ell)$-matrices. Thus if we represent a generalized singular Green symbol $\bfg$ of type zero in block-matrix form 
\eqref{eq:block-matrix} and use $\scrS(\rz^n_+)\cong\scrS(\rz^{n-1},\scrS(\rz_+))$ then 
the Poisson operator can be written as 
\begin{equation*}
 [\op(k)(\mu)u](x)
 =\int e^{ix'\xi'}\wt k(x',\xi',\mu;x_n)\wh{u}(\xi')\,\dbar\xi', \qquad u\in\scrS(\rz^{n-1},\cz^{M_0}), 
\end{equation*} 
the trace operator can be written as 
\begin{equation*}
 [\op(t)(\mu)u](x')
 =\iint_0^{+\infty}e^{ix'\xi'}\wt t(x',\xi',\mu;y_n)\wh{u}(\xi',y_n)\,dy_n\dbar\xi', 
 \quad u\in\scrS(\rz^n_+,\cz^{L_0}),  
\end{equation*} 
the Green operator can be written as 
\begin{equation*}
 [\op(g)(\mu)u](x)
 =\iint_0^{+\infty}e^{ix'\xi'}\wt g(x',\xi',\mu;x_n,y_n)\wh{u}(\xi',y_n)\,dy_n\dbar\xi', 
 \quad u\in\scrS(\rz^n_+,\cz^{L_0}),   
\end{equation*} 
where $x=(x',x_n)\in\rz^{n}_+$ and $\wh{u}$ indicates the $($partial$)$ Fourier transform in the 
$x'$-variable. The involved so-called \emph{symbol-kernels} are functions of $(x',\xi',\mu)$ and take values in 
the spaces of rapidly decreasing functions as indicated above in a), b), and c). In order to simplify notation we 
shall identify symbols and symbol-kernels, i.e, also denote symbol kernels by $k$, $t$, and $g$ rather than 
$\wt k$, $\wt t$, and $\wt g$, respectively.

\begin{theorem}\label{thm:conjugation01}
Let $\frakg=(L_0,M_0),(L_1,M_1))$ and $\bfg(x',\xi',\mu)$ be smooth with values in 
$\scrL(L^2(\rz_+,\cz^{L_0})\oplus\cz^{M_0},L^2(\rz_+,\cz^{L_1})\oplus\cz^{M_1}))$. 
The following are equivalent: 
\begin{itemize}
 \item[a$)$]  $\bfg\in B^{d;0}_G(\rz^{n-1}\times\rpbar;\frakg)$. 
 \item[b$)$]  If $\bfg^*=\bfg(x',\xi',\mu)^*$ is the pointwise adjoint of $\bfg$, then
 \begin{align*} 
  \bfg&\in\mathop{\mbox{\Large$\cap$}}_{s^\prime,\delta^\prime\in\rz} 
      S^{d}(\rz^{n-1}\times\rpbar;L^2(\rz_+,\cz^{L_0})\oplus\cz^{M_0},
    H^{s^\prime,\delta^\prime}(\rz_+,\cz^{L_1})\oplus\cz^{M_1}),\\
  \bfg^*&\in\mathop{\mbox{\Large$\cap$}}_{s^\prime,\delta^\prime\in\rz} 
      S^{d}(\rz^{n-1}\times\rpbar;L^2(\rz_+,\cz^{L_1})\oplus\cz^{M_1},
    H^{s^\prime,\delta^\prime}(\rz_+,\cz^{L_0})\oplus\cz^{M_0}).
 \end{align*}
 \item[c$)$] $\bfg':=\bfkappa^{-1}\bfg\bfkappa\in S^{d}(\rz^{n-1}\times\rpbar;
  \Gamma^0_G(\rz_+;\frakg))$. 
\end{itemize}
For c$)$ recall the notation from $\eqref{eq:kappa-matrix}$. 
The analogous result is true for symbols with finite regularity number, i.e., replacing $B^{d;0}_G$ by 
$\wt B^{d,\nu;0}_G$ and $S^d$ by $\wt S^{d,\nu}$, respectively. 
\end{theorem}

Part c$)$ of the previous theorem can be reformulated on the level of symbol-kernels. Write 
$\bfg':=\begin{pmatrix}g'&k'\\t'&q'\end{pmatrix}$ and let $g'$, $k'$, and $t'$ denote $($also$)$ the associated 
symbol-kernels. If $g$, $k$, and $t$ are the symbol-kernels associated with $\bfg$, then by direct calculation one 
finds that $\bfg'=\bfkappa^{-1}\bfg\bfkappa$ is equivalent to $q=q'$ and 
\begin{align}\label{eq:symbol-kernel-relation}
\begin{split}
 g(x',\xi',\mu;x_n,y_n)&=[\xi',\mu]g'(x',\xi',\mu;[\xi',\mu]x_n,[\xi',\mu]y_n),\\
 k(x',\xi',\mu;x_n)&=[\xi',\mu]^{1/2}k'(x',\xi',\mu;[\xi',\mu]x_n),\\
 t(x',\xi',\mu;y_n)&=[\xi',\mu]^{1/2}t'(x',\xi',\mu;[\xi',\mu]y_n). 
\end{split}
\end{align}

\begin{theorem}
For $g(x',\xi',\mu;x_n,y_n)$ the following are equivalent: 
\begin{itemize}
 \item[a$)$]  $g$ has the form \eqref{eq:symbol-kernel-relation} with 
  $g'\in S^d_{1,0}(\rz^{n-1}\times\rpbar;\scrS(\rz_+\times\rz_+,\cz^{L_1\times L_0}))$. 
 \item[b$)$] For every choice of $($multi-$)$indices, $g$ satisfies the estimates 
  \begin{align*}
   \|x_n^\ell D_{x_n}^{\ell'}y_n^m D_{y_n}^{m'}
   &D^j_\mu D^{\alpha'}_{\xi'}D^{\beta'}_{x'} g(x',\xi',\mu;x_n,y_n)\|_{L^2(\rz_{+,x_n}\times\rz_{+,y_n})}\\
   &\lesssim [\xi',\mu]^{d-|\alpha'|-j+\ell'-\ell+m'-m}.
  \end{align*}
 \item[c$)$] For every choice of $($multi-$)$indices and arbitrary $1\le p\le+\infty$, $g$ satisfies the estimates 
  \begin{align*}
   \|x_n^\ell D_{x_n}^{\ell'}y_n^m D_{y_n}^{m'}
   &D^j_\mu D^{\alpha'}_{\xi'}D^{\beta'}_{x'} g(x',\xi',\mu;x_n,y_n)\|_{L^p(\rz_{+,x_n}\times\rz_{+,y_n})}\\
   &\lesssim [\xi',\mu]^{d+1-\frac{2}{p}-|\alpha'|-j+\ell'-\ell+m'-m}.
  \end{align*}
\end{itemize}
An analogous result is true for symbol-kernels with finite regularity number $\nu$. Similar results hold for Poisson and 
trace symbol-kernels.  
\end{theorem}

For the characterization of generalized singular Green symbols of positive typr $r>0$ recall the notations from 
$\eqref{eq:kappa-matrix}$ and $\eqref{eq:partial-plus}$ and note that 
\begin{align}\label{eq:partial-plus02}
 \bfkappa^{-1}(\xi',\mu)\,\boldsymbol{\partial}_+^j\,\bfkappa(\xi',\mu)=
      \begin{pmatrix}[\xi',\mu]^j&0\\0&1\end{pmatrix}\boldsymbol{\partial}_+^j.
\end{align} 
This yields immediately the following result: 

\begin{theorem}\label{thm:conjugation02}
The following are equivalent: 
\begin{itemize}
 \item[a$)$]  $\bfg\in B^{d;r}_G(\rz^{n-1}\times\rpbar;\frakg)$. 
 \item[b$)$] $\bfg':=\bfkappa^{-1}\bfg\bfkappa\in S^{d}(\rz^{n-1}\times\rpbar;
  \Gamma^{r}_G(\rz_+;\frakg))$. 
\end{itemize}
The analogous result is true for symbols with finite regularity number, i.e., replacing $B^{d;r}_G$ by 
$\wt B^{d,\nu;r}_G$ and $S^d$ by $\wt S^{d,\nu}$, respectively. 
\end{theorem}

\subsection{Singular Green symbols with expansion at infinity}\label{sec:4.3}

The following definition is a natural consequence of Theorem \ref{thm:conjugation02}. 

\begin{definition}\label{def:conjugation01}
$\wtbfB^{d,\nu;r}_G(\rz^{n-1}\times\rpbar;\frakg)$ consists of all symbols $\bfg$ such that 
 $$\bfg':=\bfkappa^{-1}\bfg\bfkappa\in \wtbfS^{d,\nu}(\rz^{n-1}\times\rpbar;
  \Gamma^{r}_G(\rz_+;\frakg)).$$  
\end{definition}

If $\bfg'$ is as in the previous definition then, due to \eqref{eq:partial-plus02}, it has the form 
 $$\bfg'=\bfg'_0+\sum_{j=1}^{r}\bfg'_j\boldsymbol{\partial}_+^j,\qquad 
   \bfg'_j\in \wtbfS^{d,\nu}(\rz^{n-1}\times\rpbar;\Gamma^{0}_G(\rz_+;\frakg)).$$

\begin{remark}
In view of the above Theorem $\ref{thm:conjugation01}$, 
\begin{align*}
 \wtbfB^{d,\nu;r}_G(\rz^{n-1}\times\rpbar;\frakg)
 &\subset \wt{B}^{d,\nu;r}_G(\rz^{n-1}\times\rpbar;\frakg),\\
 B^{d;r}_G(\rz^{n-1}\times\rpbar;\frakg)
 &\subset \wtbfB^{d,0;r}_G(\rz^{n-1}\times\rpbar;\frakg). 
\end{align*}
\end{remark}

Due to the first inclusion in the previous remark, with each generalized singular Green symbol  
$\bfg\in\wtbfB^{d,\nu;r}_G(\rz^{n-1}\times\rpbar;\frakg)$ we can associate its homogeneous 
principal symbol $\bfg^{(d,\nu)}(x',\xi',\mu)$, $\xi'\not=0$, as in \eqref{eq:symbol-positive-type-02}. 
If $\bfg'=\bfkappa^{-1}\bfg\bfkappa$ as in Definition \ref{def:conjugation01}, then 
\begin{align*}
  \bfg^{(d,\nu)}(x',\xi',\mu)=\bfkappa_{|\xi',\mu|}{\bfg'}^{(d,\nu)}(x',\xi',\mu)\bfkappa^{-1}_{|\xi',\mu|}.
\end{align*}
For the following definition recall Definition \ref{def:symbol-expansion}. 
 
\begin{definition}\label{def:b-tilde-limit}
The \emph{principal limit-symbol} of $\bfg\in\wtbfB^{d,\nu;r}_G(\rz^{n-1}\times\rpbar;\frakg)$ is 
\begin{align*}
    \bfg^\infty_{[d,\nu]}&:=(\bfkappa^{-1}\bfg\bfkappa)^\infty_{[d,\nu]}
    \in S^{\nu}(\rz^{n-1};\Gamma^{r}_G(\rz_+;\frakg)).
\end{align*}
\end{definition}

Again, as a subordinate principal symbol, we define the principal angular symbol 
\begin{align*}
 \bfg_{\spk{d,\nu}}:=(\bfkappa^{-1}\bfg\bfkappa)_{\spk{d,\nu}}
    \in S^{\nu}_\Hom(\rz^{n-1};\Gamma^{r}_G(\rz_+;\frakg)),
\end{align*}
cf. Definition \ref{def:angular-symbol02}. If $\bfg\in B^{d;r}_G(\rz^{n-1}\times\rpbar;\frakg)$ then 
\begin{align*}
 \bfg^\infty_{[d,0]}(x',\xi')=(\bfkappa^{-1}\bfg\bfkappa)^{(d)}(x',0,1)=\bfg^{(d)}(x',0,1)=\bfg_{\spk{0}}(x',\xi').
\end{align*}

\forget{
***************** BRAUCHT MAN DAS??******************

Whenever 
$\bfg_j\in \wtbfB^{d_j,\nu_j;r_j}_G(\rz^{n-1}\times\rpbar;\frakg_j)$, $j=0,1$, and $\frakg_1$ is composable 
with $\frakg_0$, then obviously 
$\bfg_1\bfg_0\in \wtbfB^{d_0+d_1,\nu_0+\nu_1;r_0}_G(\rz^{n-1}\times\rpbar;\frakg_1\frakg_0)$ 
and all principal symbols behave multiplicatively.   

ADJOINT

*****************
}

\subsection{(Un-)twisting on operator-level}\label{sec:4.4}

In Sections \ref{sec:4.1} and \ref{sec:4.2} we have seen that conjugation by 
$\bfkappa=\begin{pmatrix}\kappa&0\\0&1\end{pmatrix}$ establishes, on the level of symbols, 
a one-to-one correspondence between generalized singular Green symbols and symbols with values in 
generalized singular Green operators on the half-axis. 
In this Section we shall show that this correspondence persists on 
operator-level, i.e., conjugation with the group-action operator establishes a one-to-one correspondence 
between the various classes. 

\begin{theorem}\label{thm:conjugation03}
Let $\frakB_G^{d;r}$ represent one choice of $B^{d;r}_G$, $\wt B^{d,\nu;r}_G$, or $\wtbfB^{d,\nu;r}_G$ and 
let $\frakS^d$ represent the corresponding choice $S^d$, $\wt{S}^{d,\nu}$, or $\wtbfS^{d,\nu}$. Let 
$\frakg=((L_0,M_0),(L_1,M_1))$. 
\begin{itemize}
\item[a$)$] Let $\bfg\in\frakB_G^{d;r}(\rz^{n-1}\times\rpbar;\frakg)$. Then there exists a unique symbol 
 $\bfg^\prime\in \frakS^d(\rz^{n-1}\times\rpbar;\Gamma^r_G(\rz_+;\frakg))$ such that
  $$\op(\bfg')(\mu)=\op(\bfkappa^{-1})(\mu)\op(\bfg)(\mu)\op(\bfkappa)(\mu).$$
 We shall write $\bfkappa^{-1}\#\bfg\#\bfkappa:=\bfg^\prime$. One has 
 \begin{align}\label{eq:expansion01}
    \bfkappa^{-1}\#\bfg\#\bfkappa- \bfkappa^{-1}\bfg\,\bfkappa
    \in \frakS^{d-1}(\rz^{n-1}\times\rpbar;\Gamma^r_G(\rz_+;\frakg))
  \end{align}
\item[b$)$] Let $\bfg^\prime\in \frakS^d(\rz^{n-1}\times\rpbar;\Gamma^r_G(\rz_+;\frakg))$. Then there 
 exists a unique symbol $\bfg\in\frakB_G^{d;r}(\rz^{n-1}\times\rpbar;\frakg)$ such that
  $$\op(\bfg)(\mu)=\op(\bfkappa)(\mu)\op(\bfg^\prime)(\mu) \op(\bfkappa^{-1})(\mu).$$ 
 We shall write $\bfkappa\#\bfg^\prime\#\bfkappa^{-1}:=\bfg$. One has 
 \begin{align}\label{eq:expansion02}
  \bfkappa\#\bfg'\#\bfkappa^{-1}-\bfkappa\bfg'\bfkappa^{-1}
  \in \frakB_G^{d-1;r}(\rz^{n-1}\times\rpbar;\frakg).
 \end{align}
\end{itemize}
\end{theorem}

The proof of this theorem will occupy the remaining part of this section. 
Let us remark that in the proof we shall obtain a complete asymptotic expansions of $\bfkappa^{-1}\#\bfg\#\bfkappa$ 
and $\bfkappa\#\bfg'\#\bfkappa^{-1}$ which include 
\eqref{eq:expansion01} and \eqref{eq:expansion02} as special cases $($see Proposition \ref{prop:osint} and
the subsequent paragraph$)$. 
The main case to verify is that of type $r=0$ which we will assume from now on. 
We will first prove a$)$ and \eqref{eq:expansion01}. 
The proof makes essential use of the concept of oscillatory-integrals and 
related techniques in the spirit of \cite{Kuma}, here in a variant for amplitude functions with values 
in Fréchet spaces (see Theorem \ref{thm:oscillatory-int}). 

We start out from the representation 
 $\bfg=\begin{pmatrix}g&k\\t&q\end{pmatrix}
    =\bfkappa\bfg_0'\bfkappa^{-1}$
with a symbol 
 $$\bfg'_0=\begin{pmatrix}g_0'&k_0'\\t_0'&q_0'\end{pmatrix}
    \in \frakS^d(\rz^{n-1}\times\rpbar;\Gamma^0_G(\rz_+;\frakg)),\qquad \frakg=((L_0,M_0),(L_1,M_1));$$
this is possible by Theorem \ref{thm:conjugation01} and Definition \ref{def:conjugation01}, respectively. 
Since $\bfkappa$ is independent of $x'$, we find that 
\begin{align}\label{eq:conjugation03}
\begin{split}
  \op(\bfkappa^{-1})(\mu)
  &\op(\bfg)(\mu)\op(\bfkappa)(\mu)
  =\op(\bfkappa^{-1})(\mu)\op(\bfkappa\bfg_0')\\
  &=\begin{pmatrix}\op(\kappa^{-1})(\mu)\op(\kappa g_0')(\mu)
  &\op(\kappa^{-1})(\mu)\op(\kappa k_0')(\mu)\\\op(t_0')(\mu)&\op(q_0')(\mu)\end{pmatrix}.
\end{split}
\end{align}
We thus have to show the existence of 
 $$\bfg'=\begin{pmatrix}g'&k'\\t'&q'\end{pmatrix}
    \in \frakS^d(\rz^{n-1}\times\rpbar;\Gamma^0_G(\rz_+;\frakg))$$
such that $\op(\bfg')(\mu)$ coincides with the right-hand side of \eqref{eq:conjugation03} and such that 
\eqref{eq:expansion01} is valid. To this end we obviously must define $t':=t_0'$ and $q':=q_0'$ and have to find 
$g'$ and $k'$ such that 
 $$\op(g')(\mu)=\op(\kappa^{-1})(\mu)\op(\kappa g_0')(\mu),\qquad 
     \op(k')(\mu)=\op(\kappa^{-1})(\mu)\op(\kappa k_0')(\mu)$$
as well as  
 $$g'-g_0'\in\frakS^{d-1}(\rz^{n-1}\times\rpbar;\Gamma^0_G((L_0,0),(L_1,0)))$$
and 
\begin{align}\label{eq:expansion03}
 k'-k_0'\in \frakS^{d-1}(\rz^{n-1}\times\rpbar;\Gamma^0_G((0,M_0),(L_1,0))).
\end{align}
We shall now verify this for the Poisson part, i.e., show the existence of $k'$ together with \eqref{eq:expansion03}. 
The proof for the Green part $g'$ works in the same way and only comes along with lengthier notation, since the 
symbol-kernel of Green symbol depends on $(x_n,y_n)$ while that of a Poisson symbol only on $x_n$ (note that 
the group-action and the group-action operator act from the left which only has an effect on the $x_n$-variable). 

If $u\in\scrS(\rz^{n-1},\cz^{M_0})$ then, by direct computation,  
\begin{align*}
\begin{split}
 [\op(\kappa^{-1})(\mu)&\op(\kappa k_0')(\mu)u](x')\\
 &=\iint\hspace*{-1.1ex}\iint e^{i(x'-y')\eta'+i(y'-y'')\eta''}p(\eta',y',\eta'',\mu)u(y'')
 \,dy''\dbar\eta''dy'\dbar\eta'
\end{split}
\end{align*}
$($understood as iterated integrals$)$ where 
\begin{align*}
\begin{split}
 p(\eta',y',\eta'',\mu)
 =\kappa(\eta',\mu)^{-1}\kappa(\eta'',\mu)k'_0(y',\eta'',\mu)
 =\kappa_{\omega(\eta',\eta'',\mu)}k_0'(y',\eta'',\mu)
\end{split}
\end{align*}
where 
\begin{align}\label{eq:comp03}
 \omega(\eta',\eta'',\mu):=[\eta',\mu]^{-1}[\eta'',\mu].
\end{align}
Now let us pass from the setting of operator-valued symbols to the formulation with symbol-kernels. 
Recall that the symbol-kernel of $k_0'$ satisfies  
 $$k_0'(x',\xi',\mu;x_n)\in \frakS^d(\rz^{n-1}\times\rpbar;\scrS(\rz_+,\cz^{L_1\times M_0})).$$
Then $p$ has the symbol-kernel 
\begin{align}\label{eq:comp02}
\begin{split}
 p(\eta',y',\eta'',\mu;x_n)
 &=\omega(\eta',\eta'',\mu)^{1/2}k_0'(x',\xi',\mu;\omega(\eta',\eta'',\mu)x_n)\\
 &= (\kappa_{\omega(\eta',\eta'',\mu)}k'_0)(y',\eta'',\mu;x_n)
\end{split}
\end{align}
(i.e., the group-action acts on the $x_n$-variable$)$ and 
\begin{align*}
\begin{split}
 [\op&(\kappa^{-1})(\mu)\op(\kappa k_0')(\mu)u](x)\\
 &=\iint\hspace*{-1.1ex}\iint e^{i(x'-y')\eta'+i(y'-y'')\eta''}p(\eta',y',\eta'',\mu;x_n)u(y'')
 \,dy''\dbar\eta''dy'\dbar\eta'.
\end{split}
\end{align*}
From \eqref{eq:chainrule} it easily follows that, for an arbitrary choice of multi-indices, 
 $$|D^{\alpha'}_{\eta'}D^{\alpha''}_{\eta''}D^{\beta'}_{y'}p(\eta',y',\eta'',\mu;x_n)|\lesssim C(\mu)
     \spk{\eta',\mu}^{-\frac{1}{2}-|\alpha'|}\spk{\eta'',\mu}^{d+\frac{1}{2}-|\alpha''|},$$
where $C(\mu)=1$ if $\frakS^d=S^d$ and $C(\mu)=\spk{\mu}^{(-\nu)_+}$ otherwise. 
Therefore, for fixed $\mu$ and $x_n$, $p(\eta',y',\eta'',\mu;x_n)$ is a double symbol from the class 
$S^{-\frac12,d+\frac12}_{1,0}$ in the sense of \cite[Definition 2.1, Chapter 2]{Kuma} (where the variables 
$(\eta',y',\eta'')$ correspond to $(\xi,x',\xi')$ in \cite{Kuma}). We thus can apply the results of \cite{Kuma}, 
in particular Theorem 2.5 and Theorem 3.1 of Chapter 2, to obtain the following$:$

\begin{proposition}\label{prop:osint}
With the above notation, 
 $$[\op(\kappa^{-1})(\mu)\op(\kappa k_0')(\mu)u](x)=[\op(p_L)(\mu;x_n)u](x'),$$
where 
 $$p_L(x',\xi',\mu;x_n)=\mathrm{Os}-\iint e^{-iy'\eta'}p(\xi'+\eta',x'+y',\xi',\mu;x_n)\,dy'\dbar\eta'$$
is, for each $\mu$ and $x_n$, a symbol in  $S^d_{1,0}(\rz^{n-1};\cz^{L_1\times M_0})$. Moreover, 
\begin{align}\label{eq:expansion}
\begin{split}
 p_L(x',\xi',\mu;x_n)=&\sum_{|\alpha'|=0}^{N-1}\frac{1}{\alpha'!}\partial^{\alpha'}_{\eta'}D^{\alpha'}_{y'}
 p(\eta',y',\xi',\mu;x_n)\Big|_{\substack{y'=x'\\ \eta'=\xi'}}+\\
 &+N\sum_{|\gamma'|=N}\int_0^1\frac{(1-\theta)^{N-1}}{\gamma'!}r_{\gamma',\theta}(x',\xi',\mu;x_n)\,d\theta
\end{split}
\end{align}
where 
 $$r_{\gamma',\theta}(x',\xi',\mu;x_n)
    =\mathrm{Os}-\iint e^{-iy'\eta'}(\partial^{\gamma'}_{\eta'}D^{\gamma'}_{y'} p)
    (\xi'+\theta\eta',x'+y',\xi',\mu;x_n)\,dy'\dbar\eta'$$
is, for each $\mu$ and $x_n$, a symbol in  $S^{d-N}_{1,0}(\rz^{n-1};\cz^{L_1\times M_0})$. 
\end{proposition} 

By virtue of \eqref{eq:chainrule} the terms in the first sum on the right-hand side of 
\eqref{eq:expansion} satisfiy 
 $$\partial^{\alpha'}_{\eta'}D^{\alpha'}_{y'}
    p(\eta',y',\xi',\mu;x_n)\Big|_{\substack{y'=x'\\ \eta'=\xi'}}\in 
     \frakS^{-|\alpha'|}(\rz^{n-1}\times\rpbar;\scrS(\rz_+,\cz^{L_1\times M_0}));$$
note here that $\omega(\xi',\xi',\mu)\equiv1$. 
In particular, the first term (with $\alpha'=0)$ is 
 $$p(\eta',y',\mu;x_n)\Big|_{\substack{y'=x'\\ \eta'=\xi'}}=k_0'(x',\xi',\mu;x_n).$$
Hence, to complete the proof of Theorem \ref{thm:conjugation02} it remains to show that 
\begin{equation}\label{eq:osint01}
 \int_0^1 (1-\theta)^{N-1}
    r_{\gamma',\theta}(x',\xi',\mu;x_n)\,d\theta\in 
    \frakS^{d-N}(\rz^{n-1}\times\rpbar;\scrS(\rz_+,\cz^{L_1\times M_0}))
\end{equation}
for every $\gamma'$ with $|\gamma'|=N$. 
To this end we make use of the following result:  

\begin{theorem}[and Definition]\label{thm:oscillatory-int}
Let $F$ be a Frèchet space whose topology is given by a system of semi-norms 
$\{\trinorm{\cdot}^{(\ell)}\mid \ell\in\nz\}$. A smooth function $a(y',\eta')$ on $\rz^{n-1}\times\rz^{n-1}$ 
with values in $F$ is called an \emph{amplitude function} if there exists a sequence 
$\tau=(\tau_\ell)$ of real numbers such that 
 $$\trinorm{D^{\alpha'}_{\eta'}D^{\beta'}_{y'}a(y',\eta')}^{(\ell)}\lesssim \spk{\eta'}^{\tau_\ell}$$
uniformly in $(y',\eta')$ for every choice of indices $\ell$, $\alpha'$, $\beta'$. Then 
 $$\mathrm{Os}[a]:=\mathrm{Os}-\iint e^{iy'\eta'}a(y',\eta')\,dy'\dbar\eta':=\lim_{\eps\to 0}\iint e^{iy'\eta'}
     \chi(\eps y',\eps\eta')a(y',\eta')\,dy'\dbar\eta'$$
exists in $F$, where $\chi\in\scrC^\infty_{\mathrm{comp}}(\rz^{n-1}\times\rz^{n-1})$ with $\chi(0,0)=1$. 
The space $\calA^\tau(\rz^{n-1};F)$ of all such amplitudes is a Fréchet space in the obvious way and the map 
$a\mapsto\mathrm{Os}[a]$ is continuous. 
\end{theorem}

If $\tau^0,\tau^1$ are two sequences and $B:F_0\times F_1\to F$ is a bilinear continuous mapping between 
Fréchet spaces, then $B$ induces a continuous map 
 $$B:\calA^{\tau^0}(\rz^{n-1};F_0)\times \calA^{\tau^1}(\rz^{n-1};F_1)\lra 
     \calA^{\tau}(\rz^{n-1};F),$$
with a resulting sequence $\tau=\tau(\tau^0,\tau^1)$.  

\begin{proposition}\label{prop:osint02}
With the above notation and $|\gamma'|=N$, the function 
 $$a(y',\eta';x',\xi',\mu,\theta;x_n):= 
     (\partial^{\gamma'}_{\eta'}D^{\gamma'}_{y'} p)(\xi'+\theta\eta',x'+y',\xi',\mu;x_n)$$
is an amplitude function of $(y',\eta')$ with values in 
$$\scrC\big([0,1],\frakS^{d-N}(\rz^{n-1}\times\rpbar;\scrS(\rz_+,\cz^{L_1\times M_0}))\big).$$   
\end{proposition}

This result yields that the oscillatory integral defining $r_{\gamma',\theta}$ in Proposition \ref{prop:osint} 
converges in $\scrC\big([0,1],\frakS^{d-N}(\rz^{n-1}\times\rpbar;\scrS(\rz_+,\cz^{L_1\times M_0}))\big)$, 
hence integration with respect to $\theta$ yields 
\eqref{eq:osint01} and thus finishes the proof of Theorem \ref{thm:conjugation02}. 

The proof of Proposition $\ref{prop:osint02}$ is again split into different steps. To begin with, 
let us recall that $\scrS(\rz_+)$ is a nuclear Fréchet space and that 
 $$\scrS(\rz_+,F)=\scrS(\rz_+)\wh\otimes_\pi F$$
with the completed projective tensor-product. By a classical result thus every function $u\in \scrS(\rz_+,F)$ 
can be written as a series 
  $$u(t)=\sum_{\ell=0}^{+\infty} c_\ell u_\ell(t) f_\ell$$
where $(c_\ell)\subset\cz$ is an absolutely summable numerical sequence and $(u_\ell)$, $(f_\ell)$ are 
infinitesimal sequences in  $\scrS(\rz_+)$ and $F$, respectively. Applying this to the symbol-kernel $k_0'$ 
it is therefore no restriction of generality to assume that $k_0'$ has the form 
 $$k_0'(x',\xi',\mu;x_n)=r(x',\xi',\mu)\phi(x_n)$$
with $r(x',\xi',\mu)\in \frakS^d(\rz^{n-1}\times\rpbar;\cz^{L_1\times M_0})$ and $\phi(x_n)\in\scrS(\rz_+)$. 
Correspondingly, 
\begin{align}\label{eq:amplitude}
 p(\eta',y',\eta'',\mu;x_n)=r(y',\eta'',\mu) 
 \underbrace{(\kappa_{\omega(\eta',\eta'',\mu)}\phi)(x_n)}_{=:p_\phi(\eta',\eta'',\mu;x_n)}.
\end{align}
It is obvious that  
 $$a_1(y',x',\xi',\mu):=D^{\gamma'}_{y'}r(x'+y',\xi',\mu) $$
is an amplitude function in $y'$ with values in $\frakS^d(\rz^{n-1}\times\rpbar;\cz^{L_1\times M_0})$. 
In fact, whenever $f\in\scrC^\infty_b(\rz^{n-1},F)$ with some Fréchet space $F$, then $f(x'+y')$ 
is an amplitude function of $y'$ with values in $\scrC^\infty_b(\rz^{n-1},F)$. 

\begin{lemma}\label{lem:osint03}
With the above notation and $|\gamma'|=N$, the function 
 $$a_\phi(\eta',\xi',\mu,\theta;x_n):= 
     (\partial^{\gamma'}_{\eta'}p_\phi)(\xi'+\theta\eta',\xi',\mu;x_n)$$
is an amplitude in $\eta'$ with values in 
$\scrC\big([0,1],S^{-N}(\rz^{n-1}\times\rpbar;\scrS(\rz_+,\cz^{L_1\times M_0}))\big)$.  
\end{lemma}
\begin{proof}
For simplicity of presentation we shall show the lemma only in case $N=0$; the general case 
is verified in exactly the same way and only comes along with more complex notation due to the involved 
chain rule. 

Next note that it is enough to show that $a_\phi(\eta',\xi',\mu;x_n):=a_\phi(\eta',\xi',\mu,1;x_n)$ is an amplitude 
function with values in $S^{0}(\rz^{n-1}\times\rpbar;\scrS(\rz_+,\cz^{L_1\times M_0}))$; 
in fact, if $a(y',\eta')$ is an amplitude function with values in a Fréchet space $F$ then 
$\wt{a}(y',\eta';\theta):=a(y',\theta\eta')$ is an amplitude with values in 
$\scrC([0,1],F)$. 

We need to show that the homogeneous components 
$a_\phi^{(-j)}(\eta',\xi',\mu;x_n)$ are amplitude functions with values in 
$S^{-j}_{hom}(\rz^{n-1}\times\rpbar;\scrS(\rz_+,\cz^{L_1\times M_0}))$ and that 
 $$a_\phi(\eta',\xi',\mu;x_n)
     -\sum_{j=0}^{L-1}\chi(\xi',\mu)a_\phi^{(-j)}(\eta',\xi',\mu;x_n)$$
is an amplitude with values in $S^{-L}_{1,0}(\rz^{n-1}\times\rpbar;\scrS(\rz_+))$ for arbitrary $L$. 

Due to \eqref{eq:chainrule}, $\partial^{\alpha'}_{\eta'}a_\phi(\eta',\xi',\mu;x_n)$ is a finite linear combination 
of terms 
\begin{align}\label{eq:hilf01}
 [\xi'+\eta',\mu]^{m}\Big(\prod_{\ell=1}^m \partial^{\alpha'_\ell}_{\eta'}[\xi'+\eta',\mu]^{-1}\Big)
 \kappa_{\omega(\xi'+\eta',\xi',\mu)}\Theta_m\phi(x_n)
\end{align}
with $0\le m\le|\alpha'|$ and $\alpha'_1+\ldots+\alpha_m'=\alpha'$. Therefore, 
since $\omega(\xi',\xi',\mu)\equiv1$,
 $$\partial^{\alpha'}_{\eta'}a_\phi(\eta',\xi',\mu;x_n)\big|_{\eta'=0}
    =\alpha'!\sum_{m=0}^{|\alpha'|} q_{\alpha',m}(\xi',\mu)\Theta_m\phi(x_n)$$
with certain $q_{\alpha',m}(\xi',\mu)\in S^{-|\alpha'|}(\rz^{n-1}\times\rpbar)$ which are positively homogeneous 
of degree $-|\alpha'|$ for $|\xi',\mu|\ge1$. Thus, by Taylor expansion, 
\begin{align*}
 a_\phi(\eta',\xi',\mu;x_n)
 =&\sum_{|\alpha'|=0}^{L-1}\sum_{m=0}^{|\alpha'|} 
     q_{\alpha',m}(\xi',\mu)\Theta_m\phi(x_n)\eta'^{\alpha'}+\\
 &+N\sum_{|\beta'|=L}\eta'^{\beta'}\int_0^1\frac{(1-t)^{L-1}}{\beta'!}
    (\partial^{\beta'}_{\eta'}a_\phi)(t\eta',\xi',\mu;x_n)\,dt
\end{align*}
It remains to show that $(\partial^{\beta'}_{\eta'}a_\phi)(\eta',\xi',\mu;x_n)$, $|\beta'|=N$, is an amplitude 
with values in 
$S^{-L}_{1,0}(\rz^{n-1}\times\rpbar;\scrS(\rz_+))$ because then we have verified our claim, 
since then $a_\phi$ has the homogeneous components 
$$a_\phi^{(-j)}(\eta',\xi',\mu;x_n)
    =\sum_{|\alpha'|=j}\sum_{m=0}^{j} 
     q_{\alpha',m}(\xi',\mu)\Theta_m\phi(x_n)\eta'^{\alpha'},$$
which are clearly amplitude functions of $\eta'$. As seen from \eqref{eq:hilf01},  
$(\partial^{\beta'}_{\eta'}a_\phi)(\eta',\xi',\mu;x_n)$ is a linear combination of terms 
  $$q(\xi'+\eta',\mu)\kappa_{\omega(\xi'+\eta',\xi',\mu)}\psi(x_n),\qquad q\in S^{-L}(\rz^{n-1}\times\rpbar),
      \quad     \psi\in\scrS(\rz_+).$$
From 
 $$|D^{\gamma'}_{\eta'}D^{\alpha'}_{\xi'}D^j_\mu q(\xi'+\eta',\mu)|\lesssim 
    \spk{\xi'+\eta',\mu}^{-L-|\alpha'|-|\gamma'|-j}\le
    \spk{\eta'}^{L+|\alpha'|+j}\spk{\xi',\mu}^{-L-|\alpha'|-j}$$
it follows that $q(\xi'+\eta',\mu)$ is an amplitude function with values in $S^{-L}(\rz^{n-1}\times\rpbar)$. 
Finally, let 
 $$b(\eta';\xi',\mu;x_n)=\kappa_{\omega(\xi'+\eta',\xi',\mu)}\psi(x_n).$$
Then $x_n^\ell D_{x_n}^{\ell'}D^{\gamma'}_{\eta'}D^{\alpha'}_{\xi'}D^j_\mu b$ is a finite linear combination 
of terms 
 $$\omega^{\ell'-\ell}\omega^{-m}\Big(\prod_{i=1}^m D^{\gamma'_i}_{\eta'}D^{\alpha_i'}_{\xi'}D^{j_i}_\mu
     \omega\Big) \kappa_\omega x_n^\ell D_{x_n}^{\ell'}\Theta_m\psi$$
with $m\le|\gamma'|+|\alpha'|+j$, $\gamma'_1+\ldots+\gamma'_m=\gamma'$, 
$\alpha'_1+\ldots+\alpha'_m=\alpha'$, and $j_1+\ldots+j_m=j$; here, $\omega$ is a short notation for 
$\omega(\xi'+\eta',\xi',\mu)$. We can estimate 
\begin{align*}
 |\omega(\xi'+\eta',\xi',\mu)^{\ell'-\ell}|=[\xi',\mu]^{\ell'-\ell}[\xi'+\eta',\mu]^{\ell-\ell'}
  \lesssim \spk{\eta'}^{\ell+\ell'}
\end{align*}
as well as 
\begin{align*}
\omega^{-1}|D^{\gamma'_i}_{\eta'}D^{\alpha_i'}_{\xi'}D^{j_i}_\mu\omega|
&\lesssim 
\frac{[\xi'+\eta',\mu]}{[\xi',\mu]}\sum_{\substack{\alpha''_i+\alpha'''_i=\alpha'_i\\ j'_i+j''_i=j_i}}
|D^{\alpha''_i}D_\mu^{j'_i}[\xi',\mu]||D^{\gamma'_i}_{\eta'}D^{\alpha'''_i}D_\mu^{j''_i}[\xi'+\eta',\mu]^{-1}|\\
&\lesssim \sum_{\substack{\alpha''_i+\alpha'''_i=\alpha'_i\\ j'_i+j''_i=j_i}}
[\xi',\mu]^{-|\alpha''_i|-j'_i}[\xi'+\eta',\mu]^{-|\alpha'''_i|-j''_i-|\gamma'_i|}\\
&\lesssim \spk{\xi',\mu}^{-|\alpha'_i|-j_i}\spk{\eta'}^{|\alpha'_i|+j_i}
\end{align*}
This yields
\begin{align*}
|x_n^\ell D_{x_n}^{\ell'}D^{\gamma'}_{\eta'}D^{\alpha'}_{\xi'}D^j_\mu b(\eta';\xi',\mu;x_n)|\lesssim 
\spk{\xi',\mu}^{-|\alpha'|-j}\spk{\eta'}^{\ell+\ell'+|\alpha'|+j+\frac12}. 
\end{align*}
Hence $b$ is an amplitude function with values in $S^0_{1,0}(\rz^{n-1}\times\rpbar;\scrS(\rz_+))$. 
Thus $(\partial^{\beta'}_{\eta'}a_\phi)(\eta',\xi',\mu;x_n)$ is as claimed and the proof is complete.
\end{proof}

This finishes the proof of of Theorem \ref{thm:conjugation03}.a$)$. The proof of b$)$ is very similar 
and therefore we shall only indicate the necessary adjustments in 
the above proof but leave the details to the reader. Using block-matrix representations as above,  
we need to find $\bfg\in\frakB^{d;r}_G$ such that 
 $$\op(\bfg)(\mu)=\begin{pmatrix}\op(\kappa)(\mu)\op(g')(\mu)\op(\kappa^{-1})(\mu)&\op(\kappa)(\mu)\op(k')(\mu)\\
                                          \op(t')(\mu)\op(\kappa^{-1})(\mu)&\op(q')(\mu)\end{pmatrix}.$$
Therefore $q=q'$ and $t=t'\kappa^{-1}$. Then determine symbols $\kappa\#g'$ and $\kappa\#k'$ such that 
 $$\op(\kappa)(\mu)\op(g')(\mu)=\op(\kappa\#g')(\mu),\qquad \op(\kappa)(\mu)\op(k')(\mu)=\op(\kappa\#k')(\mu)$$
and define $g:=(\kappa\#g')\kappa^{-1}$ and $k:=\kappa\#k'$. One then has to show that 
\begin{align*}
 \kappa^{-1}g\kappa=\kappa^{-1}(\kappa\#g')\in \frakS(\rz^{n-1}\times\rpbar;\Gamma^r_G(\rz_+;(L_0,0),(L_1,0)),\\
 \kappa^{-1}k=\kappa^{-1}(\kappa\#k')\in \frakS(\rz^{n-1}\times\rpbar;\Gamma^r_G(\rz_+;(0,M_0),(0,M_1))
\end{align*}
and, in order to verify \eqref{eq:expansion02}, 
\begin{align*}
 \kappa^{-1}g\kappa-g'\in \frakS^{d-1}(\rz^{n-1}\times\rpbar;\Gamma^r_G(\rz_+;(L_0,0),(L_1,0)),\\
 \kappa^{-1}k-k'\in \frakS^{d-1}(\rz^{n-1}\times\rpbar;\Gamma^r_G(\rz_+;(0,M_0),(0,M_1)). 
\end{align*}
Treating as above only on the Poisson part, to calculate $\kappa\#k'$ we have 
\begin{align*}
\begin{split}
 [\op(\kappa)(\mu)&\op(k')(\mu)u](x')\\
 &=\iint\hspace*{-1.1ex}\iint e^{i(x'-y')\eta'+i(y'-y'')\eta''}p'(\eta',y',\eta'',\mu)u(y'')
 \,dy''\dbar\eta''dy'\dbar\eta'
\end{split}
\end{align*}
with $p'(\eta',y',\eta'',\mu)=\kappa(\eta',\mu)k'(y',\eta'',\mu)$. Thus passing to symbol-kernels, 
in \eqref{eq:comp02}, $p$ needs to be replaced by 
 $$p(\eta',y',\eta'',\mu;x_n)=(\kappa(\eta',\mu)k')(y',\eta'',\mu;x_n).$$
Then the analogue of Proposition \ref{prop:osint} and \eqref{eq:expansion} for $\op(\kappa)(\mu)\op(k')(\mu)$
is valid; in particular, the leading term in \eqref{eq:expansion} is $(\kappa(\xi',\mu)k')(x',\xi',\mu;x_n)$. 
In \eqref{eq:osint01} we have to substitute $r_{\gamma',\theta}$ by $\kappa^{-1}(\xi',\mu)r_{\gamma',\theta}$
and, correspondingly, in Proposition \ref{prop:osint02}, $a$ by $\kappa^{-1}(\xi',\mu)a$. With 
$k'(x',\xi',\mu;x_n)=r(x',\xi',\mu)\phi(x_n)$ in \eqref{eq:amplitude} we then need to consider instead 
\begin{align*}
 p(\eta',y',\eta'',\mu;x_n)=r(y',\eta'',\mu) (\kappa^{-1}(\eta'',\mu)\kappa(\eta',\mu)\phi)(x_n)
 =\underbrace{(\kappa_{1/\omega(\eta',\eta'',\mu)}\phi(x_n)}_{=:p_\phi(\eta',\eta'',\mu;x_n)}, 
\end{align*}
with $\omega$ as defined in \eqref{eq:comp03}. Therefore, to prove b), it remains to repeat the proof of Lemma 
\ref{lem:osint03} with the function $\omega$ replaced by $1/\omega$.

\begin{corollary}\label{cor:comp03}
Let $\bfg\in\wtbfB^{d,\nu;r}_G(\rz^{n-1}\times\rpbar;\frakg)$. Then 
$\bfg^\infty_{[d,\nu]}=(\bfkappa^{-1}\#\bfg\#\bfkappa)^\infty_{[d,\nu]}$. 
\end{corollary}

\subsection{The calculus of generalized singular Green symbols}\label{sec:4.5}

The composition of generalized singular Green operators on $\rz_+$ induces continuous mappings 
 $$\Gamma^{r_1}_G(\rz_+;\frakg_1)\times \Gamma^{r_0}_G(\rz_+;\frakg_0)
     \lra \Gamma^{r_0}_G(\rz_+;\frakg_1\frakg_0),$$ 
taking the formal adjoint with respect to the $L^2$-inner product yields continous mappings
 $$\Gamma^{0}_G(\rz_+;\frakg)\lra \Gamma^{0}_G(\rz_+;\frakg^{-1}).$$
Given $\bfg'_j\in\wtbfS^{d_j,\nu_j}(\rz^{n-1}\times\rpbar;\Gamma^{r_j}_G(\rz_+;\frakg_j))$ 
then, following the proof in Section 8 of \cite{Seil22-1}, one can show that 
 $$(\bfg'_1\#\bfg'_0)(x',\xi',\mu)
     :=\mathrm{Os}-\iint e^{-iy'\eta'}\bfg_1'(x'+y',\xi',\mu)\bfg_0'(x',\xi'+\eta',\mu)\,dy'\dbar\eta'$$
is an oscillatory integral converging in 
$\wtbfS^{d_0+d_1,\nu_0+\nu_1}(\rz^{n-1}\times\rpbar;\Gamma^{r_0}_G(\rz_+;\frakg_1\frakg_0))$
and  
\begin{align*}
 \op(\bfg'_1\#\bfg'_0)(\mu)&=\op(\bfg'_1)(\mu)\op(\bfg'_0)(\mu),\\
    (\bfg'_1\#\bfg'_0)^\infty_{[d_0+d_1,\nu_0+\nu_1]}
    &=(\bfg'_1)^\infty_{[d_1,\nu_1]}(\bfg'_0)^\infty_{[d_00,\nu_0]}.
\end{align*}
Likewise, if $\bfg'\in\wtbfS^{d,\nu}(\rz^{n-1}\times\rpbar;\Gamma^{0}_G(\rz_+;\frakg))$ with $d\le0$ then 
 $${\bfg'}^{(*)}(x',\xi',\mu)
     :=\mathrm{Os}-\int e^{-iy'\eta'}\bfg'(x'+y',\xi'+\eta',\mu)^*\,dy'\dbar\eta'$$
converges in  $\wtbfS^{d,\nu}(\rz^{n-1}\times\rpbar;\Gamma^{0}_G(\rz_+;\frakg^{-1}))$ and 
 $$\op({\bfg'}^{(*)})=\op(\bfg')^*,\qquad ({\bfg'}^{(*)})^\infty_{[d,\nu]}=({\bfg'}^\infty_{[d,\nu]})^{(*)}.$$ 
Combining this with Theorem \ref{thm:conjugation03} and Corollary \ref{cor:comp03} we immediately get the 
analogue of Theorem \ref{thm:comp-strong} for generalized singular Green symbols with limit symbol at infinity.
At the same time, this way of reasoning also yields an alternative proof for the algebra property of  
generalized Green symbols with strong parameter dependence. 

\begin{theorem}\label{thm:adjoint-limit}
Let $d\le0$. Then $\bfg\mapsto \bfg^{(*)}$ induces continuous maps
\begin{align*} 
 \wtbfB^{d,\nu;0}_G&(\rz^{n-1}\times\rpbar;\frakg)\lra 
 \wtbfB^{d,\nu;0}_G(\rz^{n-1}\times\rpbar;\frakg^{-1}).  
\end{align*}
Both homogeneous principal symbol and principal limit symbol are well-behaved, 
\begin{align*} 
  ({\bfg}^{(*)})^{(d,\nu)}=({\bfg}^{(d,\nu)})^{*},\qquad 
  ({\bfg}^{(*)})^\infty_{[d,\nu]}=({\bfg}^\infty_{[d,\nu]})^{(*)}. 
\end{align*}
\end{theorem}
\begin{proof}
Due to Theorem \ref{thm:conjugation03} we can write $\bfg=\bfkappa\#\bfg'\#\bfkappa^{-1}$. 
Snce the dilation semi-group is unitary on $L^2$, the group-action operator satisfies 
$\op(\bfkappa)(\mu)^*=\op(\bfkappa^{-1})(\mu)$. Therefore 
  $$\bfg^{(*)}=(\bfkappa\#\bfg'\#\bfkappa^{-1})^{(*)}
     =\bfkappa\#(\bfg')^{(*)}\#\bfkappa^{-1}.$$
Then the claim follows immediately by Corollary \ref{cor:comp03}. 
\end{proof}

\begin{theorem}\label{thm:comp-limit}
Let $\frakg_1$ be composable with $\frakg_0$. Then 
$(\bfg_1,\bfg_0)\mapsto \bfg_1\#\bfg_0$ induces continuous maps
\begin{align*} 
 \wtbfB^{d_1,\nu_1;r_1}_G&(\rz^{n-1}\times\rpbar;\frakg_1)\times  
 \wtbfB^{d_0,\nu_0;r_0}_G(\rz^{n-1}\times\rpbar;\frakg_0)\\
    &\lra \wtbfB^{d_0+d_1,\nu_0+\nu_1;r_0}_G(\rz^{n-1}\times\rpbar;\frakg_1\frakg_0).  
\end{align*}
Both homogeneous principal symbol and principal limit symbol are multiplicative, 
\begin{align*} 
  (\bfg_1\#\bfg_0)^{(d_0+d_1,\nu_0+\nu_1)}&=\bfg_1^{(d_1,\nu_1)}\bfg_0^{(d_0,\nu_0)},\\
  (\bfg_1\#\bfg_0)^\infty_{[d_0+d_1,\nu_0+\nu_1]}&=(\bfg_1)^\infty_{[d_1,\nu_1]}(\bfg_0)^\infty_{[d_0,\nu_0]}.
\end{align*}
\end{theorem}
\begin{proof}
According to Theorem \ref{thm:conjugation03} we can write $\bfg_j=\bfkappa\#\bfg'_j\#\bfkappa^{-1}$. Then  
 $$\bfg_1\#\bfg_0=(\bfkappa\#\bfg'_1\#\bfkappa^{-1})\#(\bfkappa\#\bfg'_0\#\bfkappa^{-1})
    =\bfkappa\#(\bfg'_1\#\bfg'_0)\#\bfkappa^{-1}.$$
From this together with Corollary \ref{cor:comp03} the claim follows immediately. 
\end{proof}

The analogue of the previous theorem also holds true in the classes $\wt B^{d,\nu;r}_G$. We shall also need 
the following result on asymptotic summation. 

\begin{proposition}\label{prop:asymptotic-summation01}
Let $\bfg_j\in \wtbfB^{d-j,\nu-j;r}_G(\rz^{n-1}\times\rpbar;\frakg)$, $j\in\nz_0$. Then there exists a symbol 
$\bfg_j\in \wtbfB^{d,\nu;r}_G(\rz^{n-1}\times\rpbar;\frakg)$ such that, for every $N\in\nz_0$, 
 $$\bfg-\sum_{j=0}^{N-1}\bfg_j\in \wtbfB^{d-j,\nu-j;r}_G(\rz^{n-1}\times\rpbar;\frakg).$$
\end{proposition}
\begin{proof}
By conjugation with $\bfkappa$, the claim follows from the corresponding claim of asymptotic summation 
for a sequence of symbols $\bfg'_j\in\wtbfS^{d-j,\nu-j}(\rz^{n-1}\times\rpbar;\Gamma^r_G(\rz_+;\frakg))$ in the class 
$\wtbfS^{d,\nu}(\rz^{n-1}\times\rpbar;\Gamma^r_G(\rz_+;\frakg))$. In turn this is  a special case of asymptotic 
summation of a sequence  of symbols $p_j\in\wtbfS^{d-j,\nu-j}(\rz^{n-1}\times\rpbar;F)$ in the class 
$\wtbfS^{d,\nu}(\rz^{n-1}\times\rpbar;F)$, where $F$ is some Fréchet space. In case $F=\cz$ this has been proved 
in \cite{Seil22-1}, the general case is verified in the same way. 
\end{proof}

\begin{proposition}\label{prop:asymptotic-summation02}
Let $\bfg_j\in \wtbfB^{d-j,\nu;r}_G(\rz^{n-1}\times\rpbar;\frakg)$, $j\in\nz_0$. Then there exists a symbol 
$\bfg_j\in \wtbfB^{d,\nu;r}_G(\rz^{n-1}\times\rpbar;\frakg)$ such that, for every $N\in\nz_0$, 
 $$\bfg-\sum_{j=0}^{N-1}\bfg_j\in \wtbfB^{d-j,\nu;r}_G(\rz^{n-1}\times\rpbar;\frakg).$$
\end{proposition}

The proof of this result is analogous to that of Proposition \ref{prop:asymptotic-summation01}.
Let us conclude this section with a remark concerning regularizing symbols: 

\begin{remark}\label{rem:smoothing-operators}
For arbitrary weight-datum $\frakg$, 
\begin{align*} 
 B^{-\infty;r}_G(\rz^{n-1}\times\rpbar;\frakg)
 \subset \wtbfB^{d-\infty,\nu-\infty;r}_G(\rz^{n-1}\times\rpbar;\frakg)
 \subset\wt B^{d-\infty,\nu-\infty;r}_G(\rz^{n-1}\times\rpbar;\frakg),
\end{align*} 
and 
\begin{align*} 
 B^{-\infty;r}_G(\rz^{n-1}\times\rpbar;\frakg)
 &=\mathop{\mbox{\Large$\cap$}}_{N\in\nz_0} \wtbfB^{d-N,\nu;r}_G(\rz^{n-1}\times\rpbar;\frakg)\\
 &=\mathop{\mbox{\Large$\cap$}}_{N\in\nz_0}\wt B^{d-N,\nu;r}_G(\rz^{n-1}\times\rpbar;\frakg)\\
 &=S^{-\infty}(\rz^{n-1}\times\rpbar;\Gamma^r_G(\rz_+,\frakg)).
\end{align*} 
\end{remark}
\section{The full calculus with expansion at infinity}\label{sec:5}

\subsection{Definition and symbolic structure}\label{sec:5.1}

The full calculus is obtained by enlarging the strongly parameter-dependent calculus by the new class 
of generalized singular Green symbols.  

\begin{definition}\label{def:singular-green-new}
Let $d,\nu\in\gz$, $r\in\nz_0$, and $\frakg=((L_0,M_0),(L_1,M_1))$. Then 
 $$\bfB^{d,\nu;r}_G(\rz^{n-1}\times\rpbar;\frakg):=B^{d;r}_G(\rz^{n-1}\times\rpbar;\frakg)
    +\wtbfB^{d,\nu;r}_G(\rz^{n-1}\times\rpbar;\frakg).$$
\end{definition}

It is obvious by construction that 
\begin{align*} 
 \bfB^{d,\nu;r}_G(\rz^{n-1}\times\rpbar;\frakg)
  &=\wtbfB^{d,\nu;r}_G(\rz^{n-1}\times\rpbar;\frakg),\qquad\nu\le0,\\
 \bfB^{d,\nu;r}_G(\rz^{n-1}\times\rpbar;\frakg)
  &\subset\wtbfB^{d,0;r}_G(\rz^{n-1}\times\rpbar;\frakg),\qquad\nu\ge1.
\end{align*} 
Therefore we can associate with $\bfg\in\bfB^{d,\nu;r}_G(\rz^{n-1}\times\rpbar;\frakg)$ its 
homogeneous principal symbol $\bfg^{(d,\nu_-)}$ and the principal 
limit-symbol $\bfg^\infty_{[d,\nu_-]}$ where $\nu_-:=\min(0,\nu)$.

\begin{definition}\label{def:boundary-symbol-new}
Let $d\in\gz$, $\nu,r\in\nz_0$, and $\frakg=((L_0,M_0),(L_1,M_1))$. Then 
 $$\bfB^{d,\nu;r}(\rz^{n-1}\times\rpbar;\frakg):=B^{d;r}(\rz^{n-1}\times\rpbar;\frakg)
    +\wtbfB^{d,\nu;r}_G(\rz^{n-1}\times\rpbar;\frakg).$$
The elements of $\bfB^{d,\nu;r}(\rz^{n-1}\times\rpbar;\frakg)$ are called boundary symbols. 
The class of regularizing boundary symbols is 
 $$\bfB^{-\infty,\nu;r}(\rz^{n-1}\times\rpbar;\frakg)=B^{-\infty;r}(\rz^{n-1}\times\rpbar;\frakg).$$
\end{definition}

Equivalently, any $\bfp\in \bfB^{d,\nu;r}(\rz^{n-1}\times\rpbar;\frakg)$ has the form 
 $$\bfp(x',\xi',\mu)
     =\begin{pmatrix}\op^+(a)(x',\xi',\mu)&0\\0&0\end{pmatrix}+\bfg(x^\prime,\xi^\prime,\mu),$$
where $a\in S^d(\rz^n\times\rpbar;\scrL(\cz^{L_0},\cz^{L_1}))$ satisfies the transmission condition 
and $\bfg$ is a generalized singular Green symbol from $\bfB^{d,\nu;r}_G(\rz^{n-1}\times\rpbar;\frakg)$.  
Then the homogeneous principal symbol of  $\bfp$ is  
\begin{align}\label{eq:b-homogeneous-principal}
 \sigma_\psi^{(d)}(\bfp)(x,\xi,\mu)=a^{(d)}(x,\xi,\mu),\qquad  x_n\ge0,\quad (\xi,\mu)\not=0.
\end{align}

\begin{definition}\label{def:principal-symbols-new}
Representing $\bfp\in\bfB^{d,\nu;r}(\rz^{n-1}\times\rpbar;\frakg)$ as $\bfp=\bfp_0+\wt \bfg$ with 
$\bfp_0\in B^{d;r}(\rz^{n-1}\times\rpbar;\frakg)$ and 
$\wt\bfg\in\wtbfB^{d,\nu;r}_G(\rz^{n-1}\times\rpbar;\frakg)$, the principal boundary symbol of $\bfp$ is 
 $$\sigma_\partial^{(d,\nu)}(\bfp)(x',\xi',\mu)
    :=\sigma_\partial^{(d)}(\bfp_0)(x',\xi',\mu)+\wt\bfg^{(d,\nu)}(x',\xi',\mu),\qquad \xi'\not=0,$$
the principal limit symbol of $\bfp$ is 
\begin{align*}
 \sigma_\infty^{[d]}(\bfp)(x',\xi')
    :=\sigma_\partial^{(d)}(\bfp_0)(x',0,1)+\wt\bfg^\infty_{[d,0]}(x',\xi').
\end{align*}
\end{definition}

Note that, in the previous definition, $\wt\bfg^\infty_{[d,0]}=0$ in case $\nu\ge1$. The subordinate 
principal angular symbol is 
\begin{align*}
 \wh\sigma^{\spk{d}}(\bfp)(x',\xi')
 :&=\lim_{r\to0+}\sigma_\partial^{(d,\nu)}(\bfp)\Big(x',r\frac{\xi'}{|\xi'|},\sqrt{1-r^2}\Big)\\
  &=\sigma_\partial^{(d)}(\bfp_0)(x',0,1)+\wt\bfg_{\spk{d,0}}(x',\xi'),\qquad\xi'\not=0.
\end{align*}
Also here $\wt\bfg_{\spk{d,0}}=0$ in case $\nu\ge 1$. Moreover, $\wh\sigma^{\spk{d}}(\bfp)$ is the homogeneous 
principal symbol of the limit symbol $\sigma_\infty^{[d]}(\bfp)$, i.e., 
 $$\wh\sigma^{\spk{d}}(\bfp)=(\sigma_\infty^{[d]}(\bfp))^{(0)}.$$

\begin{remark}
For $\bfp\in\bfB^{d;r}(\rz^{n-1}\times\rpbar;\frakg)$ considered as an element of 
$\bfB^{d,0;r}(\rz^{n-1}\times\rpbar;\frakg)$ we have 
 $$\sigma_\infty^{[d]}(\bfp)(x',\xi')=\wh\sigma^{\spk{d}}(\bfp)(x',\xi')
    =\sigma_\partial^{(d)}(\bfp)(x',0,1),$$
i.e., both principal limit-symbol and principal angular symbol are determined by the principal boundary symbol. 
This is why ellipticity for symbols with regularity number $\nu\ge1$ $($thus, in particular, for strongly 
parameter-dependent symbols$)$ is covered by the homogeneous principal symbol and the boundary symbol alone. 
The principal limit-symbol will be relevant for symbols with regularity number $\nu=0$. 
\end{remark}

For a further description of the principal symbols we need to understand better the bahaviour of pseudodifferential 
operators on the half-axis under congugation with the dilation group-action. By direct calculation, if 
$a(x_n,\xi_n)\in S^d_\tr(\rz)$, then 
 $$\kappa_\lambda^{-1} \op^+(a)\kappa_\lambda=\op^+(a_\lambda),\qquad 
     a_\lambda(x_n,\xi_n)=a(x_n/\lambda,\lambda\xi_n)$$
for every $\lambda>0$. This already yields formula \eqref{eq:transmission} in the following lemma. 

\begin{lemma}\label{lem:transmission}
Let $a(x',\xi,\mu)\in S^{d}_\tr(\rz^n\times\rpbar;\scrL(\cz^{L_1},\cz^{L_2}))$ be independent of the 
$x_n$-variable. Define 
 $$a'(x',\xi',\mu;\xi_n)=a(x',\xi',[\xi',\mu]\xi_n,\mu).$$ 
Then $a'\in S^{d}(\rz^{n-1}\times\rpbar;S^d_\tr(\rz,\scrL(\cz^{L_1},\cz^{L_2})))$ and 
\begin{align}\label{eq:transmission}
 \kappa^{-1}(\xi',\mu)\op^+(a)(\xi',\mu)\kappa(\xi',\mu)=\op^+(a')(\xi',\mu).
\end{align}
The homogeneous principal symbol of $a'$ is 
 $$a'^{(d)}(x',\xi',\mu;\xi_n)=a^{(d)}(x',\xi',|\xi',\mu|\xi_n,\mu).$$
\end{lemma}
\begin{proof}
For convenience of notation let $L_1=L_2=1$. 
It is a well-known fact that $a'\in S^{d}_{1,0}(\rz^{n-1}\times\rpbar;S^d_\tr(\rz))$;
see, for example, \cite[p. 10]{Schr01} for a short explanation (where $\xi'$ has to be substituted by $(\xi',\mu))$. 
It remains to show the poly-homogeneity. 

Let $\chi(\xi',\mu)$ and $\wh\chi(\xi,\mu)$ be two zero-excision functions such that $\chi\wh\chi=\chi$ and 
$\chi(\xi',\mu)=0$ whenever $|\xi',\mu|\le1$. Then, for arbitrary $N\in\nz_0$, 
 $$a(x',\xi,\mu)=\sum_{j=0}^{N-1}\wh\chi(\xi,\mu)a^{(d-j)}(x',\xi,\mu)+r_N(x',\xi,\mu)$$
with $r_N\in S^{d-N}_\tr(\rz^n\times\rpbar)$. Therefore 
 $$a'(x',\xi',\mu;\xi_n)=\sum_{j=0}^{N-1}\wh\chi(\xi',[\xi',\mu]\xi_n,\mu)a^{(d-j)}(x',\xi',[\xi',\mu]\xi_n,\mu)+
    r_N'(x',\xi,\mu;\xi_n)$$
with $r_N'\in S^{d-N}_{1,0}(\rz^{n-1}\times\rpbar;S^{d-N}_\tr(\rz))\subset 
S^{d-N}_{1,0}(\rz^{n-1}\times\rpbar;S^{d}_\tr(\rz))$. Moreover, 
\begin{align*}
 \wh\chi(\xi',[\xi',\mu]\xi_n,\mu)a^{(d-j)}(x',\xi',[\xi',\mu]\xi_n,\mu)
    =&\,\chi(\xi',\mu)a^{(d-j)}(x',\xi',|\xi',\mu|\xi_n,\mu)+\\
     & +(1-\chi)(\xi',\mu)(\wh\chi a^{(d-j)})'(x',\xi',\mu;\xi_n).
\end{align*}
The second summand belongs to $S^{-\infty}(\rz^{n-1}\times\rpbar;S^{d-j}_\tr(\rz))$. Thus 
\begin{align*}
 a'(x',\xi',\mu;\xi_n)\equiv\sum_{j=0}^{N-1}\chi(\xi',\mu)a^{(d-j)}(x',\xi',|\xi',\mu|\xi_n,\mu)
\end{align*}
modulo $S^{d-N}_{1,0}(\rz^{n-1}\times\rpbar;S^{d}_\tr(\rz))$.
This completes the proof. 
\end{proof}

As a consequence of Lemma \ref{lem:transmission}, Theorem \ref{thm:conjugation02} and 
Definition \ref{def:conjugation01},
\begin{align*} 
 \bfkappa^{-1}_{|\xi',\mu|}\sigma_\partial^{(d,\nu)}(\bfp)(x',\xi',\mu)\bfkappa_{|\xi',\mu|}
 \in 
 \bfS^{d,\nu}_\Hom(\rz^{n-1}\times\rpbar;B^{d,r}(\rz_+;\frakg))
\end{align*}
where the latter space is that of \eqref{eq:sum-of-hom} with $F=B^{d,r}(\rz_+;\frakg)$. Moreover, 
\begin{align*}
 \sigma_\infty^{[d]}(\bfp)\in S^0(\rz^{n-1};B^{d;r}(\rz_+,\frakg)).
\end{align*}

\subsubsection{Composition with powers of $x_n$}\label{sec:5.1.1}

If $a=a(x_n)\in\scrC^\infty(\rpbar,\scrL(\cz^L))$ is a smooth function on $\rpbar$ 
we may consider the associated operator $\op^+(a)$, i.e. the operator of multiplication by $a$, as a symbol constant in  
the variables $(x',\xi',\mu)$. For simplicity of notation we shall use the same notation $a$ for the function $a$, the 
associated operator on $\rz_+$, and the associated pseudodifferential symbol. 

A case of particular importance is $a(x_n)=x_n^j\,\mathrm{id}_{\cz^L}$, $j\in\nz_0$. 
In this case we shall write simply $x_n^j$ for every choice of the value $L$. Note that 
  $$[(\kappa_\lambda^{-1} x_n^j \kappa_\lambda)u](x^n)
     =(t^j u(\lambda t)\big|_{t=\lambda^{-1}x_n}=\lambda^{-j} x_n^j u(x_n),$$
i.e., the symbol $x_n^j$ is twisted homogeneous of degree $-j$; in particular, 
$x_n^j\in S^{-j}(\rz^{n-1}\times\rpbar;E_0,E_1)$ 
for every choice of spaces $E_0$, $E_1$ carrying the dilation group-action such that $x_n^j\in\scrL(E_0,E_1)$.    
We will use the notation 
 $$\bfx_n^j:=\begin{pmatrix}x_n^j&0\\0&0\end{pmatrix}\;\in\; 
    S^{-j}(\rz^{n-1}\times\rpbar;E_0\oplus\cz^{M_0},E_1\oplus\cz^{M_1})$$
without explicitly indicating the involved spaces $E_\ell$ and numbers $M_j$. 

\begin{lemma}\label{lem:cut-off01}
Let $\frakB_G^{d;r}$ represent one choice of $B^{d;r}_G$, $\wtbfB^{d,\nu;r}_G$, or $\bfB^{d,\nu;r}_G$. 

If $\bfg\in \frakB^{d;r}_G(\rz^{n-1}\times\rpbar;\frakg)$, then  
 $\bfx_n^j\bfg,\;\bfg\bfx_n^j \in \frakB^{d-j;r}_G(\rz^{n-1}\times\rpbar;\frakg)$.  
\end{lemma}

Though rather obvious $($since $x_n^j$ is strongly parameter-dependent of order $-j)$, 
this result is technically very important for the calculus and will be used various times in the sequel.  

\forget{
Note that $\op(\bfx_n^j\bfg)(\mu)=\bfx_n^j\op(\bfg)(\mu)$ and $\op(\bfg\bfx^j_n)(\mu)=\op(\bfg)(\mu)\bfx_n^j$. 
Then using that $\omega$ 
(considered as symbol in $(x',\xi',\mu)$) is a (non poly-homogeneous) symbol of order zero and writing 
$(1-\omega)=x_n^N(x_n^N(1-\omega))$ for arbitrary $N$, we obtain the following corollary. 

\begin{proposition}
If $\bfg\in B^{d;r}_G(\rz^{n-1}\times\rpbar;\frakg)$, then  
 $$\bfg-\mathbf{\omega}\bfg, \bfg-\bfg\mathbf{\omega} \in 
    B^{-\infty;r}_G(\rz^{n-1}\times\rpbar;\frakg).$$ 
\end{proposition}
}

\subsection{Composiition and formal adjoint}\label{sec:5.2}
It is immediate from Theorems \ref{thm:adjoint-strong} and \ref{thm:adjoint-limit} that the calculus is closed 
under taking the formal adjoint:

\begin{theorem}\label{thm:b-adjoint}
Let $d\le0$. Then the map $\bfp\mapsto\bfp^{(*)}$ $($formal adjoint with respect to the $L^2$-inner product$)$ 
induces a continuous map 
 $$\bfB^{d,\nu;0}(\rz^{n-1}\times\rpbar;\frakg)\lra 
    \bfB^{d,\nu;0}(\rz^{n-1}\times\rpbar;\frakg^{-1}).$$
All principal symbols behave natural under taking the formal adjoint.
\end{theorem}
It will require more work to verify the closedness under composition. 

\begin{theorem}\label{thm:comp-new}
Let $\frakg_1$ be composable with $\frakg_0$. Then $(\bfp_1,\bfp_0)\mapsto \bfp_1\#\bfp_0$ induces continuous maps
\begin{align*} 
 \bfB^{d_1,\nu_1;r_1}&(\rz^{n-1}\times\rpbar;\frakg_1)\times  
     \bfB^{d_0,\nu_0;r_0}(\rz^{n-1}\times\rpbar;\frakg_0)\\
    &\lra \bfB^{d,\nu;r}(\rz^{n-1}\times\rpbar;\frakg_1\frakg_0),  
\end{align*}
where $($recall that $\nu_j\ge0)$
 $$d=d_0+d_1,\qquad \nu=\min(\nu_0,\nu_1),\qquad r=\max(r_1+d_0,r_0).$$
All principal symbols are multiplicative, i.e., 
\begin{align*} 
  \sigma_\psi^{(d_0+d_1)}(\bfp_1\#\bfp_0)
     &=\sigma_\psi^{(d_1)}(\bfp_1)\sigma_\psi^{(d_0)}(\bfp_0),\\
  \sigma_\partial^{(d_0+d_1,\nu)}(\bfp_1\#\bfp_0)
     &=\sigma_\partial^{(d_1,\nu_1)}(\bfp_1)\sigma_\partial^{(d_0,\nu_0)}(\bfp_0),\\
  \sigma_\infty^{[d_0+d_1]}(\bfp_1\#\bfp_0)
     &=\sigma_\infty^{[d_1]}(\bfp_1)\#\sigma_\infty^{[d_0]}(\bfp_0).
\end{align*}
Consequently, the principal angular symbol behaves multiplicatively. 
The class of generalized singular Green operators is a two-sided ideal with respect to composition. 
\end{theorem}

Before proving this result let us first make some informal discussion. Let us write 
$\bfp_j=\bfp_{j,0}+\wt\bfg_j$ with $\bfp_{j,0}\in B^{d_j;r_j}(\rz^{n-1}\times\rpbar;\frakg_j)$ and 
$\wt\bfg_j\in \wt\bfB^{d_j,\nu_j;r_j}_G(\rz^{n-1}\times\rpbar;\frakg_1)$. 
Then 
 $$\bfp_1\#\bfp_0=\bfp_{1,0}\#\bfp_{0,0}+\bfp_{1,0}\#\wt\bfg_0+\wt\bfg_1\#\bfp_{0,0}+\wt\bfg_1\#\wt\bfg_0.$$
The first and the last summand on the right-hand side are covered by Theorem \ref{thm:comp-strong}  and 
Theorem \ref{thm:comp-limit}, respectively. Hence it remains to analyze the mixed terms. Writing $\bfp_{j,0}$ as the 
sum of a pseudodifferential part and a generalized singular Green operator, the composition of the generalized 
singular Green operators is again covered by Theorem \ref{thm:comp-limit} together with the fact that 
$B^{d_j;r_j}_G(\rz^{n-1}\times\rpbar;\frakg_j)\subset \wt\bfB^{d_j,0;r_j}_G(\rz^{n-1}\times\rpbar;\frakg_j)$. 
Hence it remains to analyze compositions of the pseudodifferential part with the generalized singular Green operators. 
For the first mixed term this means to analyze compostions of the form $\op^+(a)\#g$ and $\op^+(a)\#k$ 
with a generalized singular Green symbol $g$ and a Poisson symbol $k$, for the second mixed term that of $g\#\op^+(a)$ 
and $t\#\op^+(a)$ with a trace symbol $t$, where $g,k,t$ belong to the corresponding $\wtbfB_G$-classes. 
We shall focus on the terms involving $g$, since the other terms are treated analogously. Using the explicit form of 
generalized singular Green symbols of type $r$ it is not difficult to see that the case of general type follows from that of 
type $r=0$.

\begin{lemma}\label{lem:transmission02}
Let $a(x',\xi,\mu)\in S^{d}_\tr(\rz^n\times\rpbar;\scrL(\cz^{L_1},\cz^{L_2}))$ be independent of the 
$x_n$-variable. Let $g_j\in\wtbfB^{d_j,\nu;r_j}_G(\rz^{n-1}\times\rpbar;\frakg_j)$, $j=0,2$, with 
$\frakg_j=((L_j,0),(L_{j+1},0))$ for $j=0,1,2$. Then  
\begin{align*}
 \op^+(a)g_0&\in \wtbfB^{d_0+d,\nu;r_0}_G(\rz^{n-1}\times\rpbar;\frakg_1\frakg_0),\\
 g_2\op^+(a)&\in \wtbfB^{d_2+d,\nu;r}_G(\rz^{n-1}\times\rpbar;\frakg_2\frakg_1),\qquad r=\max(r_2+d,0).
\end{align*}
Moreover, 
\begin{align*}
 (\op^+(a)g_0)^{(d_0+d,\nu)}(x',\xi',\mu)&=\op^+(a^{(d)})(x',\xi',\mu)g_0^{(d_0,\nu)}(x',\xi',\mu),\\
 (\op^+(a)g_0)^\infty_{[d_0+d,\nu]}(x',\xi')&=\op^+(a^{(d)})(x',0,1)g_{0,[d,\nu]}^\infty(x',\xi',\mu),
\end{align*}
and analogously for $g_2\op^+(a)$ with factors interchanged. 
\end{lemma}
\begin{proof}
First we recall the continuity of the maps  
\begin{align*}
 (a',g'_0)&\mapsto \op^+(a')g'_0:S^d_\tr(\rz;\scrL(\cz^{L_1},\cz^{L_2}))\times \Gamma^{r_0}_G(\rz_+;\frakg_0)
     \lra \Gamma^{r_0}_G(\rz_+;\frakg_1\frakg_0),\\
 (g'_2,a')&\mapsto g_2'\op^+(a'): \Gamma^{r_2}_G(\rz_+;\frakg_2)\times S^d_\tr(\rz;\scrL(\cz^{L_1},\cz^{L_2}))\times
     \lra \Gamma^{r}_G(\rz_+;\frakg_2\frakg_1).
\end{align*}
Then the result follows immediately, since 
\begin{align*}
 \kappa \op^+(a)g_0\kappa^{-1}&=\kappa\op^+(a)\kappa^{-1}\kappa g_0\kappa^{-1}=\op^+(a')g_0', \\
 \kappa g_2\op^+(a)\kappa^{-1}&=\kappa g_2\kappa^{-1}\kappa \op^+(a)\kappa^{-1}=g'_2\op^+(a'), 
\end{align*}
where $g_j'\in \wtbfS^{d_j,\nu}(\rz^{n-1}\times\rpbar;\Gamma^{r_j}_G(\rz_+;\frakg_j))$ and, due to 
Lemma \ref{lem:transmission}, $\op^+(a')\in S^d(\rz^{n-1}\times\rpbar;B^{d,0}(\rz_+;\frakg_1))$. 
\end{proof}

\begin{theorem}\label{thm:comp-left}
Let $\frakg_j=((L_j,0),(L_{j+1},0))$. Let $g_0\in\wtbfB^{d_0,\nu;0}_G(\rz^{n-1}\times\rpbar;\frakg_0)$ and 
$a\in S^{d}_\tr(\rz^{n}\times\rpbar;\scrL(\cz^{L_1},\cz^{L_2}))$. Then 
$\op^+(a)\#g_0\in\wtbfB^{d_0+d,\nu;0}_G(\rz^{n-1}\times\rpbar;\frakg_1\frakg_0)$  
and all principal symbols behave multiplicatively. 
\end{theorem}
\begin{proof}
For convenience of notation let us assume $L_0=L_1=L_2=1$, 
First we consider $\op^+(a)$ and $g_0$ as operator-valued symbols in the sense of Theorem \ref{thm:op-plus} and 
Definition \ref{def:singular-green03}, respectively. Then 
 $$\op^+(a)\#g_0=\sum_{|\alpha'|=0}^{N-1}\frac{1}{\alpha'!}
    \partial^{\alpha'}_{\xi'}\op^+(a)D^{\alpha'}_{x'}g_0+r_N,$$
for arbitrary $N$, where 
 $$r_N(x',\xi',\mu)=
    N\sum_{|\gamma'|=N}\int_0^1\frac{(1-\theta)^{N-1}}{\gamma'!}r_{\gamma',\theta}(x',\xi',\mu)\,d\theta$$
where 
 $$r_{\gamma',\theta}(x',\xi',\mu)
    =\mathrm{Os}-\iint e^{-iy'\eta'}\partial^{\gamma'}_{\eta'}\op^+(a)(x',\xi'+\theta\eta',\mu)
        D^{\gamma'}_{y'} g_0(x'+y',\xi',\mu)\,dy'\dbar\eta'.$$

Since $\partial^{\alpha'}_{\xi'}a\in S^{d-|\alpha'|}_\tr(\rz^{n}\times\rpbar)$, 
\begin{align*} 
 \partial^{\alpha'}_{\xi'}\op^+(a)
 &\in S^{d-|\alpha'|}_{1,0}(\rz^{n-1}\times\rpbar; H^{s,\delta}(\rz_+),H^{s-d,\delta}(\rz_+))\\
 &\subset \wt S^{d-|\alpha'|,0}_{1,0}(\rz^{n-1}\times\rpbar; H^{s,\delta}(\rz_+),H^{s-d,\delta}(\rz_+))
\end{align*} 
for every $s$ and $\delta$, the oscillatory integral defining $r_N$ converges in  
\begin{align}\label{eq:hilf01-1} 
 \mathop{\mbox{\Large$\cap$}}_{s,s',\delta,\delta'}\wt S^{d_0+d-N,\nu;0}_{1,0}(\rz^{n-1}\times\rpbar;
     H^{s,\delta}_0(\rpbar),H^{s',\delta'}(\rz_+)).
\end{align}
For $|\alpha'|\le N-1$ write 
 $$\partial^{\alpha'}_{\xi'}a(x,\xi,\mu)
     =\sum_{k=0}^{N-1-|\alpha'|}\frac{1}{k!}\partial_{x_n}^k\partial^{\alpha'}_{\xi'}a(x',0,\xi,\mu)x_n^k+
     x_n^{N-|\alpha'|}a_{\alpha',N}(x,\xi,\mu)$$
Since $x_n^\ell$ is twisted homogeneous of degree $-\ell$, we have 
 $$x_n^{N-|\alpha'|}\op^+(a_{\alpha',N})\in 
     \wt S^{d-N,0}_{1,0}(\rz^{n-1}\times\rpbar; H^{s,\delta}(\rz_+),H^{s-d,\delta-N+|\alpha'|}(\rz_+))$$
for every $s$ and $\delta$. Altogether, 
\begin{align*} 
 \op^+(a)\#g_0
 =\sum_{\ell=0}^{N-1}\sum_{|\alpha'|+k=\ell}\frac{1}{\alpha'!}   
     \op^+(\partial^{\alpha'}_{\xi'}a_k)x_n^kD^{\alpha'}_{x'}g_0+ \wt r_N,
\end{align*} 
where $a_k(x',\xi,\mu)=(\partial_{x_n}^ka)(x',0,\xi,\mu)/k!$ and $\wt r_N$ belongs to the space in \eqref{eq:hilf01-1}. 
By the previous Lemma \ref{lem:transmission02}, 
 $$\wt g_\ell:=\sum_{|\alpha'|+k=\ell}\frac{1}{\alpha'!}   
     \op^+(\partial^{\alpha'}_{\xi'}a_k)x_n^kD^{\alpha'}_{x'}g_0
     \in \wtbfB^{d_0+d-\ell,\nu;0}_G(\rz^{n-1}\times\rpbar;\frakg_1\frakg_0).$$
By asymptotic summation there exists a $\wt g\in\wtbfB^{d_0+d,\nu;0}_G(\rz^{n-1}\times\rpbar;\frakg_1\frakg_0)$ 
such that $\wt g-\sum_{\ell<N}\wt g_\ell\in\wtbfB^{d_0+d-N,\nu;0}_G(\rz^{n-1}\times\rpbar;\frakg)$ 
for every $N\in\nz_0$. In view of Remark \ref{rem:smoothing-operators} we can conclude that 
 $$\wt g-\op^+(a)\#g_0\in
     B^{-\infty;0}_G(\rz^{n-1}\times\rpbar;\frakg_1\frakg_0).$$
This shows the first claim of the theorem and that 
 $$\op^+(a)\#g_0=\op^+(a_0)g_0 \mod  \wtbfB^{d_0+d-1,\nu;0}_G(\rz^{n-1}\times\rpbar;\frakg_1\frakg_0).$$
The multiplicativity of the principal symbols then holds by Lemma \ref{lem:transmission02}. 
\end{proof}

\begin{theorem}\label{thm:comp-right}
Let $\frakg_j=((L_j,0),(L_{j+1},0))$. Let $g_2\in\wtbfB^{d_2,\nu;0}_G(\rz^{n-1}\times\rpbar;\frakg_2)$ and 
$a\in S^{d}_\tr(\rz^{n}\times\rpbar;\scrL(\cz^{L_1},\cz^{L_2}))$. Then 
$g_2\#\op^+(a)\in\wtbfB^{d_2+d,\nu;d_{+}}_G(\rz^{n-1}\times\rpbar;\frakg_2\frakg_1)$  
and all principal symbols behave multiplicatively. 
\end{theorem}
\begin{proof}
Suppose the claim is true in case $d_2\le 0$. For general $d_2$ we make use the reduction of orders 
$\lambda^{d}_-(\xi,\mu)$ from Section \ref{sec:2.3.1} by writing
 $$g_2\#\op^+(a)=\op^+(\lambda^{d_2}_-)\#(\wt g_2\#\op^+(a)),$$
where $\wt g_2:=\op^+(\lambda^{-d_2}_-)\#g_2\in \wtbfB^{0,\nu;0}_G\rz^{n-1}\times\rpbar;\frakg_2)$ due to 
Theorem \ref{thm:comp-left}. By the claim with $d_2=0$, $\wt g_2\#\op^+(a)$ belongs to 
$\wtbfB^{d,\nu;d_{+}}_G(\rz^{n-1}\times\rpbar;\frakg_2\frakg_1)$. Applying Theorem \ref{thm:comp-left}
yields  $g_2\#\op^+(a)\in\wtbfB^{d_2+d,\nu;d_{+}}_G(\rz^{n-1}\times\rpbar;\frakg_2\frakg_1)$. 
The multiplicativity of the principal symbols follows analogously, first applying the case $d_2=0$ to 
$\wt g_2\#\op^+(a)$ and then using twice Theorem \ref{thm:comp-left}. 

Next let us assume $d_2\le 0$. As a first case assume that $a$ has order $0$ and is independent of $\xi$, i.e., 
$a\in\scrC^\infty_b(\rz^n)$ and $\op^+(a)(x')$ is the operator of multiplication by $a(x',\cdot)$. Writing  
$g_2\#\op^+(a)=(\op^+(a^*)\#g_2^{(*)})^{(*)}$, 
the result follows from Theorems \ref{thm:comp-left} and \ref{thm:b-adjoint}.

Finally consider the general case $a\in S^d_\tr((\rz^{n}\times\rpbar;\scrL(\cz^{L_1},\cz^{L_2}))$. 
Let us introduce for this proof the space 
$S^d_{const}(\rz^{n}\times\rpbar;F)$ of poly-homogeneous symbols of $(\xi,\mu)$ with values in a Fréchet space $F$ 
and $S^d_{\tr,const}(\rz^{n}\times\rpbar;F)$ the subspace of symbols with the transmission property 
$($i.e., symbols whose homogeneous components satisfy the symmetry-relations of 
Definition \ref{def:transmission-property}). 
Being a closed subspace of a nuclear Fréchet space, $S^d_{\tr,const}(\rz^{n}\times\rpbar)$ is nuclear itself. 
It follows that 
 $$S^d_{\tr,const}(\rz^{n}\times\rpbar;F)=S^d_{\tr,const}(\rz^{n}\times\rpbar)\wh\otimes_\pi F$$
with the completed projective tensor product. Taking $F=\scrC^\infty_b(\rz^{n}_{x},\scrL(\cz^{L_1},\cz^{L_2}))$ 
we thus can represent $a$ in the form 
 $$a(x,\xi,\mu)=\sum_{\ell=0}^{+\infty}\lambda_\ell a_\ell(x)p_\ell(\xi,\mu)$$
where $(\lambda_\ell)\subset\cz$ is an absolutely summable sequence, $(a_\ell)\subset \scrC^\infty_b(\rz^{n})$ 
and $(p_\ell)\subset S^d_{\tr,const}(\rz^{n}\times\rpbar)$ are zero-sequences. Then 
 $$g_2\#\op^+(a)=\sum_{\ell=0}^{+\infty}\lambda_\ell \big(g_2\#\op^+(a_\ell)\big)\op^+(p_\ell)$$
and the result follows from the previous consideration and Lemma \ref{lem:transmission02}. 
\end{proof}

\subsection{Ellipticity and parametrix construction}\label{sec:5.3}
We shall need the following result: 

\begin{lemma}
Let $\bfg\in \bfB^{d,\nu;r}_G(\rz^{n-1}\times\rpbar;\frakg)$, $\nu\in\nz_0$ with $\bfg^{(d,0)}=0$. 
Then $\bfg\in \bfB^{d-1,\nu-1;r}_G(\rz^{n-1}\times\rpbar;\frakg)$. 
\end{lemma}
\begin{proof}
Only the case $\nu\ge1$ needs to be proved. Write $\bfg=\bfg_0+\wt\bfg$ with 
$\bfg_0\in B^{d;r}_G(\rz^{n-1}\times\rpbar;\frakg)$ and 
$\wt\bfg\in\wtbfB^{d,\nu;r}_G(\rz^{n-1}\times\rpbar;\frakg)$. 
Then $\bfg^{(d,0)}=\bfg_0^{(d)}+\wt\bfg^{(d,\nu)}$. 
Let $\chi(\xi,\mu)$ and $\wt\chi(\xi')$ be zero-excision functions such that $\chi\wt\chi=\wt\chi$.  Then 
 $$\bfg=(\bfg_0-\chi\bfg_0^{(d)})+(\wt\bfg-\wt\chi\wt\bfg^{(d,\nu)})+\bfr,\qquad 
    \bfr=\chi\bfg_0^{(d)}+\wt\chi\wt\bfg^{(d,\nu)}.$$
Now define 
\begin{align*}
 \bfg_0'^{(d)}&=\bfkappa_{|\xi',\mu|}^{-1}\bfg_0^{(d)}\bfkappa_{|\xi',\mu|}
 \in S^d_\Hom(\rz^{n-1}\times\rpbar;F),\\
 \wt\bfg'^{(d,\nu)}&= \bfkappa_{|\xi',\mu|}^{-1}\wt\bfg^{(d,\nu)}\bfkappa_{|\xi',\mu|}
 \in    \wtbfS^{d,\nu}_\Hom(\rz^{n-1}\times\rpbar;F), 
\end{align*}
where $F:=\Gamma^r_G(\rz_+;\frakg)$. Also let $\bfr'= \bfkappa_{|\xi',\mu|}^{-1}\bfr\bfkappa_{|\xi',\mu|}$. 

From $\bfg_0'^{(d)}=-\wt\bfg'^{(d,\nu)}$ it follows that 
$\bfg_0'^{(d)}\in\scrC^\infty(\sz^{n-1}_+,\scrC^\infty_b(\rz^{n-1},F))$ vanishes to order $\nu$ in $(\xi',\mu)=(0,1)$. 
As shown in Step 2 of the proof of \cite[Proposition 5.3]{Seil22-1} this implies 
$\chi\bfg_0'^{(d)}\in \wtbfS^{d,\nu}(\rz^{n-1}\times\rpbar;F)$. But then 
 $$\bfr'=\chi(\bfg_0'^{(d)}+\wt\chi\wt\bfg'^{(d,\nu)})=(1-\wt\chi)\chi\bfg_0'^{(d)}$$
together with $(1-\wt\chi)\in\wtbfS^{0-\infty,0-\infty}(\rz^{n-1}\times\rpbar)$ yields 
$\bfr'\in \wtbfS^{d-\infty,\nu-\infty}(\rz^{n-1}\times\rpbar;F)$, hence 
$\bfr\in \wtbfB^{d-\infty,\nu-\infty;r}_G(\rz^{n-1}\times\rpbar;\frakg)$. The claim follows. 
\end{proof}

\begin{definition}\label{def:ellipticity-full}
Let $d\in\gz$, $\nu\in\nz_0$, and $\frakg=((L,M_0),(L,M_1))$. 
A symbol $\bfp\in \bfB^{d,\nu;d_+}(\rz^{n-1}\times\rpbar;\frakg)$ 
is called elliptic provided 
\begin{itemize}
 \item[$($E$1)$] the homogeneous principal symbol is pointwise invertible and
    $$|\sigma_\psi^{(d)}(\bfp)^{-1}(x,\xi,\mu)|\lesssim 1\qquad \forall\;\substack{x\in\rz^n\\x_n\ge0}\quad
        \forall\;|\xi,\mu|=1,$$
 \item[$($E$2)$] for some $($and then for all$)$ $s>d_+-\frac{1}{2}$, the principal boundary symbol is pointwise 
  invertible and
    $$\|\sigma_\partial^{(d,\nu)}(\bfp)^{-1}(x',\xi',\mu)\|_{\scrL(H^{s-d}(\rz_+,\cz^{L})\oplus\cz^{M_1},
        H^s(\rz_+,\cz^{L})\oplus\cz^{M_0})}\lesssim 1$$
  uniformly for $x'\in\rz^{n-1}$ and $|\xi',\mu|=1$, $\xi'\not=0$, 
 \item[$($E$3)$] for some $($and then for all$)$ $s>d_+-\frac{1}{2}$ the following operator is an isomorphism$:$ 
 $$\op(\sigma^{[d]}_\infty(\bfp)):
     \begin{matrix} L^2(\rz^{n-1},H^s(\rz_+,\cz^{L}))\\ \oplus\\ L^2(\rz^{n-1},\cz^{M_0})\end{matrix}
     \lra
     \begin{matrix}L^2(\rz^{n-1},H^{s-d}(\rz_+,\cz^{L}) \\ \oplus \\ L^2(\rz^{n-1},\cz^{M_1})\end{matrix}.$$ 
\end{itemize}
\end{definition}
For the independence of condition $($E$2)$ of the specific value of $s$ see \cite[Theorem 3.30]{Schu-BVP} for instance.  
Due to the spectral-invariance of pseudodifferential operators (see the proof of 
Proposition \ref{prop:invertibility-one+g} for more 
details), an equivalent formulation of $($E$3)$ is then as follows: 
\begin{itemize}
 \item[$($E$3)$] For some $($and then for all$)$ $s>d_+-\frac{1}{2}$ the principal limit symbol is invertible 
  with respect to the Leibniz product in 
   $$S^0(\rz^{n-1};\scrL(H^{s}(\rz_+,\cz^{L})\oplus\cz^{M_0},
      H^{s-d}(\rz_+,\cz^{L})\oplus\cz^{M_1})),$$ 
  i.e., , there exists a symbol 
   $$\bfq_\infty\in S^0(\rz^{n-1};\scrL(H^{s-d}(\rz_+,\cz^{L})\oplus\cz^{M_1},
        H^s(\rz_+,\cz^{L})\oplus\cz^{M_0}))$$ 
  such that 
  $\sigma^{[d]}_\infty(\bfp)\#\bfq_\infty=1$ and $\bfq_\infty\#\sigma^{[d]}_\infty(\bfp)=1$.  
\end{itemize}

Let us now show the independence of $($E$3)$ of $s:$ 

\begin{lemma}\label{lem:independence}
If $($E$3)$ holds for some $s_0>d_+-\frac{1}{2}$ then it holds for all $s>d_+-\frac{1}{2}$. 
\end{lemma}
\begin{proof}
We consider the second version of $($E$3)$ and shall use the notation from 
Definition \ref{def:principal-symbols-new} and \eqref{eq:b-homogeneous-principal}. 
Assume $($E$3)$ holds for $s=s_0$. 

Define $b(x',\xi_n)=a^{(d)}(x',0,0,\xi_n,1)^{-1}$ and 
$\bfq(x'):=\begin{pmatrix}\op^+(b)(x')&0\\ 0&0\end{pmatrix}$. Then 
$\bfq\in S^0(\rz^{n-1}; B^{-d;0}(\rz_+;\bfg^{-1}))$ and 
\begin{align}\label{eq:ind01}
 \bfq\#\sigma_\infty^{[d]}(\bfp)=1-\bfg_L,\qquad \sigma_\infty^{[d]}(\bfp)\#\bfq=1-\bfg_R
\end{align}
with 
\begin{align*}
 \bfg_L\in & \, S^0(\rz^{n-1};\Gamma_G^{d_+}(\rz_+;\bfg^{-1}\frakg))\\
     &\subset S^0(\rz^{n-1};\scrL(H^{s_0}(\rz_+,\cz^{L})\oplus\cz^{M_0},
        H^{s_0}(\rz_+,\cz^{L})\oplus\cz^{M_0})),\\
 \bfg_R\in & \, S^0(\rz^{n-1};\Gamma_G^{(-d)_+}(\rz_+;\frakg\frakg^{-1}))\\
     &\subset S^0(\rz^{n-1};\scrL(H^{s_0-d}(\rz_+,\cz^{L})\oplus\cz^{M_1},
        H^{s_0-d}(\rz_+,\cz^{L})\oplus\cz^{M_1})). 
\end{align*}
Applying the Leibniz product with $\bfq_\infty$ from the right to the first identity in \eqref{eq:ind01} and from the left 
to the second identity gives $\bfq=\bfq_\infty-\bfg_L\#\bfq_\infty$ and 
$\bfq=\bfq_\infty-\bfq_\infty\#\bfg_R$. Resolving for $\bfq_\infty$ yields 
$\bfq_\infty=\bfq+\bfg_L\#\bfq+\bfg_L\#\bfq_\infty\#\bfq_R$. The second term on the right belongs to 
$S^0(\rz^{n-1}; \Gamma_G^{(-d)_+}(\rz_+;\bfg^{-1}))$. The same is true for the third term. In fact, write  
$\bfg_R=\sum_{j=0}^{(-d)_+} \bfg_R^j\,\boldsymbol{\partial}^j_+$ with 
$\bfg_R^j\in S^0(\rz^{n-1};\Gamma_G^0(\rz_+;\frakg\frakg^{-1}))$. 
Then, using the mapping property of generalized singular Green operators on $\rz_+$, 
 $$\bfq_\infty\#\bfg_R^j\in S^0(\rz^{n-1};\scrL(H^{s,\delta}_0(\rz_+,\cz^{L})\oplus\cz^{M_1},
        H^{s_0}(\rz_+,\cz^{L})\oplus\cz^{M_0}))$$
for all $s,\delta\in\rz$. Since 
 $$\bfg_L\in S^0(\rz^{n-1};\scrL(H^{s_0}(\rz_+,\cz^{L})\oplus\cz^{M_0},
        H^{s',\delta'}(\rz_+,\cz^{L})\oplus\cz^{M_0}))$$ 
for all $s',\delta'\in\rz$, 
it follows that $\bfg_L\#\bfq_\infty\#\bfg_R^j\in S^0(\rz^{n-1};\Gamma_G^0(\rz_+;\frakg^{-1}))$. 

Altogether we have verified that $\bfq_\infty\in S^0(\rz^{n-1};B^{-d;(-d)_+}(\rz_+;\frakg^{-1}))$. 
Using the density of $H^\infty(\rz_+,\cz^L)$ in $H^s(\rz_+,\cz^L)$ and the mapping properties of 
$\op(\bfq_\infty)$ and $\op(\sigma^{[d]}_\infty(\bfp))$, respectively, it follows that 
$\op(\bfq_\infty)$ is the inverse of $\op(\sigma^{[d]}_\infty(\bfp))$ in the first version of $($E$3)$ 
for any choice of $s$, hence $\bfq_\infty$ is the inverse of 
$\sigma^{[d]}_\infty(\bfp)$ with respect to the Leibniz product for any choice of $s$ in the second version 
of $($E$3)$. 
\end{proof}

\begin{definition}\label{def:parametrix01-1}
Let $d\in\gz$, $\nu\in\nz_0$. A \emph{parametrix} for $\bfp\in \bfB^{d,\nu;d_+}(\rz^{n-1}\times\rpbar;\frakg)$, 
$\frakg=((L,M_0),(L,M_1))$ is any symbol  
$\bfq\in \bfB^{-d,\nu;(-d)_+}(\rz^{n-1}\times\rpbar;\frakg^{-1})$ such that 
\begin{align*} 
 1-\bfq\#\bfp&\in B^{-\infty;d_+}(\rz^{n-1}\times\rpbar;\frakg^{-1}\frakg),\\
 1-\bfp\#\bfq&\in B^{-\infty;(-d)_+}(\rz^{n-1}\times\rpbar;\frakg\frakg^{-1}).
\end{align*}
\end{definition}

\begin{theorem}\label{thm:parametrix-new}
Let $\bfp\in \bfB^{d,\nu;d_+}(\rz^{n-1}\times\rpbar;\frakg)$. Then 
$\bfp$ is elliptic if and only if $\bfp$ has a parametrix. 
In this case there exists a $\mu_0$ such that $\op(\bfp)(\mu)$ is invertible for $\mu\ge\mu_0$ and 
the parametrix $\bfq$ can be chosen in such a way that
 $$\op(\bfq)(\mu)=\op(\bfp)(\mu)^{-1}\qquad\forall\;\mu\ge\mu_0.$$ 
\end{theorem}

The remaining part of this section is dedicated to the proof of Theorem \ref{thm:parametrix-new}. 
Using the order reductions from Section \ref{sec:2.3.1} and an analogue of Remark \ref{rem:order-reduction} we may assume 
without loss of generality that $\bfp\in \bfB^{0,\nu;0}(\rz^{n-1}\times\rpbar;\frakg)$, i.e., $\bfp$ is of order and type zero. 
Correspondingly, the parametrix will also be of order and type zero. 

Let $\bfp$ have the pseudodifferential part $\op^+(a)$. 
By the invertibility of $\sigma^{(0)}_\psi(\bfp)$ we find a symbol 
$b\in S^{0}_\tr(\rz^{n}\times\rpbar;\scrL(\cz^L))$ such that 
$(1-a\#b)(x,\xi,\mu)$ and $(1-b\#a)(x,\xi,\mu)$ are of order $-\infty$ on 
$\rz^n_+\times\rz^n\times\rpbar$. 
Thus if $\bfq_0:=\begin{pmatrix}\op^+(b)&0\\0&0\end{pmatrix}$ then 
\begin{align}\label{eq:param-new01} 
 \bfp\#\bfq_0\equiv 1-\bfg_R,\qquad \bfq_0\#\bfp\equiv 1-\bfg_L, 
\end{align}
modulo regularizing symbols of type zero, where 
$\bfg_L\in \bfB^{0,\nu;0}_G(\rz^{n-1}\times\rpbar;\frakg^{-1}\frakg)$ and  
$\bfg_R\in \bfB^{0,\nu;0}_G(\rz^{n-1}\times\rpbar;\frakg\frakg^{-1})$. 
Applying $\sigma_\partial^{(0,\nu)}$ to these identities  
and resolving for $\sigma_\partial^{(0,\nu)}(\bfp)^{-1}$ we find 
\begin{align}\label{eq:param-new02} 
    \sigma_\partial^{(0,\nu)}(\bfp)^{-1}=\sigma_\partial^{(0,\nu)}(\bfq_0)
     +\sigma_\partial^{(0,\nu)}(\bfq_0)\bfg_R^{(0,\nu)}
     +\bfg_L^{(0,\nu)}\sigma_\partial^{(0,\nu)}(\bfp)^{-1}\bfg_R^{(0,\nu)}.
\end{align}

\begin{lemma}
The third term on the right-hand side of $\eqref{eq:param-new02}$ is the homogeneous principal symbol of a 
generalized singular Green symbol from the class $\bfB^{0,\nu;0}_G(\rz^{n-1}\times\rpbar;\frakg^{-1})$. 
\end{lemma}
\begin{proof}
Let $p'(x',\xi',\mu)=\bfkappa_{|\xi',\mu|}^{-1}\sigma^{(0,\nu)}_\partial(\bfp)(x',\xi',\mu)\bfkappa_{|\xi',\mu|}$. Then 
 $$p'\in \bfS^{0,\nu}_\Hom(\rz^{n-1}\times\rpbar;
    \scrL(H^s(\rz_+,\cz^{L})\oplus\cz^{M_0}),H^{s}(\rz_+,\cz^{L})\oplus\cz^{M_1})).$$
Due to ellipticity condition $($E$2)$, $p'$ satisfies assumption i) of Lemma \ref{lem:invertibility02}. 
By definition, the principal angular symbol of $p'$ coincides with the principal angular symbol of $\bfp$. 
The latter coincides with the homogeneous principal symbol of the limit symbol of $\bfp$; by ellipticity condition $($E$3)$, 
it is invertible and the inverse is the homogeneous principal symbol of $q_\infty$. Hence $p'$ also satisfies assumption ii) 
of Lemma \ref{lem:invertibility02}. In conclusion 
 $$p'^{-1}\in \bfS^{0,\nu}_\Hom(\rz^{n-1}\times\rpbar;
    \scrL(H^{s}(\rz_+,\cz^{L})\oplus\cz^{M_1}),H^s(\rz_+,\cz^{L})\oplus\cz^{M_0}))$$
and $p'(x',\xi',\mu)^{-1}=\bfkappa_{|\xi',\mu|}^{-1}\sigma^{(0,\nu)}_\partial(\bfp)(x',\xi',\mu)^{-1}
\bfkappa_{|\xi',\mu|}$. 
Conjugating $\bfg_L^{(0,\nu)}$ and $\bfg_R^{(0,\nu)}$ with $\bfkappa_{|\xi',\mu|}$ yields symbols 
$\bfg_L'^{(0,\nu)}$ and $\bfg_R'^{(0,\nu)}$ from 
$\wtbfS^{0,\nu}_\Hom(\rz^{n-1}\times\rpbar;\Gamma^0_G(\rz_+;\frakg^{-1}\frakg))$ and 
$\wtbfS^{0,\nu}_\Hom(\rz^{n-1}\times\rpbar;\Gamma^0_G(\rz_+;\frakg\frakg^{-1}))$, respectively. 
The third term on the right-hand side of $\eqref{eq:param-new02}$ conjugated with $\bfkappa_{|\xi',\mu|}$ 
coincides with $\bfg_L'^{(0,\nu)}p'\bfg_R'^{(0,\nu)}$, hence belongs to 
$\wtbfS^{0,\nu}_\Hom(\rz^{n-1}\times\rpbar;\Gamma^0_G(\rz_+;\frakg^{-1}))$. This shows the claim. 
\end{proof}

By the previous lemma thus there exists a 
$\bfg\in \bfB^{0,\nu;0}_G(\rz^{n-1}\times\rpbar;\frakg^{-1})$ with 
 $$\bfg^{(0,\nu)}= \sigma_\partial^{(0,\nu)}(\bfq_0)\bfg_R^{(0,\nu)}
     +\bfg_L^{(0,\nu)}\sigma_\partial^{(0,\nu)}(\bfp)^{-1}\bfg_R^{(0,\nu)}.$$
Renaming $\bfq_0+\bfg$ as $\bfq_0$ again, we thus find the identities \eqref{eq:param-new01} where, in addition, 
the homogeneous principal symbols $\bfg_L^{(0,\nu)}$ and $\bfg_R^{(0,\nu)}$ vanish, i.e., 
$\bfg_L$ and $\bfg_R$ belong to $\bfB^{-1,\nu-1;0}_G(\rz^{n-1}\times\rpbar;\frakg^{-1}\frakg)$ and 
 $\bfB^{-1,\nu-1;0}_G(\rz^{n-1}\times\rpbar;\frakg\frakg^{-1})$, respectively 

Sofar the reasoning is the same for any value $\nu\in\nz_0$. 
However, from here on we must distinguish the cases $\nu\ge1$ and $\nu=0$. 

In case $\nu-1\ge0$ we have $\bfg_R^{\#\ell}\in \bfB^{-\ell,\nu-1;0}_G(\rz^{n-1}\times\rpbar;\frakg^{-1}\frakg)$ 
for every $\ell$. 
By asymptotic summation there exists $\bfg_0\in \bfB^{-1,\nu-1;0}_G(\rz^{n-1}\times\rpbar;\frakg^{-1}\frakg)$ such 
that, for every $N\in\nz$, 
 $$\bfg_0-\sum_{\ell=1}^{N-1}\bfg_R^{\#\ell}\in \bfB^{-N,\nu-1;0}_G(\rz^{n-1}\times\rpbar;\frakg^{-1}\frakg).$$
Then $\bfq:=\bfq_0+\bfq_0\#\bfg_0\in \bfB^{0,\nu;0}(\rz^{n-1}\times\rpbar;\frakg^{-1})$ satisfies 
$\bfp\#\bfq\equiv 1-\bfg_{R}$ and $\bfq\#\bfp\equiv 1-\bfg_{L}$ modulo regularizing symbols of type zero, where  
 $$\bfg_{L}\in \mathop{\mbox{\Large$\cap$}}_{N\in\nz}
     \bfB^{-1-N,\nu-1;0}_G(\rz^{n-1}\times\rpbar;\frakg\frakg^{-1})\subset
     B^{-\infty;0}(\rz^{n-1}\times\rpbar;\frakg\frakg^{-1}).$$
and analogously for $\bfg_R$. Thus we have found 
\begin{align}\label{eq:param-new03} 
\begin{split}
 \bfr_R&:=1-\bfp\#\bfq\in B^{-\infty;0}(\rz^{n-1}\times\rpbar;\frakg^{-1}\frakg),\\
 \bfr_L&:=1 -\bfq\#\bfp\in B^{-\infty;0}(\rz^{n-1}\times\rpbar;\frakg\frakg^{-1}). 
\end{split}
\end{align}
Also in case $\nu=0$, which we will consider now, we shall construct a parametrix $\bfq$ satisfying 
\eqref{eq:param-new03}; however, the reasoning is more complex and makes use of the principal limit symbol. 
Let us return to the above point where we have constructed 
$\bfq_0\in \bfB^{0,\nu;0}(\rz^{n-1}\times\rpbar;\frakg^{-1})$ such that 
 $$ \bfp\#\bfq_0\equiv 1-\bfg_R,\qquad \bfq_0\#\bfp\equiv 1-\bfg_L$$ 
modulo regularizing symbols of type zero, where 
$\bfg_L\in \bfB^{-1,-1;0}_G(\rz^{n-1}\times\rpbar;\frakg\frakg^{-1})=
\wtbfB^{-1,-1;0}_G(\rz^{n-1}\times\rpbar;\frakg\frakg^{-1})$ and analogously 
$\bfg_R\in \wtbfB^{-1,-1;0}_G(\rz^{n-1}\times\rpbar;\frakg^{-1}\frakg)$. Noting that 
$\bfg_R^{\#\ell}\in \wtbfB^{-\ell,-\ell;0}_G(\rz^{n-1}\times\rpbar;\frakg^{-1}\frakg)$, by asymptotic summation 
we can construct $\bfg_0\in \wtbfB^{-1,-1;0}_G(\rz^{n-1}\times\rpbar;\frakg^{-1}\frakg)$ such that 
$\bfq_1:=\bfq_0+\bfq_0\#\bfg_0$ satisfies
\begin{align*}
 \bfp\#\bfq_1\equiv 1-\bfg_R,\qquad \bfq_1\#\bfp\equiv 1-\bfg_L, 
\end{align*}
modulo regularizing symbols of type zero where 
 $$\bfg_L\in \wtbfB^{0-\infty,0-\infty;0}_G(\rz^{n-1}\times\rpbar;\frakg\frakg^{-1}),\qquad 
    \bfg_R\in \wtbfB^{0-\infty,0-\infty;0}_G(\rz^{n-1}\times\rpbar;\frakg^{-1}\frakg).$$
Now we apply $\sigma^{[0]}_\infty$ to these identities and resolve for $\bfq_\infty$ (see ellipticity condition $($E$3))$; 
this yields
\begin{align*}
  q_\infty=\sigma_\infty^{[0]}(\bfq_1)
     +\sigma_\infty^{[0]}(\bfq_1)\#\bfg^\infty_{R,[0,0]}
     +\bfg^\infty_{L,[0,0]}\# \bfq_\infty\# \bfg^\infty_{R,[0,0]}.
\end{align*}
Note that $\bfg^\infty_{L/R,[0,0]}=\bfg^\infty_{L/R,[-N,-N]}$ for every $N\in\nz$ and therefore 
 $$\bfg^\infty_{L,[0,0])}\in S^{-\infty}(\rz^{n-1};\Gamma^0_G(\rz_+;\frakg\frakg^{-1})),\qquad 
     \bfg^\infty_{R,[0,0])}\in S^{-\infty}(\rz^{n-1};\Gamma^0_G(\rz_+;\frakg^{-1}\frakg)).$$ 
It follows that 
   $$\bfg'_1:=\sigma_\infty^{[0]}(\bfq_1)\#\bfg^\infty_{R,[0,0]}
      +\bfg^\infty_{L,[0,0]}\# \bfq_\infty\# \bfg^\infty_{R,[0,0]}\in 
      S^{-\infty}(\rz^{n-1};\Gamma^0_G(\rz_+;\frakg^{-1})).$$
Since 
 $$S^{-\infty}(\rz^{n-1};\Gamma^0_G(\rz_+;\frakg^{-1}))\subset 
    \wtbfS^{0-\infty,0-\infty}(\rz^{n-1}\times\rpbar;\Gamma^0_G(\rz_+;\frakg^{-1})),$$
$\bfg_1:=\bfkappa \bfg'_1\bfkappa^{-1}$ lies in $\wtbfB^{0-\infty,0-\infty;0}_G(\rz^{n-1}\times\rpbar;\frakg^{-1})$ 
and has principal limit symbol $\bfg^\infty_{1,[0,0]}=\bfg'_1$. 
Hence $\bfq_2:=\bfq_1+\bfq_1\#\bfg_1$ satisfies
\begin{align*}
 \bfp\#\bfq_2\equiv 1-\bfg_R,\qquad \bfq_2\#\bfp\equiv 1-\bfg_L, 
\end{align*}
modulo regularizing symbols of type zero where 
 $$\bfg_L\in \wtbfB^{0-\infty,0-\infty;0}_G(\rz^{n-1}\times\rpbar;\frakg\frakg^{-1}),\qquad 
    \bfg_R\in \wtbfB^{0-\infty,0-\infty;0}_G(\rz^{n-1}\times\rpbar;\frakg^{-1}\frakg)$$
have both vanishing principal limit symbol, i.e., $\bfg^\infty_{L/R,[0,0]}=0$. By the following 
Proposition \ref{prop:invertibility-one+g} there exists 
$\bfg_2\in \wtbfB^{0-\infty,0-\infty;0}_G(\rz^{n-1}\times\rpbar;\frakg^{-1}\frakg)$ such that 
$\bfq:=\bfq_2+\bfq_2\#\bfg_2$ satisfies \eqref{eq:param-new03}. 

\forget{
\begin{lemma}\label{lem:spectral-inv}
Let $\frakg=((L,M),(L,M))$ and $E:=L^2(\rz_+,\cz^L)\oplus\cz^M$. Then 
$\calA:=\{1-\op(\bfp)\mid \bfp\in S^{-\infty}(\rz^{n-1};\Gamma^0_G(\rz_+;\frakg))\}$ is a spectral-invariant 
subalgebra of $\scrL(L^2(\rz^{n-1},E))$, i.e., 
 $$\calA\cap\scrL(L^2(\rz^{n-1},E))^{-1}=\calA^{-1},$$ 
where $X^{-1}$ denotes the group of invertible elements of $X$. 
\end{lemma}
\begin{proof}
It is known that 
$\wt\calA:=\{\op(\bfp)\mid \bfp(x',\xi')\in \scrC^\infty_b(\rz^{n-1}\times\rz^{n-1};\scrL(E))\}$ 
is a spectral-invariant subalgebra 
of $\scrL(L^2(\rz^{n-1},E))$ $($this is true by (a slight extension of) the classical result of Beals on the characterization 
of pseudodifferential opeartors by mapping properties of certain commutators; for a proof see, for instance, 
\cite[Corollary A.3]{Seil22-2}). Note that $\calA\subset\wt\calA$. Thus, if $1-\op(\bfp)\in\calA$ is invertible in 
$\scrL(L^2(\rz^{n-1},E))$ then $(1-\op(\bfp))^{-1}=\op(\bfq)$ with $\bfq\in \calA)$. 
But then 
 $$(1-\op(\bfp))^{-1}=1+\op(\bfp)+\op(\bfp)(1-\op(\bfp))^{-1}\op(\bfp)
    =1+\op(\bfp+\bfp\#\bfq\#\bfp)$$
shows that $(1-\op(\bfp))^{-1}=1-\op(\wt\bfp)$ with $\wt\bfp\in S^{0}_{1,0}(\rz^{n-1};\Gamma^0_G(\rz_+;\frakg))$. 
For the latter note that 
 $$(p_0,q,p_1)\mapsto p_0qp_1:
     \Gamma^0_G(\rz_+;\frakg)\times\scrL(E)\times\Gamma^0_G(\rz_+;\frakg)\lra 
    \Gamma^0_G(\rz_+;\frakg)$$
is continuous. 
\end{proof}
}

\begin{proposition}\label{prop:invertibility-one+g}
Let $\bfg\in \wtbfB^{0-\infty,0-\infty;0}_G(\rz^{n-1}\times\rpbar;\frakg)$, $\frakg=((L,M),(L,M))$, have vanishing 
principal limit symbol. Then $1+\op(\bfg)(\mu)$ is invertible for large $\mu\in\rpbar$ and there exists an 
$\bfh\in \wtbfB^{0-\infty,0-\infty;0}_G(\rz^{n-1}\times\rpbar;\frakg^{-1})$ such that, for these $\mu$,  
 $$(1-\op(\bfg))(\mu)^{-1}=(1-\op(\bfh))(\mu).$$
\end{proposition}
\begin{proof}
By Theorem \ref{thm:conjugation03} and Corollary \ref{cor:comp03}, 
 $$\bfg':=\bfkappa^{-1}\#\bfg\#\bfkappa
    \in\wtbfS^{0-\infty,0-\infty}(\rz^{n-1}\times\rpbar;\Gamma^0_G(\rz_+,\frakg))$$ 
has vanishing principal limit symbol. By Proposition \ref{prop:smoothing} this is equivalent to 
 $$\bfg'\in S^{-1}(\rpbar;S^{-\infty}(\rz^{n-1};\Gamma^0_G(\rz_+,\frakg))).$$
Then $\bfg'^{\#\ell}\in S^{-\ell}(\rpbar;S^{-\infty}(\rz^{n-1}\times\rpbar;\Gamma^0_G(\rz_+,\frakg)))$ and, 
by asymptotic summation, there exists 
$\bfh'_0\in\wt S^{-1}(\rpbar;S^{-\infty}(\rz^{n-1};\Gamma^0_G(\rz_+,\frakg)))$ such that 
 $$\bfh'_0-\sum_{\ell=1}^{N-1}\bfg'^{\#\ell}\in 
     S^{-1-N}(\rpbar;S^{-\infty}(\rz^{n-1};\Gamma^0_G(\rz_+,\frakg)))$$
for every $N\in\nz_0$. Therefore, 
 $$(1-\bfh_0')\#(1-\bfg')=1-\bfr'_L,\qquad (1-\bfg')\#(1-\bfh_0')=1-\bfr'_R$$
with $\bfr'_L,\bfr'_R \in \wt S^{-\infty}(\rpbar;S^{-\infty}(\rz^{n-1};\Gamma^0_G(\rz_+,\frakg)))$. 

Next we note that by (a slight extension of) a classical result of Beals \cite{Beal77,Beal79} on the characterization 
of pseudodifferential opeartors via mapping properties of certain commutators (for a proof see, for instance, 
\cite[Corollary A.3]{Seil22-2}),  
 $$\calA:=\{\op(\bfp)\mid \bfp(x',\xi')\in \scrC^\infty_b(\rz^{n-1}\times\rz^{n-1};\scrL(E))\}, \quad 
    E:=L^2(\rz_+,\cz^L)\oplus\cz^M,$$ 
is spectral-invariant in $\scrL(L^2(\rz^{n-1},E))$, i.e., an element of $\calA$ has an inverse in $\calA$ if and only 
if it is invertible as an operator in $\scrL(L^2(\rz^{n-1},E))$. Moreover, inversion is a continuous operation in $\calA$. 
From this it follows that there exists a $\bfp'(\mu)\in\scrC^\infty_b(\rpbar;\calA)$ such that 
$\op(\bfp')(\mu)=\chi(\mu)(1-\op(\bfr'_L))(\mu)^{-1}$ for a suitable zero-excision function $\chi$. Then 
\begin{align*}
 (1-\op(\bfr'_L)(\mu)^{-1}&=1+\op(\bfr'_L)(\mu)+\op(\bfr'_L)(\mu)(1-\op(\bfr'_L)(\mu)^{-1}\op(\bfr'_L)(\mu)\\
 &=1+\op(\bfr'_L+\bfr'_L\#\bfp'\#\bfr'_L)(\mu)
\end{align*}
shows the existence of $\bfs'_L\in S^{-\infty}(\rpbar;S^{-\infty}(\rz^{n-1};\Gamma^0_G(\rz_+,\frakg)))$ 
such that $(1-\op(\bfr'_L))(\mu)^{-1}=1+\op(\bfs'_L)(\mu)$ for large $\mu$. 
For the latter note that 
 $$(r_0,p,r_1)\mapsto r_0pr_1:
     \Gamma^0_G(\rz_+;\frakg)\times\scrL(E)\times\Gamma^0_G(\rz_+;\frakg)\lra 
    \Gamma^0_G(\rz_+;\frakg)$$
is continuous. Analogously we find $\bfs'_R$ for $1-\bfr'_R$. 

Thus if $\bfh':=\bfh'_0-\bfs'_L+\bfs'_L\#\bfh'_0$, then $(1-\op(\bfg'))(\mu)^{-1}=1-\op(\bfh')(\mu)$ for large 
enough $\mu$. Then the claim follows with $\bfh:=\bfkappa\#\bfh'\#\bfkappa^{-1}$. 
\end{proof}

Summing up, we have constructed $\bfq$ with \eqref{eq:param-new03} in both cases $\nu\ge1$ and $\nu=0$. 
Then the following lemma completes the proof of Theorem \ref{thm:parametrix-new}. 

\begin{lemma}
Let $\bfr\in B^{-\infty;0}(\rz^{n-1}\times\rpbar;\frakg)$, $\frakg=((L,M),(L,M))$. 
Then $1-\op(\bfr)(\mu)$ is invertible for large $\mu\in\rpbar$ and there exists an 
$\bfs\in B^{-\infty;0}(\rz^{n-1}\times\rpbar;\frakg)$ such that, for these $\mu$,  
 $$(1-\op(\bfr))(\mu)^{-1}=(1-\op(\bfs))(\mu).$$
\end{lemma}
\begin{proof}
If $\bfr=\bfg\in B^{-\infty;0}_G(\rz^{n-1}\times\rpbar;\frakg)$ the proof is a simpler version of that of 
Proposition \ref{prop:invertibility-one+g}; we leave the details to the reader. 

Let $\op^+(a)(\mu)$ be the pseudodifferential part of $\bfr$, where $a\in S^{-\infty}(\rz^n\times\rpbar;\scrL(\cz^L))$. 
Then there exists a $b\in S^{-\infty}(\rz^n\times\rpbar;\scrL(\cz^L))$ such that 
 $$(1-\op(a)(\mu))^{-1}=(1-\op(b)(\mu)$$
for large enough $\mu$ as pseudodifferential operators on $\rz^n$. But then
\begin{align*}
 1-\op^+(1-a)\#\op^+(1-b)
 &=\op^+((1-a)\#(1-b))-\op^+(1-a)\#\op^+(1-b)\\
 &=\op^+(a\#b)-\op^+(a)\#\op^+(b)=:g
\end{align*}
with $g \in B^{-\infty;0}_G(\rz^{n-1}\times\rpbar;(L,0),(L,0))$. This yields 
 $$(1-\bfr)\#\begin{pmatrix}\op^+(1-b)&0\\0&1\end{pmatrix}=1+\bfg$$
with $\bfg\in B^{-\infty;0}_G(\rz^{n-1}\times\rpbar;\frakg)$. Now we are in the situation from the beginning 
of the proof. 
\end{proof}

\section{Operators of Toeplitz type}\label{sec:6}

\subsection{An abstract framework for parameter-dependent operators}\label{sec:6.1}

In this section we present some results of \cite{Seil12} that we shall apply below to boundary value problems. 
For convenience of the reader we adapt the formalism to better fit with the current paper and include a 
self-contained proof of Theorem \ref{thm:param-toeplitz}.  

We begin with a number of definitions. Let 
 $$\frakG=\{\frakg=(g_0,g_1)\mid g_0,g_1\in G\}$$
be a set of ``weight-data''. The inverse weight-datum of $\frakg=(g_0,g_1)$ is $\frakg^{-1}=(g_1,g_0)$, 
for two weight-data $\frakg_j=(g_j,g_{j+1})$, $j=0,1$, let $\frakg_1\frakg_0=(g_0,g_2)$. 
Assume that with every $g\in G$ is associated a Hilbert space $H(g)$ and with  
every $\frakg=(g_0,g_1)\in\frakG$  there are associated two vector-spaces of parameter-dependent operators 
$L^0(\rpbar;\frakg)$ and $L^{-\infty}(\rpbar;\frakg)$, the first containing the second, such that:
\begin{itemize}
 \item[$(1)$] If $A\in L^0(\rpbar;\frakg)$ then 
  $A(\mu)\in\scrL(H(g_0),H(g_1))$ for every $\mu\in\rpbar$, 
 \item[$(2)$] Taking ($\mu$-wise) the adjoint of operators induces maps 
  $$L^{d}(\rpbar;\frakg)\lra L^{d}(\rpbar;\frakg^{-1}),\qquad d\in\{0,-\infty\},$$
 \item[$(3)$] ($\mu$-wise) composition of operators induces maps 
  $$L^{d_1}(\rpbar;\frakg_1)\times L^{d_0}(\rpbar;\frakg_0)\lra L^{d_0+d_1}(\rpbar;\frakg_1\frakg_0)$$
  whenever $\frakg_1\frakg_0$ is defined and $d_0,d_1\in\{0,-\infty\}$. 
\end{itemize}

Writing simply $1$ for the identity map in each $H(g)$, we require that $1\in L^{0}(\rpbar;(g,g))$ for every 
$g\in G$. 

\begin{definition}
A \emph{parametrix} of $A\in L^0(\rpbar;\frakg)$ is any $B\in L^0(\rpbar;\frakg^{-1})$ satisfying 
 $$1-AB \in L^{-\infty}(\rpbar;\frakg\frakg^{-1}),\qquad 
     1-BA \in L^{-\infty}(\rpbar;\frakg^{-1}\frakg).$$
\end{definition}

We assume the existence of a principal symbol map, i.e., with every 
$A\in L^0(\rpbar;\frakg)$ with arbitrary weight $\frakg=(g_0,g_1)$ is associated a tuple 
 $$\sigma(A)=(\sigma_1(A),\ldots,\sigma_N(A))\qquad \text{($N$ some fixed integer)}$$
consisting of bundle homomorphisms 
$\sigma_k(A):E_k(g_0)\lra E_k(g_1)$, 
where $E_k(g)$ represents a Hilbert space bundle (finite or infinite-dimensional) associated with $g\in G$.   
We require that the principal symbols vanish for operators of order $-\infty$, that $\sigma_k(1)=1$ for every k, and    
\begin{itemize}
 \item[$(4)$] each principal symbol $\sigma_k$ is compatible with the addition, composition, and adjoint, i.e. 
  $\sigma_k(A+B)=\sigma_k(A)+\sigma_k(B)$ whenever $A+B$ is defined, 
  $\sigma_k(AB)=\sigma_k(A)\sigma_k(B)$ whenever $AB$ is defined, and 
  $\sigma_k(A^*)=\sigma_k(A)^*$ for every $A$.
 \item[$(5)$] $A\in L^0(\rpbar;\frakg)$ possesses a parametrix if and only if $\sigma_k(A)$ is a 
  bundle isomorphism for every $k=1,\ldots,N$.
\end{itemize}

\subsubsection{Ellipticity and parametrices} \label{sec:6.1.1}

We now consider \emph{Toeplitz type operators} associated with two projections. 
Given $\frakg=(g_0,g_1)$ and two projections $\Pi_j\in L^0(\rpbar;\frakg_j)$, $\frakg=(g_j,g_j)$, $j=0,1$, 
let us set, for $d\in\{0,-\infty\}$, 
\begin{align*}
 L^d(\rpbar;\frakg,\Pi_0,\Pi_1)
 =\big\{A\in L^d(\rpbar;\frakg)\mid     A(1-\Pi_0)=(1-\Pi_1)A=0\big\}
\end{align*}

\begin{definition}\label{def:param-toeplitz}
Given $A\in L^0(\rpbar;\frakg,\Pi_0,\Pi_1)$, any $B\in L^0(\rpbar;\frakg^{-1},\Pi_1,\Pi_0)$ 
satisfying 
\begin{align*}
 \Pi_1-AB &\in L^{-\infty}(\rpbar;\frakg\frakg^{-1},\Pi_1,\Pi_1),\\
 \Pi_0-BA &\in L^{-\infty}(\rpbar;\frakg^{-1}\frakg,\Pi_0,\Pi_0)
\end{align*}
is called a \emph{parametrix} of $A$. 
\end{definition}
    
The following theorem shows how ellipticity ``descents'' to the class of Toeplitz type operators: 

\begin{theorem}\label{thm:param-toeplitz}
For $A\in L^0(\rpbar;\frakg,\Pi_0,\Pi_1)$, $\frakg=(g_0,g_1)$, the following are equivalent: 
\begin{itemize}
 \item[a$)$] $A$ possesses a parametrix in the sense of Definition $\ref{def:param-toeplitz}$. 
 \item[b$)$] $A$ is \emph{elliptic}, i.e., for every $k=1,\ldots,N$ the following map is bijective$:$ 
   $$\sigma_k(A):\sigma_k(\Pi_0)\big(E_k(g_0)\big)\lra \sigma_k(\Pi_1)\big(E_k(g_1)\big).$$
\end{itemize}
\end{theorem}

Since, by assumption, $A=\Pi_1A\Pi_0$ and every principal symbol is multiplicative under
composition, $\sigma_k(A)=\sigma_k(\Pi_1)\sigma_k(A)\sigma_k(\Pi_0)$, i.e., 
the requirement in b$)$ is (only) the bijectivity. 
    
\begin{proof}   
It is clear that a$)$ implies b$)$. So let us start out from b$)$. 

\textbf{Step 1:} Let $a\in\scrL(H_0,H_1)$ and let $\pi_j\in\scrL(H_j)$ be projections such that 
$a(1-\pi_0)=(1-\pi_1)a=0$ and $a:\mathrm{im}\,\pi_0\to\mathrm{im}\,\pi_1$ bijectively. 
Define $a_0=a^*a+(1-\pi_0)^*)(1-\pi_0)$. Then $a_0\in\scrL(H_0)$ is invertible. 
In fact, define $t\in\scrL(H_0,H_1\oplus H_0)$ by $tx=(ax,(1-\pi_0)x)$. Then $t$ is injective and, 
since $t(\pi_0x+(1-\pi_0)y)=(a\pi_0x,(1-\pi_0)y)$ for all $x,y\in H_0$, 
$\mathrm{im}\,t=\mathrm{im}\,\pi_1\oplus\mathrm{im}\,(1-\pi_0)$ is closed in $H_1\oplus H_0$. 
Hence $t$ is upper semi-fredholm and therefore $a_0=t^*t$ is a self-adjoint, injective Fredholm operator, i.e., 
an isomorphism. 

\textbf{Step 2:} Let $\frakg_0=\frakg^{-1}\frakg$; then
 $$A_0(\mu):=A(\mu)^*A(\mu)+(1-\Pi_0)(\mu)^*(1-\Pi_0)(\mu)\in L^0(\rpbar;\frakg_0)$$
and 
 $$\sigma_k(A_0)=\sigma_k(A)^*\sigma_k(A)+\sigma_k(1-\Pi_0)^*\sigma_k(1-\Pi_0).$$ 
Applying the first step in each fibre of $E_k(g_0)$ shows that $\sigma_k(A_0):E_k(g_0)\to E_k(g_0)$ 
is an isomorphism. Hence $A_0$ admits a parametrix 
$B_0\in L^0(\rpbar;\frakg_0)$. Then $B_L:=\Pi_0B_0A^*\Pi_1$ belongs to 
$L^0(\rpbar;\frakg^{-1},\Pi_1,\Pi_0)$ and 
\begin{align*}
 B_L(\mu)A(\mu)=\Pi_0(\mu)B_0(\mu)A(\mu)^*A(\mu)\Pi_0(\mu)
 =\Pi_0(\mu)B_0(\mu)A_0(\mu)\Pi_0(\mu)\equiv \Pi_0(\mu)
\end{align*}
modulo $L^{-\infty}(\rpbar;\frakg_0,\Pi_0,\Pi_0)$. Hence $B_L$ is a parametrix from the left. 
    
\textbf{Step 3:} Let the notation be as in Step $1$. Then $a^*:\mathrm{im}\,\pi_1^*\to \mathrm{im}\,\pi_0^*$ is 
bijective. In fact, consider $\mathrm{im}\,\pi_j$ as Hilbert spaces (with the inner product induced from $H_j)$ and 
denote by $\wh a$ the map $a:\mathrm{im}\,\pi_0\to\mathrm{im}\,\pi_1$, by $\wh\pi_1$ the map 
$\pi_1:H_1\to\mathrm{im}\,\pi_1$, and by $\iota_0$ the inclusion map $\mathrm{im}\,\pi_0\to H_0$. 
Let $b=\iota_0\wh{a}^{-1}\wh\pi_1:H_1\to H_0$. Then 
$ba$ vanishes on $\mathrm{ker}\,\pi_0$ and on $\mathrm{im}\,\pi_0$ it coincides with 
$\iota_0\wh{a}^{-1}\wh a=\iota_0$. Thus $ba=\pi_0$. Similarly, $ab=\pi_1$. 
Hence $a^*b^*=\pi_0^*$ and $b^*a^*=\pi_1^*$. 
It follows that $b^*:\mathrm{im}\,\pi_0^*\to \mathrm{im}\,\pi_1^*$ 
is the inverse of $a^*:\mathrm{im}\,\pi_1^*\to \mathrm{im}\,\pi_0^*$. 

\textbf{Step 4:} Applying Step 3 in each fibre to $\sigma_k(A)^*=\sigma_k(A^*)$ we see that $A^*$ 
is an elliptic element of $L^0(\rpbar;\frakg^{-1},\Pi_1^*,\Pi_0^*)$. By Step 2, $A^*$ has a paramterix from the 
left. Taking the adjoint yields $B_R\in L^0(\rpbar;\frakg,\Pi_0,\Pi_1)$ which is a parametrix for $A$ 
from the right, i.e. $AB_R\equiv \Pi_1$ modulo a remainder of order $-\infty$. But then 
$B_L\equiv B_L(AB_R)=(B_LA)B_R\equiv B_R$ shows that both $B_L$ and $B_R$ are (two-sided) parametrices for $A$. 
\end{proof}   

\subsubsection{Invertibility for large values of the parameter}\label{sec:6.1.2}

In order to obtain parametrices that are inverses (in the sense explained below) for large values of $\mu$ 
we need to require:  
\begin{itemize}
 \item[$(6)$] Whenever $R\in L^{-\infty}(\rpbar;\frakg)$ with $\frakg=(g,g)$, there exists a 
  $\mu_0\ge0$ and an $S\in L^{-\infty}(\rpbar;\frakg)$ such that, for all $\mu\ge\mu_0$, 
   $$(1-R(\mu))(1-S(\mu))=(1-S(\mu))(1-R(\mu))=1.$$
\end{itemize}

\begin{lemma}
Let $R\in L^{-\infty}(\rpbar;\frakg,\Pi,\Pi)$, $\frakg=(g,g)$. Then there exists an 
$S\in L^{-\infty}(\rpbar;\frakg,\Pi,\Pi)$ such that 
   $$(\Pi(\mu)-R(\mu))(\Pi(\mu)-S(\mu))=(\Pi(\mu)-S(\mu))(\Pi(\mu)-R(\mu))=\Pi(\mu)$$
for all sufficiently large $\mu$. 
\end{lemma}
\begin{proof}
First, by $(6)$, there exists an $\wt S\in L^{-\infty}(\rpbar;\frakg)$ such that 
$(1-R(\mu))(1-\wt S(\mu))=(1-\wt S(\mu))(1-R(\mu))=1$ for large $\mu$. Multiplying with $\Pi(\mu)$ 
from the left and the right, we see that $S:=\Pi\wt S\Pi$ is as desired. 
\end{proof}

As an immediate consequence of this lemma and Theorem \ref{thm:param-toeplitz} we get: 

\begin{theorem}\label{thm:inverse-toeplitz}
Assume $(6)$. Then for every elliptic $A\in L^0(\rpbar;\frakg,\Pi_0,\Pi_1)$
exists a parametrix $B\in L^0(\rpbar;\frakg^{-1},\Pi_1,\Pi_0)$ such that,  
for all sufficiently large $\mu$, 
\begin{align*}
 A(\mu)B(\mu)=\Pi_1(\mu),\qquad  B(\mu)A(\mu)=\Pi_0(\mu)
\end{align*}
\end{theorem}

Obviously, using the notation of Theorem \ref{thm:inverse-toeplitz}, the map 
 $$A(\mu):\Pi_0(\mu)\big(H(g_0)\big)\lra \Pi_1(\mu)\big(H(g_1)\big)$$
is an isomorphism whenever $\mu$ is sufficiently large. The inverse is induced by $B(\mu)$. 

\section{Applications to boundary value problems}\label{sec:7}

We will apply the above results on Toeplitz type operators to boundary value problems

\subsection{Reformulation of the ellipticity conditions}\label{sec:7.1}

In the abstract setting of Section \ref{sec:6}, ellipticity was required to be characterized by the 
$($sole$)$ invertibility of certain principal symbols which were assumed to be bundle homomorphisms in certain 
Hilbert space bundles. 
Note that condition $($E$3)$ is of such a kind while, at a first glance, conditions $($E$1)$ and $($E$2)$ are not, since 
they contain estimates on the inverse (which are necessary in view of the non-compactness of $\rz^{n}_+)$. 
However, there is a simple way to reformulate these conditions appropriately, based on the following observation. 
For $a\in\scrC^\infty_b(\Omega,\scrL(E,\wt E))$ the following are equivalent: 
\begin{itemize}
 \item[$(1)$] $a(x)$ is invertible for every $x\in\Omega$ and $\|a(x)^{-1}\|_{\scrL(\wt E,E)}\lesssim 1$ uniformly 
   in $x\in\Omega$. 
 \item[$(2)$] $a(x)$ is invertible for every $x\in\Omega$ and 
  $x\mapsto a(x)^{-1} \in \scrC^\infty_b(\Omega,\scrL(\wt E,E))$. 
 \item[$(3)$] The operator of multiplication $f\mapsto af$ is invertible in 
  $\scrL(L^2(\Omega,E),L^2(\Omega,\wt E))$.
\end{itemize}
Let us take this point of view to reformulate ellipticity condition $($E$1)$. We consider $\sigma^{(d)}_\psi(\bfp)$ 
as a bundle homomorphism 
\begin{align}\label{eq:E1-new}
 \sigma^{(d)}_\psi(\bfp):\sz^{n}_+\times L^2(\rz^n_+,\cz^L)\lra \sz^{n}_+\times L^2(\rz^n_+,\cz^L)
\end{align}
which acts in the fibre over $(\xi,\mu)$ as operator of multiplication by $\sigma^{(d)}_\psi(\bfp)(\cdot,\xi,\mu)$. 
Then condition $($E$1)$ is equivalent to 
\begin{itemize}
 \item[$($E$1')$] The bundle homomorphism \eqref{eq:E1-new} is an isomorphism. 
\end{itemize}
Since the involved bundle is trivial, this is of course the same as requiring 
\begin{itemize}
 \item[$($E$1')$] Whenever $|\xi',\mu|=1$, the following map is an isomorphism: 
   \begin{align*}
     \sigma^{(d)}_\psi(\bfp)(\cdot,\xi',\mu):L^2(\rz^n_+,\cz^L)\lra L^2(\rz^n_+,\cz^L). 
    \end{align*}
\end{itemize}
Analogously, the boundary symbol $\sigma^{(d,\nu)}_\partial(\bfp)$ will be identified with a bundle homomorphism 
\begin{align}\label{eq:E2-new}
\begin{split}
 \sigma^{(d,\nu)}_\partial(\bfp):\wh\sz^{n-1}_+\times &L^2(\rz^{n-1},H^s(\rz_+,\cz^L)\oplus\cz^{M_0})\\
  &\lra \wh\sz^{n-1}_+\times L^2(\rz^{n-1},H^{s-d}(\rz_+,\cz^L)\oplus\cz^{M_1})
\end{split}
\end{align}
which acts in the fibre over $(\xi',\mu)$ as operator of multiplication by 
$\sigma^{(d,\nu)}_\partial(\bfp)(\cdot,\xi',\mu)$. 
Then conditions $($E$2)$ and $($E$3)$ together are equivalent to 
\begin{itemize}
 \item[$($E$2')$] The bundle hommorphism \eqref{eq:E2-new} is an isomorphism for some $($and then for all$)$  
 $s>d_+-\frac{1}{2}$. 
 \item[$($E$3')$] For some $($and then for all$)$ $s>d_+-\frac{1}{2}$, the following operator is invertible: 
   \begin{center}
    $\displaystyle 
    \op(\sigma^{[d]}_\infty(\bfp)):
     \begin{matrix}L^2(\rz^{n-1},H^s(\rz_+,\cz^L))\\ \oplus\\ L^2(\rz^{n-1},\cz^{M_0})\end{matrix}
     \lra 
     \begin{matrix}L^2(\rz^{n-1},H^{s-d}(\rz_+,\cz^L))\\ \oplus\\ L^2(\rz^{n-1},\cz^{M_1})\end{matrix}.$ 
    \end{center}
\end{itemize}

$($For $($E$3')$ recall the comment given after Definition \ref{def:ellipticity-full}).  
Again due to the triviality of the involved bundles we can write equivalently 
\begin{itemize}
 \item[$($E$2')$] Whenever $|\xi',\mu|=1$, the following map is an isomorphism  for some $($and then for all$)$  
 $s>d_+-\frac{1}{2}$: 
   \begin{center}
    $\displaystyle\sigma^{(d,\nu)}_\partial(\bfp)(\cdot,\xi',\mu):
     \begin{matrix}L^2(\rz^{n-1},H^s(\rz_+,\cz^L))\\ \oplus\\ L^2(\rz^{n-1},\cz^{M_0})\end{matrix}
     \lra 
     \begin{matrix}L^2(\rz^{n-1},H^{s-d}(\rz_+,\cz^L))\\ \oplus\\ L^2(\rz^{n-1},\cz^{M_1})\end{matrix}$.
    \end{center}
\end{itemize}


Let be given weight-data $\frakg=((L,M_0),(L,M_1))$ and $\frakg_j=((L,M_j),(L,M_j))$ for $j=0,1$. 
Suppose 
 $$\pi_j(x',\xi',\mu)\in \bfB^{0,\nu;0}(\rz^{n-1}\times\rpbar;\frakg_j)$$
are two projections, i.e. $\pi_j\#\pi_j=\pi_j$, where $\nu\in\nz_0$. Since $\pi_j$ is a projection,  
$\sigma_\psi^{(0)}(\pi_j)$,  $\sigma_\partial^{(0,\nu)}(\pi_j)$, and $\sigma_\infty^{[0]}(\pi_j)$ are projections, too. 
Recall that we consider $\sigma_\psi^{(0)}(\pi_j)$ as a bundle morphism as in \eqref{eq:E1-new}, 
$\sigma_\partial^{(0,\nu)}(\pi_j)$ as a bundle morphism as in \eqref{eq:E2-new}, and 
$\sigma_\infty^{[0]}(\pi_j)\in S^0(\rz^{n-1};\scrL(H^s(\rz_+,\cz^{L_j})\oplus\cz^{M_j}))$ for $s>-1/2$. 
In particular, 
\begin{align*}
 \mathrm{im}\,\sigma_\psi^{(0)}(\pi_j) &\subset \sz^{n}_+\times L^2(\rz^n_+,\cz^{L}),\\
 \mathrm{im}^s\,\sigma_\partial^{(0,\nu)}(\pi_j)&\subset
   \whsz^{n-1}_+\times L^2(\rz^{n-1},H^{s}(\rz_+,\cz^{L})\oplus\cz^{M_j}),\\
 \mathrm{im}^s\,\op(\sigma_\infty^{[0]}(\pi_j))&\subset
    L^2(\rz^{n-1},\scrL(H^{s}(\rz_+,\cz^{L})\oplus\cz^{M_1}); 
\end{align*}
here we added the superscript $s$ in $\mathrm{im}^s$ to indicate the involved Sobolev-regularity.   

\begin{definition}
With the above notation and $d\in\gz$, $\nu\in\nz_0$, let us define 
\begin{align*}
 \bfB^{d,\nu;d_+}&(\rz^{n-1}\times\rpbar;\frakg,\pi_0,\pi_1)\\
 &=\{\bfp\in \bfB^{d,\nu;d_+}(\rz^{n-1}\times\rpbar;\frakg)\mid (1-\pi_1)\#\bfp=\bfp\#(1-\pi_0)=0\}.  
\end{align*}
\end{definition}

\forget{
Since $\pi_1\#\bfp\#\pi_0=\bfp$ for $\bfp\in \bfB^{d,\nu;d_+}(\rz^{n-1}\times\rpbar;\frakg,\pi_0,\pi_1)$, 
\begin{align*}
 \sigma_\psi^{(d)}(\bfp)&=\sigma_\psi^{(0)}(\pi_1)\sigma_\psi^{(d)}(\bfp)\sigma_\psi^{(0)}(\pi_0),\\
 \sigma_\partial^{(d,\nu)}(\bfp)&=\sigma_\partial^{(0,\nu)}(\pi_1)\sigma_\partial^{(d,\nu)}(\bfp)
    \sigma_\partial^{(0,\nu)}(\pi_0),\\
 \sigma_\infty^{[d]}(\bfp)&=\sigma_\infty^{[0]}(\pi_1)\#\sigma_\infty^{[d]}(\bfp)\#\sigma_\infty^{[0]}(\pi_0). 
\end{align*}
}

\begin{theorem}\label{thm:parametrix-projection}
Let $\bfp\in\bfB^{d,\nu;d_+}(\rz^{n-1}\times\rpbar;\frakg,\pi_0,\pi_1)$ with $d\in\gz$, $\nu\in\nz_0$. 
Then there exists a symbol  
$\bfq\in \bfB^{-d,\nu;(-d)_+}(\rz^{n-1}\times\rpbar;\frakg^{-1},\pi_1,\pi_0)$ such that 
\begin{align*} 
 \op(\bfq)(\mu)\op(\bfp)(\mu)=\op(\pi_0)(\mu),\qquad 
 \op(\bfp)(\mu)\op(\bfq)(\mu)=\op(\pi_1)(\mu), 
\end{align*}
for sufficiently large $\mu$ if and only if the following conditions hold true$:$
\begin{itemize}
 \item[$(\Pi1)$] $\sigma_\psi^{(d)}(\bfp)$ of \eqref{eq:E1-new} restricts to an isomorphism 
  $\mathrm{im}\,\sigma_\psi^{(0)}(\pi_0)\to\mathrm{im}\,\sigma_\psi^{(0)}(\pi_1)$,
 \item[$(\Pi2)$] $\sigma_\partial^{(d,\nu)}(\bfp)$ of \eqref{eq:E2-new} restricts to an isomorphism 
  $\mathrm{im}^s\,\sigma_\partial^{(0,\nu)}(\pi_0)\to\mathrm{im}^{s-d}\,\sigma_\partial^{(0,\nu)}(\pi_1)$,
 \item[$(\Pi3)$] $\op(\sigma_\infty^{[d]}(\bfp))
 $ restricts to an isomorphism 
  $\mathrm{im}^s\,\sigma_\infty^{[0]}(\pi_0)\to\mathrm{im}^{s-d}\,\sigma_\infty^{([0]}(\pi_1)$
\end{itemize}
for some $s>d_+-\frac{1}{2}$. 
\end{theorem}
\begin{proof}
The case $d=0$ is an immediate consequence of the results from Section \ref{sec:6.1.1}. 
The general case can be reduced to it by the use of the order reductions from \eqref{eq:order-reduction}. 
Let us show this in case $d\ge1$ (the case $d\le-1$ is verified analogously$)$. 

Consider 
$\wt\bfp=\bfp\#\bflambda^{-d}_-$. Then 
 $$\wt\bfp=\pi_1\#\bfp\#\bflambda^{-d}_-\#\bflambda^{d}_-\#\pi_0\#\bflambda^{-d}_-
    =\pi_1\#\wt\bfp\#\wt\pi_0$$
with the projection 
 $$\wt\pi_0=\bflambda^{d}_-\#\pi_0\#\bflambda^{-d}_-\in
    \bfB^{0,\nu;0}(\rz^{n-1}\times\rpbar;\frakg_0).$$
This shows that $\wt\bfp\in \bfB^{0,\nu;0}(\rz^{n-1}\times\rpbar;\frakg,\wt\pi_0,\pi_1)$. 
Since $\op(\bflambda^{d}_-)$ is invertible with inverse $\op(\bflambda^{-d}_-)$ and by the multiplicativity 
under composition of all principal symbols, all principal symbols associated with $\bflambda^{\pm d}_-$ are 
isomorphisms and it is immediate by construction that $\bfp$ satisfies $(\Pi1)-(\Pi3)$ if and only if $\wt\bfp$ does 
(with $d=0$ and with respect to the projetcions $\wt\pi_0$ and $\pi_1$). Therefore we can apply the theorem 
with $d=0$ to $\wt\bfp$, yielding a corresponding symbol $\wt\bfq$. Then the claim follows with 
$\bfq:=\bflambda^{-d}_-\#\wt\bfq$.  
\end{proof}

\subsection{BVP with global projection conditions}\label{sec:7.2}

Let $A$ be a differential operator of order $d$ and let 
 $$A(\mu)=e^{i\theta}\mu^d-A=e^{i\theta}\mu^d-\sum_{|\alpha|\le d}a_{\alpha}(x)D^\alpha_x,\qquad \mu\ge0,$$
for some fixed $0\le\theta<2\pi$ and where $a_\alpha\in\scrC^\infty_b(\rz^n,\cz^{L\times L})$. 
We will apply Theorem \ref{thm:parametrix-projection} to analyze the invertibility of the problem 
\begin{equation}\label{eq:BVP-pi}
 \begin{pmatrix}A(\mu)\\ T\end{pmatrix}: 
    H^{s}(\rz^n_+,\cz^L)\lra
    \begin{matrix}H^{s-d}(\rz^n_+,\cz^L)\\ \oplus\\
    \Pi\Big( \mathop{\oplus}\limits_{j=0}^{d-1} H^{s-j-\frac{1}{2}}(\rz^{n-1},\cz^{m_j})\Big)
    \end{matrix},\qquad s>d-\frac12,
\end{equation}
where $T=\Pi S$ is a so-called \emph{global projection boundary condition} as described in the following 
Section \ref{sec:7.2.1}. 

\forget{
\begin{remark}
Our theory permits to consider in $\eqref{eq:BVP-pi}$ any general element $A(\mu)=\op(\bfp)(\mu)$ with 
$\bfp\in \bfB^{d,0;d}(\rz^{n-1}\times\rpbar;(L,0),(L,0))$. We focus here on the above case because it is important 
in applications and the expressions for the principal symbols of $A(\mu)$ are rather explicit. 
\end{remark}
}
\subsubsection{Global projection boundary conditions}\label{sec:7.2.1}

For $j\in\nz_0$ let us define 
 $$\gamma_j:\scrS(\rz_+,\cz^L)\lra \cz^L,\qquad 
     u\mapsto \frac{\partial^ju}{\partial x_n^j}(0)$$
$($we shall use the same notation for different values of $L)$. 
Then $\gamma_j$ is a strongly parameter-dependent trace symbol $($constant in $(x',\xi',\mu))$ of order 
$j+\frac{1}{2}$ and of type $j+1$. In fact, the symbol-kernels 
 $$g_0(\xi',\mu;x_n)=[\xi',\mu]\exp(-[\xi',\mu]y_n),\qquad g_1(\xi',\mu;y_n)=\exp(-[\xi',\mu]y_n)$$
define singular Green symbols $g_k\in B^{1/2-k;0}_G(\rz^{n-1}\times\rpbar;(L,0),(0,L))$ and, by integration by parts,  
 $$\gamma_0=g_0(\xi',\mu)-g_1(\xi',\mu)\partial_+.$$
For $j\ge1$ note that $\gamma_j u=\gamma_0(\partial^j_+u)$. Note that the principal symbol and the 
principal limit symbol are 
 $$\sigma^{(j+1/2)}(\gamma_j)=\sigma^{[j+1/2]}_\infty(\gamma_j)=\gamma_j.$$
Obviously, 
 $$[\op(\gamma_j)u](x')=\frac{\partial^ju}{\partial x_n^j}(x',0),\qquad u\in\scrS(\rz^n_+,\cz^L);$$
By abuse of notation (and to obtain standard notation from the literature) we shall write $\gamma_j$ 
instead of $\op(\gamma_j)$. 
If $s_{jk}(x',\xi')\in S^{j-k}(\rz^{n-1};\scrL(\cz^L,\cz^{m_j}))$, $j,k\in\nz_0$, with $m_j\in\nz_0$  
are pseudodifferential symbols then 
\begin{equation}\label{eq:bc}
  s_j(x',\xi')=\sum_{k=0}^{j} s_{jk}(x',\xi')\gamma_k:\scrS(\rz_+,\cz^L)\lra\cz^{m_j} 
\end{equation}
defines a trace symbol of order $j+1/2$, type $j+1$, and regularity number $0$, i.e., 
 $$s_j\in \bfB^{j+1/2,0;j+1}_G(\rz^{n-1}\times\rpbar;(L,0),(0,m_j)),\qquad j\in\nz_0$$
$($note that if in \eqref{eq:bc} the summation index $k$ would exceed $j$ then $s_j$ would have negative 
regularity number$)$. Writing $S_{jk}:=\op(s_{jk})$ and $S_j=\op(s_j)$ we thus have 
$$\begin{pmatrix}S_0\\\vdots\\ \vdots\\ S_{d-1}\end{pmatrix}=
    \begin{pmatrix}
     S_{00}&0&\cdots&0\\
     \vdots&\ddots&\ddots &\vdots\\
     \vdots& & \ddots &0\\
     S_{d-1,0}&\cdots&\cdots&S_{d-1,d-1}
    \end{pmatrix}
    \begin{pmatrix}\gamma_0\\\vdots\\ \vdots\\ \gamma_{d-1}\end{pmatrix}.
$$

\begin{definition}\label{def:global-projection-condition}
Using the previous notation and writing $S=(S_0,\ldots,S_{d-1})^t$, a 
\emph{global projection boundary} condition has the form $T=\Pi S$ with an idempotent  
$\Pi=(\Pi_{jk})_{0\le j,k\le d-1}$, where $\Pi_{jk}=0$ whenever $k>j$ and, otherwise, 
$\Pi_{jk}=\op(\pi_{jk})$ with $\pi_{jk}\in S^{j-k}(\rz^{n-1};\scrL(\cz^{\ell_k},\cz^{m_j}))$. 
\end{definition}

The assumption that $\Pi$ is a left lower triangular matrix ensures that the components of 
$\Pi$ have non-negative regularity number. More explicitly, $T$ has the form 
$T=(T_0,\ldots,T_{d-1})^t$ with  
 $$T_j=\sum_{k=0}^{j} \Pi_{jk}S_k=\sum_{k=0}^{j}\sum_{\ell=0}^{k} \op(\pi_{jk})\op(s_{k\ell})\gamma_\ell;$$
note that $T_j \in \bfB^{j+1/2,0;j+1}_G(\rz^{n-1}\times\rpbar;(L,0),(0,m_j))$. $T$ has the mapping property 
from  \eqref{eq:BVP-pi}.  
Let us remark that we admit also the case that some of the $m_j$ are equal to zero; in this case the condition 
$T_j$ is considered ``void'' respectively ``not present''. 

\begin{remark}
Let us compare the structure of the operator in $\eqref{eq:BVP-pi}$ with that considered by Grubb in \cite{Grub}. 
On the one hand the condition $T$ in $\eqref{eq:BVP-pi}$ is more general, since  
the projection $\Pi$ and all terms $S_{jk}$ can be pseudodifferential operators, while in \cite{Grub} one has $\Pi=1$ 
and all $S_{jj}$ are operators of multiplication with a function $s_{jj}(x')$, cf. \cite[(1.5.28)]{Grub}. On the other hand,  
each condition $T_k$ in \cite[(1.5.28)]{Grub} contains a trace operator $T_k'$ which we have to require to be either equal to zero 
or to be parameter-dependent. Similarly, $A(\mu)$ in \cite[(1.5.19)]{Grub} is permitted to contain a parameter-independent 
singular Green operator $G$ which we have to require either to be equal to zero or to be parameter-dependent. 
\end{remark}

\subsubsection{Invertibility of realizations}\label{sec:7.2.2}

In order to apply Theorem \ref{thm:parametrix-projection} we need to transform the problem in an equivalent form. 
To this end let 
 $$\Lambda(\mu)=\mathrm{diag}(\Lambda_0(\mu),\ldots,\Lambda_{d-1}(\mu)),\qquad 
     \Lambda_j(\mu)=[D',\mu]^{d-j-\frac{1}{2}}$$
and then consider 
\begin{align*}
 \begin{pmatrix}A(\mu)\\ \Lambda(\mu)T\end{pmatrix}
 =\begin{pmatrix}A(\mu)\\ \Lambda(\mu)\Pi S\end{pmatrix}
     =\begin{pmatrix}1&0\\ 0&\Pi_\Lambda(\mu)\end{pmatrix}
       \begin{pmatrix}A(\mu)\\ \Lambda(\mu)\Pi S\end{pmatrix},
\end{align*}
with
 $$\Pi_\Lambda(\mu)=\Lambda(\mu)\Pi\Lambda(\mu)^{-1}=\op(\pi_\Lambda)(\mu),
    \qquad \pi_\Lambda\in S^{0,0}(\rz^{n-1}\times\rpbar;\scrL(\cz^{M})),$$ 
where we use the identification 
 $$\cz^M=\cz^{m_0}\oplus\ldots\oplus\cz^{m_{d-1}},\qquad M=m_0+\ldots+m_{d-1}.$$
Therefore
\begin{equation}\label{eq:APS-symbol}
 \begin{pmatrix}A(\mu)\\ \Lambda(\mu)T\end{pmatrix}=\op(\bfp)(\mu),\qquad 
    \bfp\in \bfB^{d,0;d}(\rz^{n-1}\times\rpbar;(L,0),(L,M),1,\pi_\Lambda).
\end{equation}
By construction of $\bfp$, its homogeneous principal symbol is 
\begin{align*}
 \sigma_\psi^{(d)}(\bfp)(x,\xi,\mu)=e^{i\theta}\mu^d-\sum_{|\alpha|=d}a_\alpha(x)\xi^\alpha,
\end{align*}
its principal boundary symbol is 
\begin{align*}
 \sigma_\partial^{(d,0)}(\bfp)(x',\xi',\mu)=
 \begin{pmatrix}
   \sigma_\partial^{(d,0)}(A)(x',\xi',\mu)\\
   \sigma_\partial^{(d,0)}(\Lambda T)(x',\xi',\mu)
 \end{pmatrix}
 \end{align*}
where 
\begin{align*}
 \sigma_\partial^{(d,0)}(A)(x',\xi',\mu)=
  e^{i\theta}\mu^d-\sum\limits_{|\alpha|=d}a_\alpha(x',0)\xi'^{\alpha'}D_{x_n}^{\alpha_n}
\end{align*}
and
\begin{align*}
 \sigma_\partial^{(d,0)}(\Lambda T)(x',\xi',\mu)
 =\begin{pmatrix}
  |\xi',\mu|^{d-\frac12}\pi_{00}^{(0)}(x',\xi')s_{0}^{(0)}(x',\xi')\gamma_0\\
  \vdots\\
  |\xi',\mu|^{d-j-\frac12}\sum\limits_{k=0}^{j}
    \sum\limits_{\ell=0}^{k}\pi_{jk}^{(j-k)}(x',\xi')s_{k\ell}^{(k-\ell)}(x',\xi')\gamma_\ell\\
  \vdots\\
  |\xi',\mu|^{1-\frac12}\sum\limits_{k=0}^{d-1}
    \sum\limits_{\ell=0}^{k}\pi_{jk}^{(j-k)}(x',\xi')s_{k\ell}^{(k-\ell)}(x',\xi')\gamma_\ell
  \end{pmatrix}.
\end{align*}
If we introduce the ``mixed'' or ``Douglis-Nirenberg-type'' homogeneous principal symbols 
\begin{align*}
 \sigma_\Hom(\Pi)(x',\xi')&=
  \begin{pmatrix}
  \pi_{00}^{(0)} & & \\
  \vdots & \ddots & \\
  \pi_{d-1,0}^{(d-1)} & \cdots &\pi_{d-1,d-1}^{(0)}
 \end{pmatrix}(x',\xi'),\\
 \sigma_\Hom(S)(x',\xi')&=
 \begin{pmatrix}
  s_{00}^{(0)} & & \\
  \vdots & \ddots & \\
  s_{d-1,0}^{(d-1)} & \cdots & s_{d-1,d-1}^{(0)}
 \end{pmatrix}(x',\xi')
\end{align*}
and note that the factors $|\xi',\mu|^{d-j-\frac12}$ are equal to $1$ whenever $|\xi',\mu|=1$, we have 
\begin{align*}
 \sigma_\partial^{(d,0)}(\Lambda T)(x',\xi',\mu)
 =\sigma_\Hom(\Pi)(x',\xi')\sigma_\Hom(S)(x',\xi') 
    \begin{pmatrix}\gamma_0\\\vdots\\ \gamma_{d-1}\end{pmatrix}, \qquad |\xi',\mu|=1.
\end{align*}
Finally, the principal limit symbol of $\bfp$ is
\begin{align*}
 \sigma_\infty^{[d]}(\bfp)(x',\xi')
 =\begin{pmatrix}
  e^{i\theta}-a_{0,d}(x',0)D_{x_n}^d\\
  (\pi_{00}\#s_{00})(x',\xi')\gamma_0\\
  \vdots\\
  (\pi_{d-1,d-1}\#s_{d-1,d-1})(x',\xi')\gamma_{d-1}
  \end{pmatrix}.
\end{align*}
For the latter identity note that $\Lambda_j(\mu)\op(\pi_{jk})\op(s_{k\ell})\gamma_\ell$ is a trace operator with 
symbol belonging to $\wtbfB^{d,j-\ell;\ell+1}_G(\rz^{n-1}\times\rpbar;(L,0),(0,m_j))$, 
hence its principal limit symbol $\sigma_\infty^{[d]}$ is zero unless $\ell=j$; 
hence the principal limit-symbol of $\Lambda_j(\mu)T_j$ coincides with that of 
$\Lambda_j(\mu)\op(\pi_{jj})\op(s_{jj})\gamma_j=\Lambda_j(\mu)\op(\pi_{jj}\#s_{jj})\gamma_j$ which is the one 
appearing above, due the multiplicativity of the principal limit symbol under composition. Moreover recall that 
the principal limit symbol of $A(\mu)$ coincides with its boundary symbol evaluated in $(\xi',\mu)=(0,1)$. 

\begin{theorem}\label{thm:main-projection}
Let $\bfp$ from \eqref{eq:APS-symbol} satisfy $(\Pi1)-(\Pi3)$ of Theorem $\ref{thm:parametrix-projection}$ 
$($with $\pi_0=1$ and $\pi_1=(\pi_{jk}))$. Then there exist a pseudodifferential operator 
 $$B_+(\mu)=\op^+(b)(\mu),\qquad b\in S^{-d}_\tr(\rz^n\times\rpbar;\scrL(\cz^L)),$$
a singular Green operator 
 $$G(\mu)=\op(g)(\mu),\qquad g\in \wt\bfB^{-d,0;0}_G(\rz^{n-1}\times\rpbar;(L,0),(L,0)),$$
and trace operators 
 $$K_j(\mu)=\op(k_j)(\mu),\qquad k_j\in \wt\bfB^{-j-\frac{1}{2},0;0}_G(\rz^{n-1}\times\rpbar;(0,m_j),(L,0)),$$
such that, with $K(\mu)=(K_0(\mu),\ldots,K_{d-1}(\mu))$, 
\begin{align*}
  \begin{pmatrix}e^{i\theta}\mu^d-A\\ T\end{pmatrix} \begin{pmatrix}B_+(\mu)+G(\mu)&K(\mu)\end{pmatrix}
     &= \begin{pmatrix}1&0\\ 0&\Pi\end{pmatrix},\\
     \begin{pmatrix}B_+(\mu)+G(\mu)&K(\mu)\end{pmatrix}
     \begin{pmatrix}e^{i\theta}\mu^d-A\\ T\end{pmatrix}&=1
\end{align*}
for sufficiently large $\mu$. In particular, the map $\eqref{eq:BVP-pi}$ is bijective for large $\mu$. 
\end{theorem} 
\begin{proof}
By Theorem \ref{thm:parametrix-projection} there exists a 
$\bfq\in \bfB^{-d,0;0}(\rz^{n-1}\times\rpbar;(L,M),(L,0),\pi_\Lambda,1)$ such that, for large enough $\mu$, 
 $$1=\op(\bfq)(\mu)\op(\bfp)(\mu)= \op(\bfq)(\mu)\begin{pmatrix}1&0\\ 0&\Lambda(\mu)\end{pmatrix}
       \begin{pmatrix}A(\mu)\\ T\end{pmatrix}$$
and 
\begin{align*}
 \begin{pmatrix}1&0\\ 0&\Pi_\Lambda(\mu)\end{pmatrix}
   &=\begin{pmatrix}1&0\\ 0&\Lambda(\mu)\end{pmatrix}\begin{pmatrix}1&0\\ 0&\Pi\end{pmatrix}
        \begin{pmatrix}1&0\\ 0&\Lambda(\mu)^{-1}\end{pmatrix}
     =\op(\bfp)(\mu)\op(\bfq)(\mu)\\
   &=\begin{pmatrix}1&0\\ 0&\Lambda(\mu)\end{pmatrix}
       \begin{pmatrix}A(\mu)\\ T\end{pmatrix}\op(\bfq)(\mu).
\end{align*}
The latter is equivalent to 
\begin{align*}
 \begin{pmatrix}1&0\\ 0&\Pi\end{pmatrix}
   =\begin{pmatrix}A(\mu)\\ T\end{pmatrix}\op(\bfq)(\mu)
      \begin{pmatrix}1&0\\ 0&\Lambda(\mu)\end{pmatrix}. 
\end{align*}
It remains to observe that 
$\op(\bfq)(\mu)\begin{pmatrix}1&0\\ 0&\Lambda(\mu)\end{pmatrix}
=\begin{pmatrix}B_+(\mu)+G(\mu)&K(\mu)\end{pmatrix}$ as stated. 
\end{proof}

Let us reformulate the assumption of the previous theroem, i.e., that $\bfp$ satisfies $(\Pi1)-(\Pi3)$, 
in a more explicit way: 
\begin{itemize}
 \item[$(\Pi1)$] The homogeneous principal symbol of $A(\mu)$ is invertible whenever $|\xi',\mu|=1$ and 
   $x\in\overline{\rz^n}_+$ with uniform estimate 
     $$\Big\|\Big(e^{i\theta}\mu^d-\sum_{|\alpha|=d}a_\alpha(x)\xi^\alpha\Big)^{-1}\Big\|_{\scrL(\cz^L)}
         \lesssim 1.$$
 \item[$(\Pi2)$] For some $s>d-\frac12$ the maps  
 \begin{align*}
 \sigma^{(d,0)}_\partial(\bfp)(\cdot,\xi',\mu):L^2(\rz^{n-1},H^s(\rz_+,\cz^L))\lra 
  \begin{matrix}
   L^2(\rz^{n-1},H^{s-d}(\rz_+,\cz^L)\\ 
   \oplus\\ 
   \sigma_\Hom(\Pi)(\cdot,\xi')\big(L^2(\rz^{n-1},\cz^M)\big)
   \end{matrix}.
\end{align*}
   are isomorphisms for all $(\xi',\mu)$ with $|\xi',\mu|=1$, $\xi'\not=0$  
  $(a(\cdot)$ means here the operator of multiplication with the function $a)$.
 \item[$(\Pi3)$] For some $s>d-\frac12$ the following map is an isomorphism: 
 \begin{align*}
 \begin{pmatrix}
  e^{i\theta}-a_{0,d}(x',0)D_{x_n}^d\\
  \Pi_{00}S_{00}\gamma_0\\
  \vdots\\
  \Pi_{d-1,d-1}S_{d-1,d-1}\gamma_{d-1}
  \end{pmatrix}:
  L^2(\rz^{n-1},H^s(\rz_+,\cz^L))\lra 
 \begin{matrix}
 L^2(\rz^{n-1},H^{s-d}(\rz_+,\cz^L))\\\oplus \\
 \mathop{\mbox{$\oplus$}}\limits_{j=0}^{d-1}\Pi_{jj}(L^2(\rz^{n-1},\cz^{m_j}))
 \end{matrix}
\end{align*}
\end{itemize} 
In case we require that all involved principal symbols are constant in $x$ or $x'$ for large $|x|$ and $|x'|$, respectively, 
the conditions become: 
\begin{itemize}
 \item[$(\Pi1)$] $e^{i\theta}\mu^d-\sum\limits_{|\alpha|=d}a_\alpha(x)\xi^\alpha$ is invertible for all $(x',\xi',\mu)$ 
  with $|\xi',\mu|=1$. 
 \item[$(\Pi2)$] For some $s>d-\frac12$, 
  \begin{align*}
   \sigma^{(d,0)}_\partial(\bfp)(x',\xi',\mu)
  :H^s(\rz_+,\cz^L))\lra 
  \begin{matrix}
   H^{s-d}(\rz_+,\cz^L)\\ 
   \oplus\\ 
   \sigma_\Hom(\Pi)(x',\xi')(\cz^M)
   \end{matrix}.
  \end{align*}
 is an isomorphism for all $(x',\xi',\mu)$ with $|\xi',\mu|=1$, $\xi'\not=0$.
 \item[$(\Pi3)$] As above.  
\end{itemize}

\subsubsection{Resolvent and trace expansions}\label{sec:7.2.3}

With the above notation let $A_T$ denote the unbounded operator in $L^2(\rz^n_+,\cz^L)$ acting as 
$A$ on the domain 
 $$\mathcal{D}(A_T)=\{u\in H^d(\rz^n_+,\cz^L)\mid Tu=0\}.$$
By Theorem \ref{thm:main-projection}, $e^{i\theta}\mu^d-A_T$ is invertible for sufficiently large $\mu$ with inverse  
 $$(e^{i\theta}\mu^d-A_T)^{-1}=B_+(\mu)+G(\mu),\qquad \mu >> 1.$$ 

\begin{corollary}
With the above notation and assumptions, 
$\Gamma_\theta:=\{e^{i\theta}\mu\mid\mu\ge 0\}$ is a ray of minimal growth for $A_T$, i.e., $\lambda-A_T$ 
is invertible for sufficiently large $\lambda\in\Gamma_\theta$ with uniform estimate 
 $$\|(\lambda-A_T)^{-1}\|_{\scrL(L^2(\rz^n_+,\cz^L))}\lesssim \frac{1}{|\lambda|}.$$
\end{corollary}
\begin{proof}
Using notation from Theorem \ref{thm:main-projection} observe that 
 $$\op^+(b)(x',\xi',\mu)+g(x',\xi',\mu)\in \wt S^{-d,0}_{1,0}(\rz^{n-1}\times\rpbar;\scrL(L^2(\rz_+,\cz^L)),$$
since the dilation group-action is a group of unitary operators on $L^2(\rz_+,\cz^L)$. 
Moreover, $p(x',\xi,\mu)\in \wt S^{-d,0}_{1,0}(\rz^{n-1}\times\rpbar;\scrL(E))$ implies 
$p_\mu(x',\xi'):=\mu^d p(x',\xi,\mu)\in S^{0}_{1,0}(\rz^{n-1};\scrL(E))$ uniformly $\mu$, 
hence $\mu^d \op(p)(\mu)=\op(p_\mu)\in \scrL(L^2(\rz^{n-1},E))$ is uniformly bounded in $\mu$. 
\end{proof}

Given a trace class operator $K:L^2(\rz^n_+,\cz^L)\to L^2(\rz^n_+,\cz^L)$ which is also an integral operator with 
continuous kernel $k\in \scrC(\rz^n_+\times\rz^n_+,\cz^{L\times L})$, its trace coincides with 
 $$\mathrm{Tr}\,K= \int_{\rz^n_+}\mathrm{tr}\,k(x,x)\,dx
    =\int_{\rz^{n-1}}\int_0^\infty \mathrm{tr}\,k(x',x_n,x',x_n)\,dx_n dx'.$$
If $p(\mu)=\op^+(a)(\mu)+g(\mu) \in \bfB^{d,\nu;0}(\rz^{n-1}\times\rpbar;(L,0),(L,0))$ has order $d<-n$, then 
$\op(p)(\mu)\in\scrL(L^2(\rz^n_+,\cz^L))$ for every $\mu$ with continuous integral kernel 
 $$k(x,y,\mu)=\int e^{i(x-y)\xi}a(x,\xi,\mu)\,d\xi+ \int e^{i(x'-y')\xi}g(x',\xi',\mu;x_n,y_n)\,\dbar\xi'.$$
Thus, supposing in addition the trace class property, we get 
\begin{align*}
 \mathrm{Tr}\,\op(p)(\mu)
  =&\,\int_{\rz^n} \int_{\rz^n_+}\mathrm{tr}\,a(x,\xi,\mu)\,dx \dbar\xi +\\
     &+\int_{\rz^{n-1}}\int_{\rz^{n-1}} \int_0^{+\infty} \mathrm{tr}\,g(x',\xi',\mu;x_n,x_n)\,dx_n dx'\dbar\xi'.
\end{align*}
To have the trace class property it is not sufficient to assume negative order $d<-n$ but one needs 
to require decay of the coefficients in $x$, essentially like $\spk{x}^{-n-\eps}$ for some $\eps>0$.  
Note however, that the symbol-kernel of the singular Green symbol $g$ is rapidly decreasing in $x_n$ and $y_n$. 
Thus it makes sense to introduce the symbol $\mathrm{Tr}_+g$ as 
\begin{align*}
 (\mathrm{Tr}_+g)(x',\xi',\mu):=
 \int_0^{+\infty} \mathrm{tr}\,g(x',\xi',\mu;x_n,x_n)\,dx_n; 
\end{align*}
in fact, for each fixed $(x',\xi',\mu)$ this is the trace of the trace class operator in $L^2(\rz_+,\cz^L)$ with integral 
kernel $(x_n,y_n)\mapsto g(x',\xi',\mu;x_n,y_n)$. 

\begin{lemma}
If $g\in\wtbfB^{d,\nu;0}(\rz^{n-1}\times\rpbar;(L,0),(L,0))$ with arbitrary $d,\nu\in\rz$, then 
$\mathrm{Tr}_+g\in\wtbfS^{d,\nu}(\rz^{n-1}\times\rpbar)$.  
\end{lemma}
\begin{proof}
Recall that, by definition, $g$ has a symbol-kernel 
 $$g(x',\xi',\mu;x_n,x_n)=[\xi',\mu]g'(x',\xi',\mu;[\xi',\mu]x_n,[\xi',\mu]x_n)$$
where $g'\in \wtbfS^{d,\nu}(\rz^{n-1}\times\rpbar;\scrS(\rz_+\times\rz_+,\cz^{L\times L}))$. 
Inserting this in the defining expression of $\mathrm{Tr}_+g$ and making the change of variables $t:=[\xi',\mu]x_n$ 
immediately yields the claim. 
\end{proof}

As an immediate corollary of \cite[Theorem 6.3]{Seil22-1} (which is based on techniques for trace expansions 
developed by Grubb and Seeley, in particular in \cite[Theorem 2.1]{GrSe}) we get: 

\begin{theorem}
Let $g\in\wtbfB_G^{d,\nu;0}(\rz^{n-1}\times\rpbar;(L,0),(L,0))$ with $d<-n+1$ and $d-\nu\le0$. Then 
  $$\int \mathrm{Tr}_+g(x',\xi',\mu)\,\dbar\xi'\sim_{\mu\to+\infty}
     \sum_{j=0}^{+\infty}c_j(x')\mu^{d-j+n-1}+\sum_{j=0}^{+\infty}
     \big(c_j^\prime(x')\log\mu+c_j^{\prime\prime}(x')\big)\mu^{d-\nu-j}$$ 
with certain continuous and bounded functions $c_j,c_j^\prime,c_j^{\prime\prime}$. 
\end{theorem}

Concerning the pseudodifferential part, using the asymptotic expansion of $a$ in homogeneous components 
$a^{(d-j)}(x,\xi,\mu)\in S^{d-j}_\Hom$, it is easy to see that 
\begin{align*}
 \int_{\rz^n}\mathrm{tr}\,a(x,\xi,\mu)\,\dbar\xi \sim_{\mu\to+\infty}\sum_{j=0}^{+\infty} a_j(x)\mu^{d-j+n}
\end{align*}
provided $d<-n$ with coefficients $\displaystyle a_j(x)=\int a^{(d-j)}(x,\xi,1)\,\dbar\xi$. 

Let us now consider $P(\mu)=D(\mu^de^{i\theta}-A_T)^{-\ell}$, where $D$ is a differential operator of order $m$ 
and $\ell$ is so large that $m-d\ell<-n$. 
By Theorem \ref{thm:main-projection}, for large enough $\mu$, $P(\mu)=\op(p)(\mu)$ with a symbol 
 $$p\in \bfB^{m-d\ell,0;0}(\rz^{n-1}\times\rpbar;(L,0),(L,0)).$$
The above results imply at once the following:  

\begin{theorem}\label{thm:trace-expansion}
With the above notation and ellipticity assumptions, and supposing that the coefficients of $D$ are 
decaying in $x$ sufficiently fast, 
\begin{align*}
 \mathrm{Tr}\,D(\lambda e^{i\theta}-A_T)^{-\ell}\sim_{\lambda\to+\infty} 
 \sum_{j=-n}^{-1} \alpha_j \lambda^{\frac{m-j}{d}-\ell}
 +\sum_{j=0}^{+\infty}\big(\alpha_j^\prime\log\lambda+\alpha_j^{\prime\prime}\big)\lambda^{\frac{m-j}{d}-\ell}
\end{align*}
for certain coefficients $\alpha_j,\alpha_j^\prime$, and $\alpha_j^{\prime\prime}$ $($depending on $A$, $T$, and $D)$. 
\end{theorem}

If in the previous theorem the differential operator $D$ is tangential near the boundary, then $\alpha_0'=\ldots=\alpha_{m-1}'=0$ 
because in this case $p\in \bfB^{m-d\ell,m;0}(\rz^{n-1}\times\rpbar;(L,0),(L,0))$ has regularity number $m$. 

\subsection{Example: Laplace type operators}\label{sec:7.3}

The paper \cite{Grub03} investigates trace expansions for $D(\lambda-A_T)^{-\ell}$, $\ell>\frac{n+m}{2}$, 
where $A$ is a Laplace type operator on a compact manifold acting on sections of an $L$-dimensional 
vector-bundle, $D$ is a differential operator of order $m$, and $T$ is a boundary condition related to spectral boundary 
conditions. We shall consider the analogous problem in the half-space; for simplicity we shall assume that $A$ and $T$ have 
constant coefficients outside some compact set. 
We will explain that the assumptions made in \cite{Grub03} imply that  
$\begin{pmatrix}\mu^2e^{i\theta}-A\\ T\end{pmatrix}$ is an elliptic problem in the sense of Theorem \ref{thm:main-projection} 
for every $\theta\in(0,2\pi)$ and therefore the trace expansion follows from Theorem \ref{thm:trace-expansion} with  $d=2$. 

So let $A$ be a second order differential operator on $\rz^n_+$ with $\cz^{L\times L}$-valued coefficients such that, 
near the boundary, 
 $$ A=D_{x_n}^2+A'+x_nA_2+A_1,$$
where $A'$ is a second order self-adjoint, positive differential operator on the boundary and $A_j$ is a differential operator of 
order $j$, cf. \cite[Assumption 1.1]{Grub03}.  
Assume that $e^{-\theta}\mu^2-A$ for some $\theta\in(0,2\pi)$ satisfies ellipticity assumption $(\Pi1)$ $($for example, 
this is the case when $A$ is, as assumed in \cite{Grub03}, principally self-adjoint, i.e., $A^*-A$ has order one).
The boundary condition has the form 
 $$Tu=\begin{pmatrix}\Pi \gamma_0 u\\ (1-\Pi)(\gamma_1u+B\gamma_0u)\end{pmatrix},$$
where $B$ is a first order pseudo-differential operator on the boundary and $\Pi$ is zero order projection, 
cf. \cite[(1.4)]{Grub03}. Thus, in our notation, $d=2$, $S_{00}=S_{11}=1$, $S_{10}=B$, $\Pi_{00}=\Pi$, $\Pi_{11}=1-\Pi$, and 
$S_{01}=\Pi_{01}=\Pi_{10}=0$.
Let us first consider the principal limit symbol, i.e.,  
 $$ \begin{pmatrix}e^{i\theta}-D_{x_n}^2\\ \Pi \gamma_0 \\ (1-\Pi)\gamma_1\end{pmatrix}
     :L^2(\rz^{n-1},H^2(\rz_+,\cz^L))\lra 
     \begin{matrix}
      L^2(\rz^{n-1}, L^2(\rz_+,\cz^L))\\ \oplus\\ \Pi (L^2(\rz^{n-1},\cz^L))\\ \oplus\\ (1-\Pi) (L^2(\rz^{n-1},\cz^L))  
   \end{matrix}.$$
This operator is bijective whenever $\theta\not=0$. In fact, in this case,   
$e^{i\theta}-D_{x_n}^2$ is surjective with kernel 
 $$N=N(\theta)=\left\{u(x)=v(x')e^{-\rho x_n}\mid v\in L^2(\rz^{n-1},\cz^L)\right\},\qquad \rho=e^{i(\theta-\pi)/2},$$
Hence it suffices to verify that 
 $$ \begin{pmatrix}\Pi \gamma_0 \\ (1-\Pi)\gamma_1\end{pmatrix}
     :N\lra 
     \begin{matrix}
      \Pi (L^2(\rz^{n-1},\cz^L))\\ \oplus\\ (1-\Pi) (L^2(\rz^{n-1},\cz^L))  
   \end{matrix}$$
is bijective. $u\in N$ belongs to the kernel of this map if and only if 
$\Pi v=0$ and $-\rho (1-\Pi)v=0$, i.e., $v=0$. This shows the injectivity. Given $f,g\in L^2(\rz^{n-1},\cz^L)$ take $u\in N$ 
with $v:=\Pi f-\frac{1}{\rho}(1-\Pi)g$. Then $\Pi\gamma_0u=\Pi f$ and $(1-\Pi)\gamma_1u=(1-\Pi)g$. This shows the surjectivity. 
In other words, ellipticity condition $(\Pi3)$ is satisfied. 

Let us now look at the bijectivity of the principal boundary symbol 
 $$ \begin{pmatrix}
     e^{i\theta}\mu^2-{a}'^{(2)}(x',\xi')-D_{x_n}^2\\ \pi^{(0)}(x',\xi') \gamma_0 \\ 
     (1-\pi^{(0)}(x',\xi'))(\gamma_1+b^{(1)}(x',\xi')\gamma_0)
     \end{pmatrix}
     :H^2(\rz_+,\cz^L)\lra 
     \begin{matrix}
      L^2(\rz_+,\cz^L)\\ \oplus\\ \pi^{(0)}(x',\xi') (\cz^L)\\ \oplus\\ (1-\pi^{(0)}(x',\xi'))(\cz^L)  
   \end{matrix},$$
where $\xi'\not=0$ and  ${a}'^{(2)}$, $\pi^{(0)}$, and $b^{(1)}$ are the homogeneous principal symbols of $A'$, $\Pi$,  
and $B$, respectively. Let $\sigma=\sigma(x',\xi',\mu)$ denote the square root of ${a}'^{(2)}(x',\xi')-e^{i\theta}\mu^2$, 
i.e., $\sigma=\displaystyle\int_\Gamma \lambda^{1/2}
(\lambda-{a}'^{(2)}+e^{i\theta}\mu^2)^{-1}\,d\lambda$, where $\lambda^{1/2}$ is the principal branch of the square 
root on $\cz\setminus(-\infty,0]$ and $\Gamma$ is a Jordan curve around the poles of the integrand. 
The operator in the first row is surjective with kernel 
 $$N=N(x',\xi',\mu)=\left\{\phi(x_n)=e^{-\sigma x_n}z\mid z\in \cz^L\right\}.$$
Thus we have to verify the bijectivity of 
 $$ \begin{pmatrix}
      \pi^{(0)}(x',\xi') \gamma_0 \\(1-\pi^{(0)}(x',\xi'))(\gamma_1+b^{(1)}(x',\xi')\gamma_0)
     \end{pmatrix}
     :N\lra 
     \begin{matrix}
      \pi^{(0)}(x',\xi') (\cz^L)\\ \oplus\\ (1-\pi^{(0)}(x',\xi'))(\cz^L)  
   \end{matrix},$$
i.e., the bijectivity of 
\begin{align}\label{eq:map}
 \begin{pmatrix}
      \pi^{(0)}(x',\xi') \\(1-\pi^{(0)}(x',\xi'))(-\sigma+b^{(1)}(x',\xi'))
     \end{pmatrix}
     :\cz^L\lra 
     \begin{matrix}
      \pi^{(0)}(x',\xi') (\cz^L)\\ \oplus\\ (1-\pi^{(0)}(x',\xi'))(\cz^L)  
   \end{matrix}.
\end{align}
Since both domain and co-domain have dimension $L$ it suffices to verify the injectivity.   
To this end we assume that $\pi^{(0)}$ commutes with ${a'}^{(2)}$ (hence with $\sigma)$, cf. \cite[Assumption 2.4]{Grub03}, and that 
 $$-\sigma+(1-\pi^{(0)})b^{(1)}(1-\pi^{(0)}):\cz^L\lra\cz^L$$
is an isomorphism whenever $\xi'\not=0$, cf. \cite[(2.3)]{Grub03} and \cite[Assumption 2.7]{Grub03}. 
If $z$ is in the kernel of \eqref{eq:map}, then $z=(1-\pi^{(0)}(x',\xi'))z$, hence 
 $$0=(1-\pi^{(0)}(x',\xi'))(-\sigma+b^{(1)}(x',\xi'))z=[-\sigma+(1-\pi^{(0)})b^{(1)}(1-\pi^{(0)})](x',\xi')z.$$
It follows $z=0$. We thus have verified ellipticity assumption $(\Pi2)$.

\section{The calculus for manifolds}\label{sec:8}

In this section we lay out how the calculus extends to compact manifolds. We discuss coordinate-invariance and shall prove 
that the principal limit symbol has a global analogue. Moreover, we state how the other principal symbols look like globally. 
The technique to define operators on a manifold is to use local coordinates and partitions of unity. We do not enter 
in details concerning global composition and parametrix construction. 

\subsection{Twisting with group-action revisited}\label{sec:8.1}

Both characterization and definition of generalized singular Green symbols, cf. Theorem \ref{thm:conjugation02} 
and Definition \ref{def:conjugation01}, 
made use of the symbol $\kappa(\xi',\mu)=\kappa_{[\xi',\mu]}$ associated with the dilation group-action. 
We shall show now that we may substitute the smoothed norm-function $[\xi',\mu]$ by more general symbols of 
first order. In fact, let $a(x',\xi',\mu)\in S^1(\rz^{n-1}\times\rpbar)$ be homogeneous of degree $1$ in the large. 
Let us write $($cf. \eqref{eq:kappa-matrix}$)$ 
 $$\wt\kappa(x',\xi',\mu):=\kappa_{a(x',\xi',\mu)},\qquad 
     \wt\bfkappa=\begin{pmatrix}\wt\kappa&0\\0&1\end{pmatrix}.$$

\begin{proposition}\label{prop:alternative}
Let $\frakB^{d;r}_G$ represent  one choice of $B^{d;r}_G$ or $\wtbfB^{d,\nu;r}_G$ 
and let $\frakS^d$ represent the corresponding choice $S^d$ or $\wtbfS^{d,\nu}$.
Assume that 
 $$[\xi',\mu]\lesssim{a(x',\xi',\mu)}\lesssim [\xi',\mu]$$
uniformly in $(x',\xi',\mu)$. Then 
 $$\bfg\in \frakB^{d;r}_G(\rz^{n-1}\times\rpbar;\frakg) \iff \wt\bfkappa^{-1}\bfg\wt\bfkappa\in 
    \frakS^d(\rz^{n-1}\times\rpbar;\frakg).$$
\end{proposition}
\begin{proof}
It is sufficient to consider the case $r=0$. For simplicity of presentation we shall show the result for 
Poisson symbols. The case of trace symbols then follows by duality, the case of singular Green symbols works in the 
same way and only involves more lengthy notation due to the $(x_n,y_n)$-dependence of the symbol-kernels. 

Let $k\in \frakB^{d;r}_G(\rz^{n-1}\times\rpbar;(0,M_0),(L,0))$ be a Poisson symbol with symbol-kernel 
$k(x',\xi',\mu;x_n)=[\xi',\mu]^{1/2}k'(x',\xi',\mu;[\xi',\mu]x_n)$ where 
 $$k'(x',\xi',\mu;x_n)\in \frakS^d(\rz^{n-1}\times\rpbar;\scrS(\rz_+,\cz^{L\times M_0}).$$ 
Since $x'$ enters here simply as a $\scrC_b^\infty$-parameter it is no loss of generality to assume that 
$k'$ is constant in $x'$. Moreover, we may assume $L=M_0=1$. Obviously, 
  $$k(x',\xi',\mu;x_n)=a(x',\xi',\mu)^{1/2}k''(x',\xi',\mu;a(x',\xi',\mu)x_n),$$
where 
\begin{align*}
 k''(x',\xi',\mu;x_n)&=p(x',\xi',\mu)^{1/2}k'(\xi',\mu;p(x',\xi',\mu)x_n),\\ 
     p(x',\xi',\mu)&:=\frac{[\xi',\mu]}{a(x',\xi',\mu)}.
\end{align*}
Now the claim is equivalent to the statement 
\begin{align*}
 k'\in \frakS^d(\rz^{n-1}\times\rpbar;\scrS(\rz_+))\iff k''\in \frakS^d(\rz^{n-1}\times\rpbar;\scrS(\rz_+)).
\end{align*}
Since both $p$ and $1/p$ belong to $S^0(\rz^{n-1}\times\rpbar)$ 
We only show direction ``$\Rightarrow$'', since the other direction works in the same way replacing $p$ with $1/p$. 
As explained below Proposition \ref{prop:osint02}, by a tensor-product argument we may 
assume that $k'$ has the form
 $$k'(\xi',\mu;x_n)=\wt k'(\xi',\mu)\phi(x_n)$$
with $\wt k'\in \frakS^d(\rz^{n-1}\times\rpbar)$ and $\phi\in\scrS(\rz_+)$. Then 
 $$k''(x',\xi',\mu;x_n)=\wt k'(\xi',\mu)\phi_p(x',\xi',\mu),\quad \phi_p(x',\xi',\mu)=p(x',\xi',\mu)\phi(p(x',\xi',\mu)x_n).$$
Since $p$ is homogeneous of degree $0$ in the large and $1\lesssim p\lesssim 1$,  
 $$\phi_p\in S^0(\rz^{n-1}\times\rpbar;\scrS(\rz_+))\subset \wtbfS^{0,0}(\rz^{n-1}\times\rpbar;\scrS(\rz_+)).$$
This implies at once the desired property for $k''$.
\end{proof}

\subsection{Invariance under coordinate changes}\label{sec:8.2}

We will study the invariance of the calculus on the half-space. In the present work we shall limit 
outselves to diffeomorphisms of the half-space that leave the ``normal-direction untouched'', 
i.e., have the form  
 $$\wt\chi(x',x_n)=(\chi(x'),x_n),$$
where $\chi$ is a diffeomorphism of $\rz^{n-1}$. Since our calculus concerns operators with globally 
bounded coefficients we need to impose some control at infinity of the diffeomorphism $\chi$, namely 
that $\partial_j\chi_k\in\scrC^\infty_b(\rz^{n-1})$, $1\le j,k\le n-1$, and $|\mathrm{det}\,\chi'|$ 
are uniformly bounded from above and below by some positive constants. 

Given a symbol $\bfp\in\bfB^{d,\nu;r}(\rz^{n-1}\times\rpbar;\frakg)$, $\frakg=((L_0,M_0),(L_1,M_1))$, 
the pull-back of $\op(\bfp)(\mu)$ is defined by 
 $$\chi^*\op(\bfp)(\mu) u=[\op(\bfp)(\mu)(u\circ\chi^{-1})]\circ\chi.,\qquad 
     u\in \scrS(\rz^{n-1},H^{s}(\rz_+,\cz^{L_0})\oplus\cz^{M_0}).$$ 
Note that we consider $\op(\bfp)(\mu)$ as a pseudodifferential operator on $\rz^{n-1}$ acting on vector-valued 
Sobolev spaces (having values in Sobolev spaces of the half-axis). For this reason we shall use the notation $\chi^*$ 
rather than $\wt\chi^*$.   

\begin{lemma}\label{lem:coordinate-change}
Let $a\in \wtbfS^{d,\nu}(\rz^{n-1}\times\rpbar)$. Then there exists a unique symbol $\chi^*a$ in the same class 
such that $\chi^*\op(a)(\mu)=\op(\chi^*a)(\mu)$ for all $\mu$. Moreover, 
 $$(\chi^*a)^{(d,\nu)}(x',\xi',\mu)=a^{(d,\nu)}(\chi(x'),\chi'(x')^{-t}\xi',\mu),\qquad 
     (\chi^*a)^\infty_{[d,\nu]}=\chi^*(a^\infty_{[d,\nu]}).$$
\end{lemma}
\begin{proof}
Except the formula for the limit symbol this is \cite[Theorem 7.1]{Seil22-1}. There it is also shown that the class 
$\wt S^{d,\nu}_{1,0}(\rz^{n-1}\times\rpbar)$ is invariant. We have 
 $$a(x',\xi',\mu)\equiv \sum_{j=0}^{N-1}a_j(x',\xi')p_j(\xi',\mu) \mod \wt S^{d,\nu+N}_{1,0}(\rz^{n-1}\times\rpbar)$$
where for brevity we write here $a_j:=a^\infty_{[d,\nu+j]}$ and $p_j(\xi',\mu):=[\xi',\mu]^{d-\nu-j}$. 
Therefore 
 $$\chi^*a \equiv \sum_{j=0}^{N-1}(\chi^*a_j)\#(\chi^*p_j) \mod \wt S^{d,\nu+N}_{1,0}(\rz^{n-1}\times\rpbar).$$
Now $\chi^*p_j\in S^{d-\nu-j}(\rz^{n-1}\times\rpbar)\subset\wtbfS^{d-\nu-j,0}(\rz^{n-1}\times\rpbar)$. 
Inserting the corresponding expansions of $\chi^*p_j$ yields
 $$\chi^*a \equiv \sum_{j=0}^{N-1}\Big(\sum_{k+\ell=j}(\chi^*a_k)\#(\chi^* p_k)_{[d-\nu-k,\ell]}\Big) p_j \mod 
     \wt S^{d,\nu+N}_{1,0}(\rz^{n-1}\times\rpbar).$$
Since $(\chi^*p_0)_{[d-\nu,0]}(x',\xi')=(\chi^*p_0)^{(d-\nu)}(x',0,1)=1$, the leading coefficient is 
$\chi^* a_0= \chi^*(a^\infty_{[d,\nu+j]})$. 
\end{proof}

\begin{theorem}\label{thm:coordinate-change}
Let $\bfp\in\bfB^{d,\nu;r}(\rz^{n-1}\times\rpbar;\frakg)$. Then there exists a unique symbol 
$\chi^*\bfp\in \bfB^{d,\nu;r}(\rz^{n-1}\times\rpbar;\frakg)$ such that 
$\chi^*\op(\bfp)(\mu)=\op(\chi^*\bfp)(\mu)$ for all $\mu$. 
For the associated principal symbols, the following relations are valid$:$ 
\begin{align*}
 \sigma^{(d)}_\psi(\chi^*\bfp)(x,\xi,\mu)&=\sigma^{(d)}_\psi(\bfp)(\chi(x'),x_n,\chi'(x')^{-t}\xi',\xi_n,\mu),\\
 \sigma^{(d,\nu)}_\partial(\chi^*\bfp)(x',\xi',\mu)&=\sigma^{(d,\nu)}_\psi(\bfp)(\chi(x'),\chi'(x')^{-t}\xi',\mu),\\
 \sigma^{[d]}_\infty(\chi^*\bfp)&=\chi^*\sigma^{[d]}_\infty(\bfp),
\end{align*}
where the latter means the symbol of the pull-back of 
$\op(\sigma^{[d]}_\infty(\bfp))$ under $\chi$. 
\end{theorem}
\begin{proof}
The result is known to be true if $\bfp$ is strongly parameter-dependent, i.e., with $\bfp$ also 
$\chi^*\bfp$ belongs to $B^{d;r}(\rz^{n-1}\times\rpbar;\frakg)$ and, in this case, the homogeneous principal symbol 
and principal boundary symbol behave as indicated. Since the principal limit symbol, in this case, coincides with the principal 
boundary symbol evaluated in $(\xi',\mu)=(0,1)$, the claim for the limit symbol is trivially true. Hence it remains 
to show the claim for $\bfp\in \wtbfB^{d,\nu;r}_G(\rz^{n-1}\times\rpbar;\frakg)$.

Since the diffeomorphism does not effect $x_n$ it leaves the operator $\partial_+$ invariant. 
For this reason it suffices to consider generalized singular Green symbols of type $r=0$. 
For simplicity of presentation we shall prove the theorem for Poisson symbols only; trace and 
singular Green symbols are treated in an analogous way. 
So let $k\in \wtbfB^{d,\nu;0}_G(\rz^{n-1}\times\rpbar;(0,M_0),(L,0))$ 
and $k'=\kappa^{-1}\#k\in\wtbfS^{d,\nu;0}_G(\rz^{n-1}\times\rpbar;(0,M_0),(L,0))$. 
As already observed in the proof of Proposition \ref{prop:alternative}, it is no restriction to assume  
$\frakg=((0,1),(1,0))$ and that $k'$ is of product form, i.e., 
  $$k'(x',\xi',\mu;x_n)=a(x',\xi',\mu)\phi(x_n)$$
with $a\in \wtbfS^{d,\nu}(\rz^{n-1}\times\rpbar)$ and $\phi\in\scrS(\rz_+)$ 
$($we identify symbols with symbol-kernels$)$. Then 
 $${k'}_{[d,\nu]}^\infty(x',\xi';x_n)=a_{[d,\nu]}^\infty(x',\xi')\phi(x_n).$$
Moreover, 
$\op(k')(\mu)=\op(\varphi)\op(a)(\mu)$, where $\op(a)(\mu)$ is a pseudodifferential operator on 
$\rz^{n-1}$ while $[\op(\varphi)u](x)=u(x')\phi(x_n)$ is a Poisson operator. Therefore 
 $$k=\kappa\#k'=k_\phi\#a,\qquad k_\phi(x',\xi',\mu;x_n)=[\xi',\mu]^{1/2}\phi([\xi',\mu]x_n),$$
hence $\chi^*k=(\chi^*k_\phi)\#(\chi^*a)$. 
Since $k_\phi\in B^{0;0}_G(\rz^{n-1}\times\rpbar;\frakg)$, also $\chi^*k_\phi$ belongs to the same class; 
$\chi^*a\in \wtbfS^{d,\nu}(\rz^{n-1}\times\rpbar)$ by Lemma \ref{lem:coordinate-change}. This yields 
$\chi^*k\in \wtbfB^{d,\nu;0}_G(\rz^{n-1}\times\rpbar;\frakg)$. Moreover, 
\begin{align*}
  (\chi^*k_\phi)^\infty_{[0,0]}(x',\xi';x_n)
  &=(\chi^*k_\phi)^{(0)}(x',\xi',\mu;x_n)\big|_{(\xi',\mu)=(0,1)}\\
  &=k_\phi^{(0)}(\chi(x'),0,1;x_n)=|\xi',\mu|^{1/2}\phi(|\xi',\mu|x_n)\big|_{(\xi',\mu)=(0,1)}\\
  &=\phi(x_n)
\end{align*}
and therefore, again using Lemma \ref{lem:coordinate-change}, 
\begin{align*}
 (\chi^*k)_{[d,\nu]}^\infty(x',\xi';x_n)
 &=((\chi^*k_\phi)^\infty_{[0,0]}\#(\chi^*a)^\infty_{[d,\nu]})(x',\xi';x_n)\\
 &=(\chi^*a^\infty_{[d,\nu]})(x',\xi')\phi(x_n)=(\chi^*{k'}^\infty_{[d,\nu]})(x',\xi';x_n)\\
 &=(\chi^*{k}^\infty_{[d,\nu]})(x',\xi';x_n). 
\end{align*}
This gives the transformation rule for the principal limit symbol. The principal boundary symbol is treated similarly, 
again starting out from $\chi^*k=(\chi^*k_\phi)\#(\chi^*a)$. 
\end{proof}

\subsection{Opeartors on compact manifolds}\label{sec:8.3}

Let $M$ be a smooth compact manifold with boundary $\partial M$. We fix a collar neighborhood 
$\calU\cong \partial M\times[0,1)$ with a splitting of coordinates $(x',x_n)$. Moreover, let $E_j$ and $F_j$, $j=0,1$, 
be hermitian vector bundles on $M$ and $\partial M$, respectively, such that 
$E_j|_\calU= \pi^*E'_j$, where $E'=E|_{\partial M}$ and $\pi:\calU\to\partial M$ is the canonical projection. 
We want to define the classes 
 $$\bfB^{d,\nu;r}(M\times\rpbar;\frakg), \qquad \frakg=((E_0,F_0),(E_1,F_1)),$$
with $d\in\gz$ and $r,\nu\in\nz_0$. Elements will be operator-families 
 $$P(\mu)=
     \begin{pmatrix}
      A_+(\mu)+G(\mu)& K(\mu)\\ T(\mu)&Q(\mu)
     \end{pmatrix}:
     \begin{matrix}
      H^{s}(M,E_0)\\ \oplus\\ H^s(\partial M,F_0)
     \end{matrix}
     \lra
     \begin{matrix}
      H^{s-d}(M,E_1)\\ \oplus\\ H^{s-d}(\partial M,F_1)
     \end{matrix}
 $$
whenever $s>r-\frac{1}{2}$. The pseudodifferential part is of the form 
 $$A_+(\mu)=r^+A(\mu)e^+,\qquad A(\mu)\in L^d(\wt M;\wt E_0,\wt E_1),$$
where $A(\mu)$ is a strongly parameter-dependent pseudodifferential operator defined on an ambient manifold 
$\wt M$ in which $M$ is embedded and $\wt E_j$ are extensions of the vector-bundle $E_j$ to $\wt M$. 

It remains to introduce the classes of generalized singular Green operators. 

\begin{definition}
We define
 $$B^{-\infty;0}(M\times\rpbar;\frakg)=\mathop{\mbox{\Large$\cap$}}_{s,s'\in\rz}
    \scrS\left(\rpbar;\scrL\left(
     \begin{matrix}
      H^s_0(M,E_0)\\ \oplus\\ H^s(\partial M,F_0)
     \end{matrix},
     \begin{matrix}
      H^{s'}(M,E_1)\\ \oplus\\ H^{s'}(\partial M,F_1)
     \end{matrix}\right)\right)$$
and $B^{-\infty;r}(M;\frakg)$ the space of all operators of the form 
$G=\sum\limits_{j=0}^{r-1}G_j\boldsymbol{\partial}_+^j$, 
with $G_j\in B^{-\infty;0}(M;\frakg)$ and $\boldsymbol{\partial}_+$ as in \eqref{eq:partial-plus}, 
where $\partial_+$ represents a fixed choice of a first order differential operator which 
coincides with $\partial_{x_n}$ near the boundary. 
\end{definition}

Modulo this class of globally defined regularizing operators, all other contributions to the class of generalized singular 
Green operators are defined locally near the boundary. 

\begin{lemma}\label{lem:cut-off02}
Let $\bfg\in \wtbfB^{d,\nu;r}_G(\rz^{n-1}\times\rpbar;\frakg)$ and $\omega\in\scrC^\infty(\rpbar)$ be a 
compactly supported function with $\omega=1$ near the origin . Then 
 $$\bfg-\begin{pmatrix}\omega&0\\0&0\end{pmatrix}\bfg,\;\bfg-\bfg\begin{pmatrix}\omega&0\\0&0\end{pmatrix} 
    \in B^{-\infty;r}_G(\rz^{n-1}\times\rpbar;\frakg).$$
\end{lemma}    
\begin{proof}
The claim follows from Lemma \ref{lem:cut-off01} by writing 
$(1-\omega)(x_n)=a_j(x_n)x_n^j$ for arbitrary $j\in\nz$ where 
$a_j(x_n)=(1-\omega)(x_n)x_n^{-j}\in\scrC^\infty_b(\rpbar)$. 
\end{proof}


In the following definition $\omega$ is a cut-off function as in Lemma \ref{lem:cut-off02}, considered as a function 
defined on $M$ with support in the collar-neighborhood of $\partial M$. 
The definition does not depend on the choice of $\omega$. 

\begin{definition}\label{def:Green-manifold}
Let $d,\nu\in\rz$ and let $\frakB^{d;r}_G$ be one choice of $B^{d;r}_G$ or  $\wtbfB^{d,\nu;r}_G$. 
The space $\frakB^{d;r}_G(M\times\rpbar;\frakg)$, $\frakg=((E_0,F_0),(E_1,F_1))$, consist of all  
operator-families 
   $$P(\mu):
     \begin{matrix}
      \scrC^\infty(M,E_0)\\ \oplus\\ \scrC^\infty(\partial M,F_0)
     \end{matrix}\to
     \begin{matrix}
      \scrC^\infty(M,E_1)\\ \oplus\\ \scrC^\infty(\partial M,F_1)
     \end{matrix}$$
satisfying$:$
\begin{itemize}
 \item[$(1)$] $\begin{pmatrix}1-\omega&0\\0&0\end{pmatrix}P(\mu)$ and 
  $P(\mu)\begin{pmatrix}1-\omega&0\\0&0\end{pmatrix}$ belong to $B^{-\infty;r}(M\times\rpbar;\frakg)$. 
 \item[$(2)$] For every chart $U$ of $\partial M$ such that $E_j'|_U=U\times\cz^{L_j}$ and 
  $F_j'|_U=U\times\cz^{M_j}$ and for every $\phi,\psi\in\scrC^\infty_0(U)$ the operator 
  $\begin{pmatrix}\phi\omega&0\\0&\phi\end{pmatrix}P(\mu)\begin{pmatrix}\psi\omega&0\\0&\psi\end{pmatrix}$
  transported to $\rz^n_+$ via the local trivialization is a generalized singular Green operator $\op(\bfp)(\mu)$ with 
  symbol   $\bfp\in\frakB^{d;r}_G(\rz^{n-1}\times\rpbar;((L_0,M_0),(L_1,M_1)))$. 
\end{itemize}
\end{definition}

\begin{definition}\label{def:Green-manifold02}
Let $d\in\gz$ and $\nu\in\nz_0$. Then we set 
 $$B^{d;r}(M\times\rpbar;\frakg)
    =\left\{\begin{pmatrix}A_+(\mu)&0\\0&0\end{pmatrix} \mid A(\mu)\in L^d(\wt M;\wt E_0,\wt E_1)
    \right\}
    +B^{d;r}_G(M\times\rpbar;\frakg)$$
and 
\begin{align*}
  \bfB^{d,\nu;r}(M\times\rpbar;\frakg)
  =B^{d;r}(M\times\rpbar;\frakg)+ \wtbfB^{d,\nu;r}_G(M\times\rpbar;\frakg).
\end{align*}
\end{definition}

\subsubsection{Principal symbols for the global calculus}\label{sec:8.3.1}

The relevant principal symbols for the calculus on the half-space $\rz^n_+$ lead to globally defined analogues 
in the manifold case. Given $P(\mu)\in \bfB^{d,\nu;r}(M\times\rpbar;\frakg)$, $\frakg=((E_0,F_0),(E_1,F_1))$, 
the global homogeneous principal symbol is a bundle-morphism 
\begin{align*}
 \sigma^d_\psi(P):\pi^*E_0\lra\pi^* E_1, \qquad \pi:(T^*M\times\rpbar)\setminus 0\lra M,
\end{align*}
where $\pi$ is the canonical projection, i.e., $\pi(v_x,\mu)=x$ for $v_x\in T^*_xM$; equivalently, we may 
consider $\pi:\{(v_x,\mu)\mid \|v_x\|^2+\mu^2=1\}\to M$, with norm induced by some  
fixed Riemannian metric.   
The global principal boundary symbol is a bundle-morphism 
\begin{align*}
 \sigma^{(d,\nu)}_\partial(P):
 \pi_\partial^*\begin{pmatrix}H^s(\rz_+)\otimes E_0'\\ \oplus\\ F_0\end{pmatrix}
 \lra
 \pi_\partial^*\begin{pmatrix}H^{s-d}(\rz_+)\otimes E_1'\\ \oplus\\ F_1\end{pmatrix},\qquad s>r-\frac12,
\end{align*}
where $\pi_\partial:(T^*\partial M\setminus 0)\times\rpbar\to \partial M$ is the canonical projection or, 
equivalently, $\pi_\partial$ considered with domain $\{(v'_{x'},\mu)\mid \|v'_{x'}\|^2+\mu^2=1,\;v'_{x'}\not=0\}$. 
The principal angular symbol is a bundle-morphism
\begin{align*}
 \wh\sigma^{\spk{d}}(P):
 \wh\pi_\partial^*\begin{pmatrix}H^s(\rz_+)\otimes E_0'\\ \oplus\\ F_0\end{pmatrix}
 \lra
 \wh\pi_\partial^*\begin{pmatrix}H^{s-d}(\rz_+)\otimes E_1'\\ \oplus\\ F_1\end{pmatrix},\qquad s>r-\frac12,
\end{align*}
where $\wh\pi_\partial:T^*\partial M\setminus 0\to \partial M$ is the canonical projection which, equivalently, can 
also be restricted to the unit co-sphere bundle of the boundary. 

In order to define the global analogue of the principal limit symbol, let $\phi_j\in\scrC^\infty_0(U_j)$, $j=1,\ldots,N$, 
be a partition of unity such that every union $U_j\cup U_k$ is contained in a chart domain $U_{jk}$ over which all 
involved bundles are trivial. Let $\bfp_{jk}\in\bfB^{d,\nu;r}(\rz^{n-1}\times\rpbar;((L_0,M_0),(L_1,M_1)))$ be 
the local symbol associated with  
$\begin{pmatrix}\phi_j\omega&0\\0&\phi_j\end{pmatrix}P(\mu)
\begin{pmatrix}\phi_k\omega&0\\0&\phi_k\end{pmatrix}$. 
The \emph{principal limit operator} of $P(\mu)$ is defined as 
\begin{align*}
 \sigma_\infty^{[d]}(P)=\sum_{j,k=1}^N \sigma_\infty^{[d]}(P_{jk})
\end{align*}
where $\sigma_\infty^{[d]}(P_{jk})$ transported in $\rz^n_+$ coincides with $\op(\bfp^\infty_{jk,[d,0]})$. 
By construction, 
 $$\sigma_\infty^{[d]}(P):
     \begin{matrix}
      H^s(\partial M, H^{s'}(\rz_+)\otimes E_0')\\ \oplus\\ H^s(\partial M,F_0)
     \end{matrix}\lra
     \begin{matrix}
      H^{s-\nu}(\partial M,H^{s'-d}(\rz_+)\otimes E_1')\\ \oplus\\ H^{s-\nu}(\partial M,F_1)
     \end{matrix}$$
for every $s'>r-\frac{1}{2}$. 
\bibliographystyle{amsalpha}

\end{document}